\documentclass{SCAE}

\def\ls{\lesssim}
\def\gs{\gtrsim}
\def\fz{\infty}

\def\r{\right}
\def\lf{\left}

\def\supp{{\mathop\mathrm{\,supp\,}}}

\def\aa{{\mathbb A}}
\def\rr{{\mathbb R}}
\def\rn{{{\rr}^n}}
\def\zz{{\mathbb Z}}
\def\nn{{\mathbb N}}
\def\ccc{{\mathbb C}}

\newcommand{\wz}{\widetilde}
\newcommand{\oz}{\overline}

\newcommand{\cb}{{\mathcal B}}

\newcommand{\cd}{{\mathcal D}}

\newcommand{\cg}{{\mathcal G}}

\newcommand{\cp}{{\mathcal P}}

\newcommand{\cs}{{\mathcal S}}
\newcommand{\cy}{{\mathcal Y}}
\def\az{\alpha}
\def\lz{\lambda}
\def\blz{\Lambda}

\def\bfai{\Phi}
\def\dz {\delta}
\def\ez {\eta}

\def\bz{\beta}

\def\pz{\psi}
\def\xz{\xi}
\def\fai{\varphi}
\def\gz{{\gamma}}

\def\tz{\theta}
\def\sz{\sigma}
\def\zez{\zeta}
\def\wz{\widetilde}

\def\ls{\lesssim}
\def\gs{\gtrsim}

\def\boz{\Omega}
\def\oz{\omega}

\def\gfz{\genfrac{}{}{0pt}{}}

\def\noz{\nonumber}
\def\fin{{\mathop\mathrm{fin}}}

\def\esup{\mathop\mathrm{\,esssup\,}}

\def\bmo{{{\mathop\mathrm {bmo}}}}
\def\bbmo{{{\mathop\mathrm {BMO}}}}

\def\hs{\hspace{0.3cm}}

\def\dint{\displaystyle\int}
\def\dfrac{\displaystyle\frac}
\def\dsup{\displaystyle\sup}

\def\gfz{\genfrac{}{}{0pt}{}}

\numberwithin{equation}{section}
\def\supp{{\mathop\mathrm{\,supp\,}}}
\def\diam{{\mathop\mathrm{\,diam\,}}}
\def\dist{{\mathop\mathrm{\,dist\,}}}
\def\loc{{\mathop\mathrm{loc\,}}}
\def\lfz{\lfloor}
\def\rfz{\rfloor}
\numberwithin{equation}{section}

\begin{document}

\arraycolsep=1pt

\Year{2012} %
\Month{January}
\Vol{55} %
\No{1} %
\BeginPage{1} %
\EndPage{XX} %
\AuthorMark{Yang D C {\it et al.}}
\ReceivedDay{January 11, 2012}
\AcceptedDay{February 2, 2012}
\PublishedOnlineDay{; published online February 12, 2012}
\DOI{10.1007/s11425-000-0000-0} 

\title{Local Hardy spaces of Musielak-Orlicz type
and their applications}


\author{YANG DaChun}{Corresponding author}
\author{YANG SiBei}{}

\address{School of Mathematical Sciences, Beijing Normal University,
Laboratory of Mathematics}
\address{and Complex Systems, Ministry of
Education, Beijing {\rm 100875}, People's Republic of China}

\Emails{ dcyang@bnu.edu.cn,
yangsibei@mail.bnu.edu.cn}

\maketitle
\vspace{-2mm}

{\begin{center}

\parbox{14.5cm}{\begin{abstract}Let
$\varphi:\,\mathbb{R}^n\times[0,\fz)\rightarrow[0,\fz)$ be a
function such that $\varphi(x,\cdot)$ is an Orlicz function and
$\varphi(\cdot,t)\in
{\mathbb A}^{\mathop\mathrm{loc}}_{\infty}(\mathbb{R}^n)$ (the class of
local weights introduced by V. S. Rychkov). In this paper, the
authors introduce a local Musielak-Orlicz Hardy space
$h_{\varphi}(\mathbb{R}^n)$ by the local grand maximal function, and
a local $\mathop\mathrm{BMO}$-type space
$\mathop\mathrm{bmo}_{\varphi}(\mathbb{R}^n)$ which is further
proved to be the dual space of $h_{\varphi}(\mathbb{R}^n)$. As an
application, the authors prove that the class of pointwise
multipliers for the local $\mathop\mathrm{BMO}$-type space
$\mathop\mathrm{bmo}^{\phi}(\mathbb{R}^n)$, characterized by E.
Nakai and K. Yabuta, is just the dual of
$L^1(\rn)+h_{\Phi_0}(\mathbb{R}^n)$, where $\phi$ is an increasing
function on $(0,\infty)$ satisfying some additional growth
conditions and $\Phi_0$ a Musielak-Orlicz function induced by
$\phi$. Characterizations of $h_{\varphi}(\mathbb{R}^n)$,
including the atoms, the local vertical and the local
nontangential maximal functions, are presented. Using the atomic
characterization, the authors prove the existence of finite atomic
decompositions achieving the norm in some dense subspaces of
$h_{\varphi}(\mathbb{R}^n)$, from which, the authors further
deduce some criterions for the boundedness on
$h_{\varphi}(\mathbb{R}^n)$ of some sublinear operators. Finally,
the authors show that the local Riesz transforms and some
pseudo-differential operators are bounded on
$h_{\varphi}(\mathbb{R}^n)$.
\end{abstract}}
\end{center}}


\begin{center}
\parbox{14.5cm}{}
\end{center}


\begin{center}
\parbox{14.5cm}{\MSC{42B35, 46E30, 42B30,
42B25, 42B20, 35S05, 47G30,47B06}}
\end{center}

\renewcommand{\baselinestretch}{1.2}
\begin{center} \renewcommand{\arraystretch}{1.5}
{\begin{tabular}{lp{0.8\textwidth}} \hline \scriptsize {\bf
Citation:}\!\!\!\!&\scriptsize Yang D C, Yang S B.
Local Hardy spaces of Musielak-Orlicz type
and their applications. Sci China Math,
2012, 55, doi: 10.1007/s11425-000-0000-0\vspace{1mm}
\\
\hline
\end{tabular}}\end{center}

\baselineskip 11pt\parindent=10.8pt  \wuhao

\section{Introduction\label{s1}}

\hskip\parindent It is well known that the theory of the classical
local Hardy spaces $h^p (\rn)$ with $p\in(0,1]$, originally
introduced by Goldberg \cite{go79}, plays an important role in
partial differential equations and harmonic analysis (see, for
example, \cite{bu81,go79,r01,tay91,t83,t92} and their references).
In particular, pseudo-differential operators are bounded on $h^p
(\rn)$ with $p\in(0,1]$, but they are not bounded on Hardy spaces
$H^p (\rn)$ with $p\in(0,1]$ (see \cite{go79} and also \cite{t83,
t92}). Moreover, it was proved by Goldberg \cite{go79} that
$h^p(\rn)$ with $p\in (0,1]$ is closed under composition with
diffeomorphisms and under multiplication by smooth functions with
compact support; while $H^p(\rn)$ with $p\in(0,1]$ is not.

It is also well known that the classical $\bbmo$ space (the {\it
space of functions with bounded mean oscillation}) originally
introduced by John and Nirenberg \cite{jn1} and the classical
Morrey space originally introduced by Morrey \cite{m67} play an
important role in the study of partial differential equations and
harmonic analysis (see, for example, \cite{ddy05,dxy07,fs72,n07}).
In particular, Fefferman and Stein \cite{fs72} proved that
$\bbmo(\rn)$ is the dual space of $H^1 (\rn)$. Also, Goldberg
\cite{go79} introduced the local variant, $\bmo(\rn)$, of
$\bbmo(\rn)$ and proved that $\bmo(\rn)$ is the dual space of $h^1
(\rn)$.

In \cite{bu81}, Bui studied the weighted version,
$h^p_{\oz}(\rn)$, of $h^p(\rn)$  with $\oz\in A_{\fz}(\rn)$, here and in what follows, $A_{\fz}(\rn)$ denotes the \emph{class of
Muckenhoupt's weights} (see, for example, \cite{g79,gr1,gra1} for
their definitions and properties). Rychkov \cite{r01} introduced
and studied a \emph{class of local weights}, denoted by
$A^{\loc}_{\fz}(\rn)$ (see also Definition \ref{d2.1} below), and
the weighted Besov-Lipschitz spaces and the Triebel-Lizorkin
spaces with weights belonging to $A^{\loc}_{\fz}(\rn)$, which
contains $A_{\fz}(\rn)$ weights and the exponential weights
introduced by Schott \cite{sc98} as special cases. In particular,
Rychkov \cite{r01} generalized some of the theory of weighted
local Hardy spaces developed by Bui \cite{bu81} to
$A^{\loc}_{\fz}(\rn)$ weights. Very recently, Tang \cite{ta1}
established the weighted atomic characterization of
$h^p_{\oz}(\rn)$ with $\oz\in A^{\loc}_{\fz}(\rn)$ via the local
grand maximal function. Tang \cite{ta1} also established some
criterions for the boundedness of $\cb_{\bz}$-sublinear operators
on $h^p_{\oz}(\rn)$ (see Section \ref{s6} of this paper for the notion of
$\cb_{\bz}$-sublinear operators, which was first introduced in
\cite{yz08}). As applications, Tang \cite{ta1, ta2} proved that
some strongly singular integrals, pseudo-differential operators
and their commutators are bounded on $h^p_{\oz}(\rn)$. It is worth
pointing out that in recent years, many papers focused on
criterions for the boundedness of (sub)linear operators on various
Hardy spaces with different underlying spaces (see, for example,
\cite{b05,blyz08,gly,gly10,lbyz10,msv08,msv09,rv,ta1,yz08,yz09}),
and on the entropy and approximation numbers of embeddings of function spaces
with Muckenhoupt weight (see, for example,
\cite{dl1,dl2,dl3,dl4}). Moreover, let $L$ be a linear operator on
$L^2(\rn)$, which generates an analytic semigroup
$\{e^{-tL}\}_{t\ge0}$ with kernels satisfying an upper bound of
Poisson type; the local Hardy space $h^1_L(\rn)$ associated with
$L$ and its dual space were studied in \cite{jyz09a}.

On the other hand,  as the generalization of $L^p (\rn)$, the
Orlicz space was introduced by Birnbaum-Orlicz in \cite{bo31} and
Orlicz in \cite{o32}. Since then, the theory of the Orlicz spaces
themselves has been well developed and these spaces have been
widely used in probability, statistics, potential theory, partial
differential equations, as well as harmonic analysis and some
other fields of analysis (see, for example,
\cite{aikm00,io02,mw08,rr91,rr00}). Moreover, Orlicz-Hardy spaces
are also suitable substitutes of the Orlicz spaces in dealing with
many problems of analysis (see, for example,
\cite{ja80,jy10,s79,vi87}). Recall that Orlicz-Hardy spaces and
their dual spaces were studied by Str\"omberg \cite{s79} and
Janson \cite{ja80} on $\rn$ and Viviani \cite{vi87} on spaces of
homogeneous type in the sense of Coifman and Weiss \cite{cw71}.
The weighted local Orlicz-Hardy space $h^{\Phi}_{\oz}(\rn)$, with
$\oz\in A^{\loc}_{\fz}(\rn)$ and $\Phi$ being an Orlicz
function of critical
lower type index $p^-_\Phi\in(0,1]$ (see Section \ref{s2} below for
the notions of Orlicz functions and $p^-_\Phi$), via the
local grand maximal function
was introduced in \cite{yys}. Let
$\rho(t):=t^{-1}/\Phi^{-1}(t^{-1})$ for all $t\in(0,\infty)$,
where $\Phi^{-1}$ denotes the \emph{inverse function} of $\Phi$. The
$\mathop\mathrm{BMO}$-type space
$\mathop\mathrm{bmo}_{\rho,\,\omega}(\mathbb{R}^n)$ was also
introduced in \cite{yys}, which was proved to be the dual space of
$h^{\Phi}_{\omega}(\mathbb{R}^n)$ therein. Characterizations of
$h^{\Phi}_{\omega}(\mathbb{R}^n)$, including the atoms, the local
vertical and the local nontangential maximal functions, were
presented in \cite{yys}. Some criterions for the boundedness of
$\cb_{\bz}$-sublinear operators on $h^{\Phi}_{\oz}(\rn)$ were also
given in \cite{yys}. As applications, it was showed, in
\cite{yys}, that the local Riesz transforms and some
pseudo-differential operators are bounded on
$h^{\Phi}_{\omega}(\mathbb{R}^n)$. Moreover, Orlicz-Hardy spaces
associated with some differential operators and their dual spaces
were introduced and studied in \cite{lyy11,jy10,jyz09}.

Let $\phi$ be a positive increasing function on $(0,\fz)$,
and $\bbmo^{\phi}(\rn)$ and $\bmo^{\phi}(\rn)$ two classes of
$\mathop\mathrm{BMO}$-type spaces introduced by Nakai and
Yabuta \cite{ny85} (see also Definition \ref{d7.4} below).
Assume further that $\phi(r)/r$ is
\emph{almost decreasing}, namely, there exists a positive constant
$C$ such that $\phi(t)/t\le C\phi(s)/s$ for all $t,\,s\in(0,\fz)$
with $t\ge s$. Nakai and Yabuta \cite{ny85} proved that a function
$g$ on $\rn$ is a pointwise multiplier on $\bmo^{\phi}(\rn)$ if and
only if $g\in \bbmo^{\pz}(\rn)\cap L^{\fz}(\rn)$, where
$$\pz(r):=\phi(r)\lf\{\int^2_{\min(1,r)}\frac{\phi(t)}{t}\,dt\r\}^{-1}$$
for all $r\in(0,\fz)$ (see \cite[Theorem 3]{ny85} or Proposition
\ref{p7.1} below).

Recently, Ky \cite{k} introduced a \emph{new Musielak-Orlicz
Hardy space}, $H^{\fai}(\rn)$, via the grand maximal
function, which generalizes both the Orlicz-Hardy space of
Str\"omberg \cite{s79} and Janson \cite{ja80} and the weighted
Hardy space $H^p_\oz(\rn)$ studied by Garc\'\i a-Cuerva \cite{gr1}
and Str\"omberg and Torchinsky \cite{st}. Here,
$\fai:\,\rn\times[0,\fz)\rightarrow[0,\fz)$ is a function such
that $\fai(x,\cdot)$ is an Orlicz function and $\fai(\cdot,t)$ is
a uniform Muckenhoupt weight. In particular, when $\fai(x,t):=\Phi(t)$ for
all $x\in\rn$ and $t\in[0,\fz)$, $H^{\fai}(\rn)$ is just the
classical Orlicz-Hardy space, where $\Phi$ is an Orlicz function
with upper type 1 and lower type $p\in(0,1]$. In \cite{k}, Ky
established the atomic characterization of $H^{\fai}(\rn)$; as an
application, Ky also gave some criterions for the boundedness of
$\cb_{\bz}$-sublinear operators on $H^{\fai}(\rn)$. Ky \cite{k}
further introduced the $\mathop\mathrm{BMO}$-type space
$\mathop\mathrm{BMO}_{\fai}(\rn)$, which was proved to be the dual
space of $H^{\fai}(\rn)$; as an interesting application, Ky proved
that the class of pointwise multipliers for $\bbmo(\rn)$,
characterized by Nakai and Yabuta \cite{ny85}, is the dual space
of $L^1(\rn)+H^{\log}(\rn)$, where $H^{\log}(\rn)$ is the Musielak-Orlicz
Hardy space related to the Musielak-Orlicz function
$$\fai(x,t):=\frac{t}{\log(e+|x|)+\log(e+t)}$$
for all $x\in\rn$ and $t\in [0,\fz)$. Musielak-Orlicz functions are the
natural generalization of Orlicz functions that may vary in the
spatial variables (see, for example, \cite{d05,dhr09,k,m83}).
Recall that the motivation to study function spaces of
Musielak-Orlicz type comes from the applications to elasticity, fluid
dynamics, image processing, nonlinear partial differential
equations and the calculus of variation (see, for example,
\cite{bg10,bgk,bijz07,d05,dhr09,k,l05}). It is worth noticing that
some special Musielak-Orlicz Hardy spaces appear naturally
in the study of the products of functions in $\bbmo(\rn)$ and
$H^1(\rn)$ (see \cite{bgk,bijz07}), and the endpoint estimates for
the div-curl lemma (see \cite{bfg10,bgk}).

Let $\fai:\,\rn\times[0,\fz)\rightarrow[0,\fz)$ be a \emph{growth
function}, namely, $\fai(x,\cdot)$ is an Orlicz function of upper
type 1 and lower type $p\in(0,1]$, and $\fai(\cdot,t)$ belongs to
$\aa^{\mathop\mathrm{loc}}_{\infty}(\mathbb{R}^n)$ (see Definition
\ref{d2.3} below). Motivated by \cite{k,r01,yys}, in this paper, we
introduce a new local Musielak-Orlicz Hardy space,
$h_{\fai}(\rn)$, via the local grand maximal function and the
local $\mathop\mathrm{BMO}$-type space,
$\mathop\mathrm{bmo}_{\fai}(\rn)$, which is further proved to be
the dual space of $h_{\fai}(\rn)$. As an interesting application,
we show that the class of pointwise multipliers for
$\mathop\mathrm{bmo}^{\phi}(\mathbb{R}^n)$, characterized by Nakai
and Yabuta in \cite{ny85}, is the dual of
$L^1(\rn)+h_{\Phi_0}(\mathbb{R}^n)$, where $\phi$ is an increasing
function on $(0,\infty)$ satisfying some additional growth
conditions (see Theorem \ref{t7.2} and Remark
\ref{r7.1} below), $\mathop\mathrm{bmo}^{\phi}(\mathbb{R}^n)$ denotes the
local $\mathop\mathrm{BMO}$-type space introduced by Nakai and
Yabuta in \cite{ny85}, and $\Phi_0$ is a Musielak-Orlicz function
induced by $\phi$. Characterizations of $h_{\fai}(\rn)$, including
the atoms, the local vertical and the local nontangential maximal
functions, are presented. Using the atomic characterization, we
prove the existence of finite atomic decompositions achieving the
norm in some dense subspaces of $h_{\fai}(\rn)$, from which we
further deduce that for a given admissible triplet
$(\fai,\,q,\,s)$ and a $\beta$-quasi-Banach space
$\mathcal{B}_{\beta}$ with $\beta\in(0,1]$, if $T$ is a
$\mathcal{B}_{\beta}$-sublinear operator, then $T$ uniquely
extends to a bounded $\mathcal{B}_{\beta}$-sublinear operator from
$h_{\fai}(\rn)$ to $\mathcal{B}_{\beta}$ if and only if $T$ maps
all $(\fai,\,q,\,s)$-atoms and $(\fai,\,q)$-single-atoms with
$q<\infty$ (or all continuous $(\fai,\,q,\,s)$-atoms with $q=\fz$)
into uniformly bounded elements of $\mathcal{B}_{\beta}$. Finally,
we show that the local Riesz transforms and some
pseudo-differential operators are bounded on $h_{\fai}(\rn)$. We
point out that the local Musielak-Orlicz Hardy space $h_{\fai}(\rn)$
includes the classical local Hardy space in
\cite{go79}, the weighted local Hardy space with Muckenhoupt
weight in \cite{bu81} and with local Rychkov weight in \cite{ta1}
and the weighted local Orlicz-Hardy space in \cite{yys} as special
cases.

Precisely, this paper is organized as follows. In Section \ref{s2},
we recall some notions concerning Musielak-Orlicz functions,
some examples, and some properties
established in \cite{k} for growth functions.

In Section \ref{s3}, we first introduce the local grand maximal
function, the local Musielak-Orlicz Hardy space,
$h_{\fai,\,N}(\rn)$,  and then the atomic local Musielak-Orlicz
Hardy space, $h^{\fai,\,q,\,s}(\rn)$, for any admissible
triplet $(\fai,\,q,\,s)$ (see Definitions \ref{d3.2} and
\ref{d3.4} below). We point out that when
\begin{equation}\label{1.1}
\fai(x,t):=\oz(x)\Phi(t)\end{equation} for all $x\in\rn$ and
$t\in[0,\fz)$, the local Musielak-Orlicz Hardy space,
$h_{\fai,\,N}(\rn)$, is just the weighted local Hardy space
$h^{\Phi}_{\oz,\,N}(\rn)$ introduced in \cite{yys}. Here, $\oz\in
A^{\mathop\mathrm{loc}}_{\infty}(\mathbb{R}^n)$ and $\Phi$ is an
Orlicz function of upper type 1 and lower type $p\in(0,1]$. Next,
we establish the local vertical and the local nontangential
maximal function characterizations of $h_{\fai,\,N}(\rn)$ via a
local Calder\'on reproducing formula and some useful estimates
established by Rychkov \cite{r01}, which generalizes \cite[Theorem
3.14]{yys} by taking $\fai$ as in \eqref{1.1} (see Theorems
\ref{t3.1} and \ref{t3.2} and Remark \ref{r3.3} below). Finally, we
present some properties of these spaces in Propositions \ref{p3.2}
and \ref{p3.3} below.

Throughout the whole paper, as usual, $\cd(\rn)$ denotes the {\it
set of all $C^\fz(\rn)$ functions on $\rn$ with compact support},
endowed with the inductive limit topology, and $\cd'(\rn)$ its {\it
topological dual space}, endowed with the weak-$\ast$ topology. In
Section \ref{s4}, for any given $f\in\cd'(\rn)$, integer $s\ge\lfz
n[q(\fai)/i(\fai)-1]\rfz$ and $\lz>\inf_{x\in\rn}\cg_{N,\,\wz{R}}
(f)(x)$, where $q(\fai)$, $i(\fai)$ and $\cg_{N,\,\wz{R}}(f)$ are
respectively as in \eqref{2.9}, \eqref{2.3} and \eqref{3.2} below,
$\lfz\az\rfz$ for any $\az\in\rr$ denotes
the \emph{maximal integer not more than $\az$}
and $\wz{R}=2^{3(10+n)}$, following \cite{b03,blyz08,st93,ta1,yys},
via a Whitney decomposition of $\boz_{\lz}$ in \eqref{4.1}, we
obtain the Calder\'on-Zygmund decomposition $f= g+\sum_i b_i$ in
$\cd'(\rn)$ of degree $s$ and height $\lz$ associated with the local
grand maximal function $\cg_{N,\,\wz{R}}(f)$. The main task of
Section \ref{s4} is to recall some subtle estimates for $g$ and
$\{b_i\}_i$, mainly from \cite{ta1}. Precisely, Lemmas \ref{l4.2}
through \ref{l4.5} are
estimates on $\{b_i\}_i$, the bad part of $f$, while Lemmas
\ref{l4.6} and \ref{l4.7} on $g$, the good part of $f$. As an
application of these estimates, we obtain the density of
$L^{q}_{\fai(\cdot,1)}(\rn)\cap h_{\fai,\,N}(\rn)$ in
$h_{\fai,\,N}(\rn)$, where $q\in(q(\fai),\fz)$ (see Corollary
\ref{c4.1} below).

In Section \ref{s5}, we prove that for any admissible triplet
$(\fai,\,q,\,s)$, $h^{\fai,\,\fz,\,s}(\rn)=h_{\fai,\,N}(\rn)$ with
equivalent quasi-norms when positive integer $N\ge N_{\fai}$ (see
\eqref{3.20} below for the definition of $N_{\fai}$), by using the
Calder\'on-Zygmund decomposition and some subtle estimates
presented in Section \ref{s4}, which completely covers
\cite[Theorem\,5.6]{yys} by taking $\fai$ as in \eqref{1.1} (see
Theorem \ref{t5.1} and Remark \ref{r5.2} below). For simplicity,
\emph{in the remainder of this introduction}, we denote by
$h_{\fai}(\rn)$ the \emph{local Musielak-Orlicz Hardy space,
$h_{\fai,\,N}(\rn)$, with positive integer $N\ge N_{\fai}$}.

Assume that $(\fai,\,q,\,s)$ is an admissible triplet and $\fai$
satisfies the uniformly locally $q$-dominated convergence
condition when $q\in(q(\fai),\fz)$ and the uniformly locally
$q_0$-dominated convergence condition with some
$q_0\in(q(\fai),\fz)$ when $q=\fz$ (see Definition \ref{d6.2}
below). Let $h^{\fai,\,q,\,s}_{\fin}(\rn)$ be the {\it space of
all finite linear combinations of $(\fai,\,q,\,s)$-atoms or
$(\fai,\,q)$-single-atoms} (see Definition \ref{d6.1} below), and
$h^{\fai,\,\fz,\,s}_{\fin,\,c}(\rn)$ the {\it space of all $f\in
h^{\fai,\,\fz,\,s}_{\fin}(\rn)$ with compact support}. In Section
\ref{s6}, we prove that $\|\cdot\|_{h^{\fai,\,q,\,s}_{\fin}(\rn)}$
and $\|\cdot\|_{h_{\fai}(\rn)}$ are equivalent quasi-norms on
$h^{\fai,\,q,\,s}_{\fin}(\rn)\cap L^q_{\fai(\cdot,1)}(\rn)$ when
$q <\fz$ and on $h^{\fai,\,\fz,\,s}_{\fin,\,c}(\rn)\cap C(\rn)$
when $q=\fz$ (see Theorem \ref{t6.1} below). As an application, we
prove that for a given admissible triplet $(\fai,\,q,\,s)$ and a
$\beta$-quasi-Banach space $\mathcal{B}_{\beta}$ with
$\beta\in(0,1]$, if $T$ is a $\cb_{\bz}$-sublinear operator from
$h^{\fai,\,q,\,s}_{\fin}(\rn)$ to $\mathcal{B}_{\beta}$ when
$q\in(q(\fai),\fz)$ or from $h^{\fai,\,\fz,\,s}_{\fin}(\rn)\cap
C(\rn)$ to $\mathcal{B}_{\beta}$ when $q=\fz$, and $\fai$ is of
uniformly upper type $\wz{p}\in(0,\beta]$, then $T$ uniquely
extends to a bounded $\cb_{\bz}$-sublinear operator from
$h_{\fai}(\rn)$ to $\cb_{\bz}$ if and only if $T$ maps all
$(\fai,\,q,\,s)$-atoms and $(\fai,\,q)$-single-atoms with
$q\in(q(\fai),\fz)$ (or all continuous $(\fai,\,q,\,s)$-atoms with
$q=\infty$) into uniformly bounded elements of $\cb_{\bz}$ (see
Theorem \ref{t6.2} below). We remark that this
 extends both the results of Meda-Sj\"ogren-Vallarino
\cite{msv08,msv09} and Yang-Zhou \cite{yz09} to the setting of local
Musielak-Orlicz Hardy spaces. We also point out that Theorem
\ref{t6.1} and Theorem \ref{t6.2} below completely cover
\cite[Theorem\,6.2]{yys} and \cite[Theorem \,6.4]{yys},
respectively, by taking $\fai$ as in \eqref{1.1} (see
Remark \ref{r6.2} below).

In Subsection \ref{s7.1}, we introduce the $\bbmo$-type space
$\bmo_{\fai}(\rn)$ which is proved to be the dual space of
$h^{\fai,\,\fz,\,s}(\rn)$ (see Theorem \ref{t7.1} below). Combining
Theorems \ref{t5.1} with \ref{t7.1}, we see that
$\lf[h_{\fai} (\rn)\r]^{\ast}=\bmo_{\fai}(\rn)$
(see Corollary \ref{c7.1} below). Assume that $\phi\equiv1$ or
$\phi$ is an increasing function on $(0,\fz)$ satisfying the fact that
$\phi(r)/r$ is almost decreasing and $\phi$ is of lower type
$p_0\in (0,1]$ (more generally, $\phi$ satisfies \eqref{7.10}
below). As an application of Corollary \ref{c7.1}, in Subsection
\ref{s7.2}, we show that the class of pointwise multipliers for
$\mathop\mathrm{bmo}^{\phi}(\mathbb{R}^n)$ characterized by Nakai
and Yabuta \cite{ny85} is the dual of
$L^1(\rn)+h_{\Phi_0}(\mathbb{R}^n)$ (see Theorem \ref{t7.2} below).
Recall that when $\phi\equiv1$, $\bmo^{\phi}(\rn)$ is just
$\bmo(\rn)$ introduced by Goldberg in \cite{go79}.

In Subsection \ref{s8.1}, we first prove that the local Riesz
transforms are bounded on $h_{\fai}(\rn)$ with
$\fai\in\aa^{\loc}_p(\rn)$ for $p\in[1,\fz)$, which completely
covers \cite[Theorem\,8.2]{yys} by taking $\fai$ as in
\eqref{1.1}, where $\Phi$ is further assumed to satisfy the fact that
$p_{\bfai}=p_{\bfai}^+$ and $p_{\Phi}^+$ is the attainable
critical upper type index of $\Phi$ (see Section \ref{s2} below
for the definitions of $p_{\Phi}$ and $p_{\bfai}^+$), which is now
proved to be superfluous in Theorem \ref{t8.1} below (see also
Remark \ref{r8.1} below). Finally, let $\fai(\cdot,t)$ be a weight
having better properties than $\aa^\loc_\fz(\rn)$
and $T$ an $S^0_{1,\,0}(\rn)$
pseudo-differential operator. We prove that $T$ is bounded on
$h_{\fai}(\rn)$, which completely covers \cite[Theorem\,8.18]{yys} (see
Theorem \ref{t8.2} and also Remark \ref{r8.2} below).

Finally, we make some conventions on some notation. Throughout the whole
paper, we denote by $C$ a \emph{positive constant} which is
independent of the main parameters, but it may vary from line to
line. We also use $C(\gz,\bz,\cdots)$ to denote a {\it positive
constant depending on the indicated parameters $\gz$, $\bz$,
$\cdots$}. The {\it symbol} $A\ls B$ means that $A\le CB$. If $A\ls
B$ and $B\ls A$, then we write $A\sim B$. The {\it symbol} $\lfz
s\rfz$ for $s\in\rr$ denotes the maximal integer not more than $s$.
For any given normed spaces $\mathcal A$ and $\mathcal B$ with the
corresponding norms $\|\cdot\|_{\mathcal A}$ and
$\|\cdot\|_{\mathcal B}$, the {\it symbol} ${\mathcal
A}\subset{\mathcal B}$ means that for all $f\in \mathcal A$, then
$f\in\mathcal B$ and $\|f\|_{\mathcal B}\ls \|f\|_{\mathcal A}$. For
any measurable subset $E$ of $\rn$, we denote by $E^\complement$ the \emph{set}
$\rn\setminus E$ and by
$\chi_{E}$  its \emph{characteristic function}. We also set $\nn:=\{0,\,1,\,
\cdots\}$ and $\zz_+:=\{1,\,2,\,\cdots\}$. For any
$\tz:=(\tz_{1},\ldots,\tz_{n})\in\nn^{n}$, let
$|\tz|:=\tz_{1}+\cdots+\tz_{n}$ and
$\partial^{\tz}_x:=\frac{\partial^{|\tz|}}{\partial
{x_{1}^{\tz_{1}}}\cdots\partial {x_{n}^{\tz_{n}}}}$. Given a
function $g$ on $\rn$, if $\int_\rn g(x)\,dx\neq 0$, we let $L_g:= -1$;
otherwise, we let $L_g\in\nn$ be the {\it maximal integer} such that
$g$ has vanishing moments up to order $L_g$, namely,
$\int_{\rn}g(x)x^{\az}\,dx=0$ for all multi-indices $\az$ with
$|\az|\le L_g$.

\section{Preliminaries\label{s2}}

\hskip\parindent In Subsection \ref{s2.1}, we recall some notions and
notation concerning Musielak-Orlicz functions considered in
this paper; some specific examples of Musielak-Orlicz
functions satisfying the assumptions
are also given in this subsection. Subsection \ref{s2.2}
is devoted to recalling some properties of Musielak-Orlicz
functions established in \cite{k}.

\subsection{Musielak-Orlicz functions\label{s2.1}}

\hskip\parindent We recall that a function $\Phi:[0,\fz)\to[0,\fz)$
is called an \emph{Orlicz function} if it is nondecreasing,
$\Phi(0)=0$, $\Phi(t)>0$ for $t\in(0,\fz)$ and
$\lim_{t\to\fz}\Phi(t)=\fz$ (see, for example,
\cite{m83,rr91,rr00}). The function $\Phi$ is said to be of \emph{
upper type $p$} (resp. \emph{lower type $p$}) for some $p\in[0,\fz)$, if
there exists a positive constant $C$ such that for all $t\in[1,\fz)$
(resp. $t\in[0,1]$) and $s\in[0,\fz)$,
\begin{equation}\label{2.1}
\Phi(st)\le Ct^p \Phi(s).
\end{equation}
If $\Phi$ is of both upper type $p_1$ and lower type $p_0$, then
$\Phi$ is said to be of \emph{type $(p_0,\,p_1)$}. Let
\begin{eqnarray*}
&&p_{\Phi}^{+}:=\inf\{p\in[0,\fz):\,\text{there exists}\,
\,C\in(0,\fz)\,\,\\
&&\hspace{5 em}\text{such that}\,\eqref{2.1} \,\, \text{holds for
all} \,\,t\in[1,\fz)\,\, \text{and}\,\,s\in[0,\fz)\},
\end{eqnarray*}
and
\begin{eqnarray*}
&&p_{\Phi}^{-}:=\sup\{p\in[0,\fz):\,\text{there exists}\,
\,C\in(0,\fz)\,\,\\
&&\hspace{5 em}\,\text{such that}\,\eqref{2.1} \,\, \text{holds for
all} \,\,t\in[0,1]\,\, \text{and}\,\,s\in[0,\fz)\}.
\end{eqnarray*}

The function $\Phi$ is said to be of {\it strictly lower type $p$}
if for all $t\in[0,1]$ and $s\in[0,\fz)$, $\Phi(st)\le t^p\Phi(s)$,
and define
\begin{equation*}
p_{\Phi}:=\sup\{p\in(0,\fz):\,\Phi(st)\le t^p\Phi(s) \,\text{holds
for all}\ t\in[0,1]\,\text{and}\ s\in[0,\fz)\}.
\end{equation*}

It is easy to see that $p_{\Phi}\le p_{\Phi}^{-}\le p_{\Phi}^{+}$
for all $\Phi$. In what follows, $p_{\bfai},\,p_{\bfai}^{-}$ and
$p_{\bfai}^{+}$ are respectively called the {\it strictly critical
lower type index}, the  {\it critical lower type index} and the {\it
critical upper type index} of $\bfai$. Moreover, it was proved in
\cite[Remark 2.1]{jy10} that $\Phi$ is also of strictly lower type
$p_\Phi$; in other words, $p_\Phi$ is \emph{attainable}.

\begin{remark}\label{r2.1}
We observe that, via the Aoki-Rolewicz theorem in \cite{ao42,ro57},
all results in \cite{yys} are still true if the assumptions on
$p_{\Phi}$ therein are relaxed into the same assumptions on
$p_{\Phi}^-$.
\end{remark}

Given a function $\fai:\,\rn\times[0,\fz)\to[0,\fz)$ such
that for any $x\in\rn$,
$\fai(x,\cdot)$ is an Orlicz function, and we say that $\fai$ is of
\emph{uniformly upper type $p$} (resp. \emph{uniformly lower type
$p$}) for some $p\in[0,\fz)$ if there exists a positive constant $C$
such that for all $x\in\rn$, $t\in[0,\fz)$ and $s\in[1,\fz)$ (resp.
$s\in[0,1]$),
\begin{equation}\label{2.2}
\fai(x,st)\le Cs^p\fai(x,t).
\end{equation}
We say that $\fai$ is of \emph{positive uniformly upper type}
(resp. \emph{uniformly lower type}) if it is of uniformly upper
type (resp. uniformly lower type) $p$ for some $p\in(0,\fz)$, and
let
\begin{equation}\label{2.3}
i(\fai):=\sup\{p\in(0,\fz):\ \fai\ \text{is of uniformly lower
type}\ p\}.
\end{equation}

The following notion of local weights was introduced by Rychkov
\cite{r01} when $q\in(1,\fz]$ and by Tang \cite{ta1} when $p=1$.
Let $Q$ be a cube in $\rn$ and denote its Lebesgue measure by
$|Q|$. Throughout the whole paper, \emph{all cubes are assumed to
be closed with their sides parallel to the coordinate axes}.

\begin{definition}\label{d2.1}
Let $p\in(1,\fz)$. The {\it weight class $A^{\loc}_p (\rn)$} is
defined to be the set of all nonnegative locally integrable
functions $\oz$ on $\rn$ such that
\begin{equation}\label{2.4}
A^{\loc}_p (\oz):=\sup_{|Q|\le1}\frac{1}{|Q|^p}\int_Q \oz(x)\,dx
\lf(\int_Q [\oz(x)]^{-p'/p}\,dx\r)^{p/p'}<\fz,
\end{equation}
where the supremum is taken over all cubes $Q\subset\rn$ with
$|Q|\le1$ and $\frac{1}{p}+\frac{1}{p'}=1$.

When $p=1$, the {\it weight class $A^{\loc}_1 (\rn)$} is defined to
be the set of all nonnegative locally integrable functions $\oz$ on
$\rn$ such that
\begin{equation}\label{2.5}
A^{\loc}_1 (\oz):=\sup_{|Q|\le1}\frac{1}{|Q|}\int_Q \oz(x)\,dx
\lf(\esup_{y\in Q}[\oz(y)]^{-1}\r)<\fz,
\end{equation}
where the supremum is taken over all cubes $Q\subset\rn$ with
$|Q|\le1$.

When $p=\fz$, the {\it weight class $A^{\loc}_{\fz} (\rn)$} is
defined to be the set of all nonnegative locally integrable
functions $\oz$ on $\rn$ such that for any $\az\in(0,1)$,
\begin{equation}\label{2.6}
A^{\loc}_{\fz} (\oz;\,\az):=\sup_{|Q|\le1} \lf[\sup_{F\subset
Q,\,|F|\ge\az|Q|}\frac{\oz(Q)}{\oz(F)}\r]<\fz,
\end{equation}
where the first supremum is taken over all cubes $Q\subset\rn$
with $|Q|\le1$ and the second one over all measurable sets
$F\subset Q$ with the indicated properties. Here and in what
follows, for any subset $E\subset\rn$, let $\oz(E):=\int_E
\oz(x)\,dx$.
\end{definition}

\begin{remark}\label{r2.2}
(i) By H\"older's inequality, we see that $A^{\loc}_{p_1}
(\rn)\subset A^{\loc}_{p_2} (\rn)\subset A^{\loc}_{\fz} (\rn)$, if
$1\le p_1<p_2<\fz$. Conversely, it was proved in \cite[Lemma
1.3]{r01} that if $\oz\in A^{\loc}_{\fz} (\rn)$, then $\oz\in
A^{\loc}_{p} (\rn)$ for some $p\in[1,\fz)$. Thus, we see that
$A^{\loc}_{\fz} (\rn)=\cup_{1\le p<\fz}A^{\loc}_p (\rn)$.

(ii) For any given constant $\wz{C}\in(0,\fz)$, the condition $|Q|\le1$
can be replaced by $|Q|\le\wz{C}$ in \eqref{2.4}, \eqref{2.5} and
\eqref{2.6} (see \cite[Remark\,1.5]{r01}). In this case, $A^{\loc}_p
(\oz)$ with $p\in[1,\fz)$ and $A^{\loc}_{\fz}(\oz,\,\az)$ depend on
$\wz{C}$.

(iii) By the definition of $A^{\loc}_p(\rn)$, we know that if
$\oz\in A^{\loc}_q(\rn)$ for some $q\in[1,\fz)$, then $\oz\in
A^{\loc}_p(\rn)$ for all $p\in (q,\fz]$. It was proved in
\cite[Lemma 2.1(ii)]{ta1} that if $\oz\in A^{\loc}_p(\rn)$ for
some $p\in(1,\fz)$, then $\oz\in A^{\loc}_r(\rn)$ for some
$r\in[1,p)$. Thus, for any given $\oz\in A^{\loc}_{\fz}(\rn)$, we
define the \emph{critical index of $\oz$} by
$$q_{\oz}:=\inf\lf\{p\in[1,\fz):\ \oz\in A^{\loc}_p(\rn)\r\}.$$
Obviously, $q_{\oz}\in[1,\fz)$. When $q_{\oz}\in(1,\fz)$, it is
easy to see that $\oz\not\in A^{\loc}_{q_{\oz}}(\rn)$. Moreover,
motivated by Johnson and Neugebauer \cite[p.\,254,\,Remark]{jn87},
it was proved in \cite[Remark 2.5]{yys} that there exists an
$\oz\not\in A^{\loc}_1 (\rn)$ such that $q_{\oz}=1$.
\end{remark}

Let $\fai:\rn\times[0,\fz)\to[0,\fz)$ satisfy the fact that $x\mapsto\fai(x,t)$
is measurable for all $t\in[0,\fz)$. Following \cite{k},
$\fai(\cdot,t)$ is called \emph{uniformly locally integrable} if for
all compact sets $K$ in $\rn$,
$$\int_{K}\sup_{t\in(0,\fz)}\lf\{\fai(x,t)
\lf[\int_{K}\fai(y,t)\,dy\r]^{-1}\r\}\,dx<\fz.$$

\begin{definition}\label{d2.2}
Let $\fai:\rn\times[0,\fz)\to[0,\fz)$ be uniformly locally
integrable. The function $\fai(\cdot,t)$ is
said to satisfy the \emph{uniformly local weight condition for some
$q\in[1,\fz)$}, denoted by $\fai\in\aa^{\loc}_q(\rn)$, if, when $q\in (1,\fz)$,
\begin{equation}\label{2.7}
\aa^{\loc}_q (\fai):=\sup_{t\in [0,\fz)}\sup_{|Q|\le1}\frac{1}{|Q|^q}\int_Q
\fai(x,t)\,dx \lf(\int_Q [\fai(y,t)]^{-q'/q}\,dy\r)^{q/q'}<\fz,
\end{equation}
where $1/q+1/q'=1$, or
\begin{equation}\label{2.8}
\aa^{\loc}_1 (\fai):=\sup_{t\in [0,\fz)}
\sup_{|Q|\le1}\frac{1}{|Q|}\int_Q \fai(x,t)\,dx
\lf(\esup_{y\in Q}[\fai(y,t)]^{-1}\r)<\fz.
\end{equation}
Here, the first supremums are taken over all $t\in[0,\fz)$ and the
second ones over all cubes $Q\subset\rn$ with $|Q|\le1$.
\end{definition}

Recall that in Definition \ref{2.2}, if $\sup_{|Q|\le1}$ in
\eqref{2.7} and \eqref{2.8} is replaced by $\sup_Q$, then
$\aa^{\loc}_q(\rn)$ is just $\aa_q(\rn)$, which was introduced by
Ky \cite{k}.

Let $\aa_{\fz}^{\loc}(\rn):=\cup_{q\in[1,\fz)}\aa_{q}^{\loc}(\rn)$
and define
\begin{equation}\label{2.9}
q(\fai):=\inf\lf\{q\in[1,\fz):\ \fai\in\aa_{q}^{\loc}(\rn)\r\}.
\end{equation}
Obviously, $q(\fai)\in[1,\fz)$. Similar to the proof of \cite[Lemma
2.1(ii)]{ta1}, we know that if $\fai\in \aa^{\loc}_p(\rn)$ for some
$p\in(1,\fz)$, then $\fai\in \aa^{\loc}_r(\rn)$ for some
$r\in[1,p)$. Thus, when $q(\fai)\in(1,\fz)$, we know that
$\fai\not\in \aa^{\loc}_{q(\fai)}(\rn)$. Moreover, from Remark
\ref{r2.2}(iii), we deduce that there exists $\fai\not\in
\aa^{\loc}_1(\rn)$ such that $q(\fai)=1$.

Now, we introduce the notion of growth functions.

\begin{definition}\label{d2.3}
A function $\fai:\rn\times[0,\fz)\rightarrow[0,\fz)$ is called
 a \emph{growth function} if the following hold: \vspace{-0.25cm}
\begin{enumerate}
\item[(i)] The function $\fai$ is a \emph{Musielak-Orlicz function}, namely,
\vspace{-0.2cm}
\begin{enumerate}
    \item[(i)$_1$] the function $\fai(x,\cdot):[0,\fz)\to[0,\fz)$ is an
    Orlicz function for all $x\in\rn$;
    \vspace{-0.2cm}
    \item [(i)$_2$] the function $\fai(\cdot,t)$ is a measurable
    function for all $t\in[0,\fz)$.
\end{enumerate}
\vspace{-0.25cm}
\item[(ii)] $\fai\in \aa^{\loc}_{\fz}(\rn)$.
\vspace{-0.25cm}
\item[(iii)] The function $\fai$ is of positive uniformly lower
    type $p$ for some $p\in(0,1]$ and of
uniformly upper type 1.
\end{enumerate}
\end{definition}

For a growth function $\fai$, we define
\begin{equation}\label{2.10}
m(\fai):=\lf\lfz n\lf[\frac{q(\fai)}{i(\fai)}-1\r]\r\rfz.
\end{equation}

Clearly, $\fai$ as in \eqref{1.1} is a growth function if
$\oz\in A^{\loc}_{\fz}(\rn)$ and $\Phi$ is an Orlicz function of
lower type $p$ for some $p\in(0,1]$ and of upper type 1. It is
known that, for $p\in(0,1]$, if $\Phi(t):=t^p$ for all $t\in
[0,\fz)$, then $\Phi$ is an Orlicz function with
$p_\Phi=p_\Phi^-=p_\Phi^+=p$; for $p\in[\frac{1}{2},1]$, if
$\Phi(t):= t^p/\ln(e+t)$ for all $t\in [0,\fz)$, then $\Phi$ is an
Orlicz function with $p_\Phi^-=p=p_\Phi^+$, $p_\Phi^-$ is not
attainable, but $p_\Phi^+$ is attainable; for
$p\in(0,\frac{1}{2}]$, if $\Phi(t):=t^p\ln(e+t)$ for all $t\in
[0,\fz)$, then $\Phi$ is an Orlicz function with
$p_\Phi^-=p=p_\Phi^+$, $p_\Phi^-$ is attainable, but $p_\Phi^+$ is
not attainable. Another typical and useful growth function is
$$\fai(x,t):=\frac{t^{\az}}{[\ln(e+|x|)]^{\bz}+[\ln(e+t)]^{\gz}}$$
for all $x\in\rn$ and $t\in[0,\fz)$ with some $\az\in(0,1]$,
$\bz\in[0,\fz)$ and $\gz\in [0,2\az(1+\ln2)]$; more precisely,
$\fai\in \aa_1^{\loc}(\rn)$, $\fai$ is of uniformly upper type
$\az$ and $i(\fai)=\az$ (see \cite{k}).

\subsection{Some basic properties of growth functions\label{s2.2}}

\hskip\parindent From now on till the end of this paper, we \emph{always assume
that $\fai$ is a growth function} as in Definition \ref{d2.3}.
Let us now introduce the Musielak-Orlicz space.

The \emph{Musielak-Orlicz space $L^{\fai}(\rn)$} is defined to be the set of
all measurable functions $f$ such that
$\int_{\rn}\fai(x,|f(x)|)\,dx<\fz$ with \emph{Luxembourg norm}
$$\|f\|_{L^{\fai}(\rn)}:=\inf\lf\{\lz\in(0,\fz):\ \int_{\rn}
\fai\lf(x,\frac{|f(x)|}{\lz}\r)\,dx\le1\r\}.
$$
In what follows, for any measurable subset $E$ of $\rn$, we denote
$\int_E\fai(x,t)\,dx$ by $\fai(E,t)$ for any $t\in[0,\fz)$.

The following Lemmas \ref{l2.1}, \ref{l2.2} and \ref{l2.3}
on the properties of growth functions are, respectively,
\cite[Lemmas 4.1, 4.2 and 4.3]{k}.

\begin{lemma}\label{l2.1}
{\rm(i)} Let $\fai$ be a growth function. Then $\fai$ is uniformly
$\sigma$-quasi-subadditive on $\rn\times[0,\fz)$, namely, there
exists a positive constant $C$ such that for all
$(x,t_j)\in\rn\times[0,\fz)$ with $j\in\zz_+$,
$$\fai\lf(x,\sum_{j=1}^{\fz}t_j\r)\le C\sum_{j=1}^{\fz}\fai(x,t_j).$$

{\rm(ii)} Let $\fai$ be a growth function and
$\wz{\fai}(x,t):=\int_0^t\frac{\fai(x,s)}{s}\,ds$ for all
$(x,t)\in\rn\times[0,\fz)$. Then $\wz{\fai}$ is a growth function,
which is equivalent to $\fai$; moreover, $\wz{\fai}(x,\cdot)$ is continuous
and strictly increasing.

{\rm(iii)} A Musielak-Orlicz function $\fai$ is a growth function if
and only if $\fai$ satisfies Definition \ref{d2.3}(ii),
and is of positive uniformly lower type and uniformly
quasi-concave, namely, there exists a positive constant $C$ such
that for all $x\in\rn$, $t,\,s\in[0,\fz)$ and $\lz\in[0,1]$,
$$\lz\fai(x,t)+(1-\lz)\fai(x,s)\le C\fai(x,\lz t+(1-\lz)s).$$
\end{lemma}

\begin{lemma}\label{l2.2}
Let $\fai$ be a growth function. Then

{\rm(i)}
$\int_{\rn}\fai(x,\frac{|f(x)|}{\|f\|_{L^{\fai}(\rn)}})\,dx=1$ for
all $f\in L^{\fai}(\rn)\setminus\{0\}$;

{\rm(ii)} $\lim_{k\to\fz}\|f_k\|_{L^{\fai}(\rn)}=0$ if and only if
$\lim_{k\to\fz}\int_{\rn}\fai(x,|f_k(x)|)\,dx=0$.
\end{lemma}

\begin{lemma}\label{l2.3}
Let $c$ be a positive constant. Then, there exists a positive
constant $C$ such that

{\rm(i)} the inequality $\int_\rn \fai(x,\frac{|f(x)|}{\lz})\,dx\le
c$ for some $\lz\in(0,\fz)$ implies that
$\|f\|_{L^{\fai}(\rn)}\le C\lz;$

{\rm(ii)} the inequality $\sum_j\fai(Q_j,\frac{t_j}{\lz})\le c$ for
some $\lz\in(0,\fz)$ implies that
$$\inf\lf\{\az\in(0,\fz):\ \sum_j\fai\lf(Q_j,\frac{t_j}{\az}\r)\le1\r\}\le
C\lz,$$
where $\{t_j\}_j$ is a sequence of positive numbers and $\{Q_j\}_j$
a sequence of cubes.
\end{lemma}

In what follows, $Q(x,t)$ denotes the \emph{closed cube centered
at $x$ and of the sidelength $t$}. Similarly, given $Q=Q(x,t)$ and
$\lz\in(0,\fz)$, we write $\lz Q$ for the {\it$\lz$-dilated cube},
which is the cube with the same \emph{center} $x$ and with
\emph{sidelength} $\lz t$. Given a Lebesgue measurable set $E$ and
a weight $\oz\in A^{\loc}_{\fz} (\rn)$, let
$\oz(E):=\int_{E}\oz(x)\,dx$. For any $\oz\in A^{\loc}_{\fz}
(\rn)$, $L^p_{\oz}(\rn)$, with $p\in(0,\fz)$, denotes the \emph{set
of all measurable functions $f$ such that}
$$\|f\|_{L^p_{\oz}(\rn)}:=\lf\{\int_{\rn}|f(x)|^p
\oz(x)\,dx\r\}^{1/p}<\fz,$$
and $L^{\fz}_{\oz} (\rn):=
L^{\fz}(\rn)$. For a positive constant $\wz{C}$, any locally
integrable function $f$ and $x\in\rn$, the {\it local
Hardy-Littlewood maximal function $M^{\loc}_{\wz{C}}(f)(x)$} is defined
by
$$M^{\loc}_{\wz{C}}(f)(x):=\sup_{Q\ni x,\,|Q|\le \wz{C}}
\frac{1}{|Q|}\int_{Q}|f(y)|\,dy,$$ where the supremum is taken over
all cubes $Q\subset\rn$ such that $Q\ni x$ and $|Q|\le\wz{C}$. If
$\wz{C}=1$, we denote $M^{\loc}_{\wz{C}}(f)$ simply by
$M^{\loc}(f)$.

\begin{lemma}\label{l2.4}
$\mathrm{(i)}$ Let $p\in[1,\fz),\,\fai\in \aa^{\loc}_p(\rn)$, and
$Q$ be a unit cube, namely, $l(Q)=1$. Then there exist a
$\wz{\fai}\in \aa_p (\rn)$, such that
$\wz{\fai}(\cdot,t)=\fai(\cdot,t)$ on $Q$ for all $t\in[0,\fz)$,
and a positive constant $C$, independent of $Q$, $t$ and $\fai$,
such that $\aa_p (\wz{\fai})\le C\aa^{\loc}_p(\fai)$.

$\mathrm{(ii)}$ Let $\fai\in \aa^{\loc}_{q}(\rn)$ with $q\in[1,\fz)$
and $Q:=Q(x_0,l(Q))$. Then there exists a positive constant $C$ such
that for all $t\in[0,\fz)$, $\fai(2Q,t)\le C\fai(Q,t)$ when
$l(Q)<1$, and $\fai(Q(x_0,l(Q)+1),t)\le C\fai(Q,t)$ when $l(Q)\ge1$.

$\mathrm{(iii)}$ If $p\in(1,\fz)$ and $\fai\in \aa^{\loc}_{p}(\rn)$,
then there exists a positive constant $C$ such that for all
measurable functions $f$ on $\rn$ and $t\in[0,\fz)$,
$$\int_{\rn}\lf[M^{\loc}(f)(x)\r]^p\fai(x,t)\,dx\le
C\int_{\rn}|f(x)|^p\fai(x,t)\,dx.$$

$\mathrm{(iv)}$ If $\fai\in \aa_{p}(\rn)$ with $p\in[1,\fz)$, then
there exists a positive constant $C$ such that for all cubes
$Q_1,\,Q_2\subset\rn$ with $Q_1\subset Q_2$ and $t\in[0,\fz)$,
$$\frac{\fai(Q_2,t)}{\fai(Q_1,t)}\le C\lf[\frac{|Q_2|}{|Q_1|}\r]^p.$$
\end{lemma}

\begin{proof}
If $\fai$ is as in \eqref{1.1}, then (i) is just
\cite[Lemma\,1.1]{r01}, (ii) and (iii) are just \cite[Lemma
2.1 and Corollary 2.1]{ta1} and (vi) is included, for example, in
\cite{g79,gr1,st93}. Since $\fai$ satisfies the uniformly local
weight condition on $t$, the proof of the desired conclusions of
Lemma \ref{l2.4} is a slight modification of these classical
results. We omit the details here. This finishes the proof of Lemma
\ref{l2.4}.
\end{proof}

\begin{remark}\label{r2.3}
Let $\wz{C}$ be a positive constant.
Then (i) through (v) of Lemma \ref{l2.4} are also true, if $l(Q)=1$,
$l(Q)\ge1,\,l(Q)<1,\, Q(x_0,\,l(Q)+1)$ and $M^{\loc}$ are
respectively replaced  by $l(Q)=\wz{C}$,
$l(Q)\ge\wz{C},\,l(Q)<\wz{C}$, $Q(x_0,\,l(Q)+\wz{C})$ and
$M^{\loc}_{\wz{C}}$ (see \cite[Remark 1.5]{r01} for the proofs). In
this case, the constants, appearing in (i), (ii) and (iii) of
Lemma \ref{l2.4}, depend on $\wz{C}$.
\end{remark}

For $\cd(\rn),\,\cd'(\rn)$ and $L^{\fz}_{\fai}(\rn)$, we have the
following conclusions.
\begin{lemma}\label{l2.5}
Let $\fai\in \aa^{\loc}_{\fz}(\rn)$, $q(\fai)$ be as in \eqref{2.9}
and $p\in(q(\fai),\fz]$.

$\mathrm{(i)}$ If $\frac{1}{p}+\frac{1}{p'}=1$, then
$\cd(\rn)\subset L^{p'}_{[\fai(\cdot,1)]^{-1/(p-1)}}(\rn)$.

$\mathrm{(ii)}$ $L^{p}_{\fai(\cdot,1)}(\rn)\subset\cd' (\rn)$ and
the inclusion is continuous.

$\mathrm{(iii)}$ Let $\phi\in\cd(\rn)$ and
$\int_{\rn}\phi(x)\,dx=1$. If $q\in(q(\fai),\fz)$, then for any
$f\in L^q_{\fai(\cdot,1)}(\rn)$, $f\ast\phi_t\to f$ in
$L^q_{\fai(\cdot,1)}(\rn)$ as $t\to0$, here and in what follows,
$\phi_t (x):=\frac{1}{t^n} \phi(\frac{x}{t})$ for all $x\in\rn$
and $t\in(0,\fz)$.
\end{lemma}

We remark that (i) and (ii) of Lemma \ref{l2.5}, and Lemma
\ref{l2.5}(iii) are, respectively, \cite[Lemma 2.2 and Proposition
2.1]{ta1}.

\section{Local Musielak-Orlicz Hardy spaces and their maximal
function characterizations\label{s3}}

\hskip\parindent In this section, we introduce the local Musielak-Orlicz
Hardy space $h_{\fai,\,N}(\rn)$ via the local
grand maximal function and establish its local vertical and
nontangential maximal function characterizations via a local
Calder\'on reproducing formula and some useful estimates obtained
by Rychkov \cite{r01}. We also introduce the atomic local
Musielak-Orlicz Hardy space $h^{\fai,\,q,\,s}(\rn)$ and give some of their basic properties.

First, we introduce some local maximal functions. For $N\in\nn$ and
$R\in(0,\fz)$, let
\begin{eqnarray*}
\cd_{N,\,R}(\rn):=\lf\{\pz\in\cd(\rn):\ \supp(\pz)
\subset B(0,R),\ \|\pz\|_{\cd_{N}(\rn)}:=\sup_{x\in\rn}
\sup_{{\az\in\nn^n},\,{|\az|\le N}}|\partial^{\az}\pz(x)|\le1\r\}.
\end{eqnarray*}

\begin{definition}\label{d3.1}
Let $N\in\nn$ and $R\in(0,\fz)$. For any $f\in\cd'(\rn)$, the {\it
local nontangential grand maximal function} $\wz{\cg}_{N,\,R} (f)$
of $f$ is defined by setting, for all $x\in\rn$,
\begin{eqnarray}\label{3.1}
\wz{\cg}_{N,\,R} (f)(x):=\sup\lf\{|\pz_t\ast
f(z)|:\,|x-z|<t<1,\,\pz\in\cd_{N,\,R}(\rn)\r\},
\end{eqnarray}
and the {\it local vertical grand maximal function} $\cg_{N,\,R}
(f)$ of $f$ is defined by setting, for all $x\in\rn$,
\begin{eqnarray}\label{3.2}
\cg_{N,\,R} (f)(x):=\sup\lf\{|\pz_t\ast
f(x)|:\,t\in(0,1),\,\pz\in\cd_{N,\,R}(\rn)\r\}.
\end{eqnarray}
\end{definition}

For the sake of convenience, when $R=1$, we denote $\cd_{N,\,R}(\rn)$,
$\wz{\cg}_{N,\,R} (f)$ and $\cg_{N,\,R}(f)$ simply by
$\cd^0_{N}(\rn)$, $\wz{\cg}^0_{N}(f)$ and $\cg^0_{N}(f)$,
respectively; when $R=2^{3(10+n)}$, we denote $\cd_{N,\,R}(\rn)$,
$\wz{\cg}_{N,\,R} (f)$ and $\cg_{N,\,R}(f)$ simply by
$\cd_{N}(\rn)$, $\wz{\cg}_{N}(f)$ and $\cg_{N}(f)$, respectively.
For any $N\in\nn$ and $x\in\rn$, obviously,
$$\cg^{0}_N(f)(x)\le\cg_N (f)(x)\le\wz{\cg}_N (f)(x).$$

For the local grand maximal function $\cg^0_N (f)$, we have the
following Proposition \ref{p3.1}, which is just \cite[Proposition
2.2]{ta1}.

\begin{proposition}\label{p3.1}
Let $N\ge2$.

$\mathrm{(i)}$ Then there exists a positive constant $C$ such that
for all $f\in L^1_{\loc}(\rn)\cap\cd'(\rn)$ and almost every
$x\in\rn$, 
$$|f(x)|\le \cg^0_N (f)(x)\le M^{\loc}(f)(x).$$

$\mathrm{(ii)}$ If $\fai\in \aa_p^{\loc}(\rn)$ with $p\in(1,\fz)$,
then $f\in L^p_{\fai(\cdot,1)}(\rn)$ if and only if $f\in \cd'(\rn)$
and $\cg^0_N (f)\in L^p_{\fai(\cdot,1)}(\rn)$; moreover,
$$\|f\|_{L^p_{\fai(\cdot,1)}(\rn)}\sim\|\cg^0_N
(f)\|_{L^p_{\fai(\cdot,1)}(\rn)}$$ 
with the implicit positive constants
independent of $f$.
\end{proposition}

Now, we introduce the local Musielak-Orlicz Hardy space via
the local grand maximal function as follows.

\begin{definition}\label{d3.2}
Let $\fai$ be a growth function as in Definition \ref{d2.3},
$m(\fai)$ as in \eqref{2.10}, and
$\wz{N}_{\fai}:=m(\fai)+2$. For each $N\in\zz_+$ with $N\ge
\wz{N}_{\fai}$, the {\it local Musielak-Orlicz Hardy space $h_{\fai,\,N}(\rn)$}
is defined by
$$h_{\fai,\,N}(\rn):=\lf\{f\in\cd'(\rn):\ \cg_N (f)\in
L^{\fai}(\rn)\r\}.$$
Moreover, let $\|f\|_{h_{\fai,\,N}(\rn)}:=\|\cg_N (f)\|_{L^{\fai}(\rn)}$.
\end{definition}

Obviously, for any integers $N_1$ and $N_2$ with $N_1\ge N_2\ge
\wz{N}_{\fai}$, $h_{\fai,\,\wz{N}_{\fai}}(\rn)\subset
h_{\fai,\,N_2}(\rn) \subset h_{\fai,\,N_1}(\rn)$, and the inclusions
are continuous.

Next, we introduce the weighted local atoms, via which, we introduce
the \emph{atomic local Musielak-Orlicz Hardy space}.

\begin{definition}\label{d3.3}
(i) For each cube $Q\subset\rn$, denote by $L^q_{\fai}(Q)$, $q\in [1, \fz]$,
the \emph{set of all measurable functions $f$ on $\rn$
supported in $Q$ such that}
\begin{equation*}
\|f\|_{L^q_{\fai}(Q)}:=
\begin{cases}\dsup_{t\in (0,\fz)}
\lf[\dfrac{1}
{\fai(Q,t)}\dint_{\rn}|f(x)|^q\fai(x,t)\,dx\r]^{1/q}<\fz,& q\in [1,\fz),\\
\|f\|_{L^{\fz}(Q)}<\fz,&q=\fz.
\end{cases}
\end{equation*}

(ii) Let $\fai$ be a growth function as in Definition
\ref{d2.3} which satisfies the fact that $\int_\rn\fai(x,t)\,dx<\fz$
for all $t\in[0,\fz)$. The \emph{space $L^q_{\fai}(\rn)$} for $q\in[1,\fz]$
is defined as in (i) with $Q$ replaced by $\rn$, here and in what follows,
$$\fai(\rn,t):=\int_\rn\fai(x,t)\,dx.$$
\end{definition}

It is easy to see that both
$(L^q_{\fai}(Q),\|\cdot\|_{L^q_{\fai}(Q)})$ and
$(L^q_{\fai}(\rn),\|\cdot\|_{L^q_{\fai}(\rn)})$ are Banach spaces.

\begin{definition}\label{d3.4}
Let $\fai$ be a growth function as in Definition
\ref{d2.3}, $q(\fai)$ and $m(\fai)$
respectively as in \eqref{2.9} and \eqref{2.10}. A triplet
$(\fai,\,q,\,s)$ is called {\it admissible} if $q\in(q(\fai),\fz]$,
$s\in\nn$ and $s\ge m(\fai)$. A measurable function $a$ on $\rn$ is
called a {\it $(\fai,\,q,\,s)$-atom} if there exists a cube
$Q\subset\rn$ such that

$\mathrm{(i)}$ $\supp(a)\subset Q$;

$\mathrm{(ii)}$
$\|a\|_{L^q_{\fai}(Q)}\le\|\chi_Q\|_{L^\fai(\rn)}^{-1}$;

$\mathrm{(iii)}$ $\int_{\rn}a(x)x^{\az}\,dx=0$ for all $\az\in\nn^n$
with $|\az|\le s$, when $l(Q)<1$.\\ Moreover, a function $a$ on
$\rn$ is called  a {\it $(\fai,\,q)$-single-atom} with
$q\in(q(\fai),\fz]$, if
$$\|a\|_{L^q_{\fai}(\rn)}\le\|\chi_{\rn}\|_{L^\fai(\rn)}^{-1}.$$
\end{definition}

\begin{remark}\label{r3.1}
Let $q\in(q(\fai),\fz]$. It is easy to see that a finite linear
combination of $(\fai,\,q)$-single-atoms is still a multiple of a
$(\fai,\,q)$-single-atom.
\end{remark}

\begin{definition}\label{d3.5}
Let $\fai$ be a growth function as in Definition \ref{d2.3}
and $(\fai,\,q,\,s)$ admissible.
The {\it atomic local Musielak-Orlicz Hardy space},
$h^{\fai,\,q,\,s}(\rn)$, is defined to be the set of all $f\in\cd'(\rn)$
satisfying the fact that $f=\sum_{i=0}^{\fz}b_i$ in $\cd'(\rn)$, where
$\{b_i\}_{i=1}^\fz$ is a sequence of multiples of
$(\fai,\,q,\,s)$-atoms with $\supp(b_i)\subset Q_i$ and $b_0$ is a
multiple of some $(\fai,\,q)$-single-atom, and
$$\sum_{i=1}^{\fz}\fai\lf(Q_i,\|b_i\|_{L^q_{\fai}(Q_i)}\r)
+\fai\lf(\rn,\|b_0\|_{L^q_{\fai}(\rn)}\r)<\fz.$$ Moreover, letting
\begin{eqnarray*}
&&\blz_q(\{b_i\}_{i\in\nn}):= \inf\lf\{\lz\in(0,\fz):\
\sum_{i=1}^{\fz}\fai\lf(Q_i,\frac{\|b_i\|_{L^q_{\fai}(Q_i)}}{\lz}\r)
+\fai\lf(\rn,\frac{\|b_0\|_{L^q_{\fai}(\rn)}}{\lz}\r)\le1\r\},
\end{eqnarray*}
the \emph{quasi-norm} of $f\in h^{\fai,\,q,\,s}_{\oz}(\rn)$ is
defined by 
$$\|f\|_{h^{\fai,\,q,\,s}(\rn)}:=\inf\lf\{
\blz_q\lf(\{b_i\}_{i\in\nn}\r)\r\},$$ 
where the infimum is taken over all
the decompositions of $f$ as above.
\end{definition}

Next, we introduce some local vertical, tangential and
nontangential maximal functions, and establish the
characterizations of the \emph{local Musielak-Orlicz Hardy space
$h_{\fai,\,N}(\rn)$} in terms of  these
local maximal functions.

\begin{definition}\label{d3.6}
Let
\begin{eqnarray}\label{3.3}
\pz_0\in\cd(\rn)\,\, \text{with}\,\,\int_{\rn}\pz_0 (x)\,dx\neq0.
\end{eqnarray}
For $j\in\nn$, $A,\,B\in[0,\fz)$ and $y\in\rn$, let
$m_{j,\,A,\,B}(y):=(1+2^j |y|)^A 2^{B|y|}$. The {\it local vertical
maximal function} $\pz_0^{+}(f)$ of $f$ associated with $\pz_0$ is
defined by setting, for all $x\in\rn$,
\begin{eqnarray}\label{3.4}
\pz_0^{+}(f)(x):= \sup_{j\in\nn}|(\pz_0)_j \ast f(x)|,
\end{eqnarray}
the {\it local tangential Peetre-type  maximal function
$\pz^{\ast\ast}_{0,\,A,\,B}(f)$} of $f$ associated with $\pz_0$ is
defined by setting, for all $x\in\rn$,
\begin{eqnarray}\label{3.5}
\pz^{\ast\ast}_{0,\,A,\,B}(f)(x):=\sup_{j\in\nn,\,y\in\rn}
\frac{|(\pz_0)_j \ast f(x-y)|}{m_{j,\,A,\,B}(y)}
\end{eqnarray}
and the {\it local nontangential maximal function
$(\pz_0)^{\ast}_{\triangledown}(f)$} of $f$ associated with $\pz_0$ is
defined by setting, for all $x\in\rn$,
\begin{eqnarray}\label{3.6}
(\pz_0)^{\ast}_{\triangledown}(f)(x):= \sup_{|x-y|<t<1}|(\pz_0)_t
\ast f(y)|,
\end{eqnarray}
here and in what follows, for all $x\in\rn$,
$(\pz_0)_j(x):=2^{jn}\pz_0 (2^j x)$ for all $j\in\nn$ and $(\pz_0)_t
(x):=\frac{1}{t^n}\pz_0 (\frac{x}{t})$ for all $t\in(0,\fz)$.
\end{definition}

Obviously, for any $x\in\rn$, we have
$$\pz_0^{+}(f)(x)\le(\pz_0)^{\ast}_{\triangledown}(f)(x)
\ls\pz^{\ast\ast}_{0,\,A,\,B}(f)(x).$$ 
We remark that the local
tangential Peetre-type maximal function
$\pz^{\ast\ast}_{0,\,A,\,B}(f)$ was introduced by Rychkov
\cite{r01}.

In order to establish the local vertical and the local
nontangential maximal function characterizations of
$h_{\fai,\,N}(\rn)$, we first establish some relations in the norm
of $L^{\fai}(\rn)$ of the local maximal functions
$\pz^{\ast\ast}_{0,\,A,\,B}(f),\,\pz_0^{+}(f)$ and
$\wz{\cg}_{N,\,R}(f)$. We begin with some technical lemmas and the
following Lemma \ref{l3.1} is just \cite[Theorem\,1.6]{r01}.

\begin{lemma}\label{l3.1}
Let $\pz_0$ be as in \eqref{3.3} and $\pz(x):=\pz_0
(x)-\frac{1}{2^n}\pz_0 (\frac{x}{2})$ for all $x\in\rn$. Then for
any given integer $L\in\nn$, there exist $\ez_0,\,\ez\in\cd(\rn)$
such that $L_{\ez}\ge L$ and for all $f\in\cd'(\rn)$,
$$f=\ez_0 \ast\pz_0 \ast f+\sum_{j=1}^{\fz}\ez_j\ast\pz_j\ast f$$
in $\cd'(\rn)$.
\end{lemma}

\begin{remark}\label{r3.2}
Let $\pz_0,\,\pz,\,\ez_0$ and $\ez$ be as in Lemma \ref{l3.1}. From
the proof of \cite[Theorem 1.6]{r01}, it is easy to deduce that for
any $j\in\nn$ and $f\in\cd'(\rn)$,
$$f=(\ez_0)_j\ast(\pz_0)_j \ast f+\sum_{k=j+1}^{\fz}
\ez_k\ast\pz_k \ast f$$ in $\cd'(\rn)$ (see also
\cite[(2.11)]{r01}). We omit the details.
\end{remark}

For $f\in L^1_{\loc}(\rn)$, $B\in[0,\fz)$ and $x\in\rn$, let
\begin{eqnarray}\label{3.7}
K_B f(x):=\int_{\rn}|f(y)|2^{-B|x-y|}\,dy,
\end{eqnarray}
 here and in what follows,
$L^1_{\loc}(\rn)$ denotes the \emph{set of all locally integrable
functions on $\rn$}.
\begin{lemma}\label{l3.2}
Let $p\in(1,\fz)$, $q\in(1,\fz]$, and $\fai\in \aa^{\loc}_p (\rn)$.
Then there exists a positive constant $C$ such that for all sequences
$\{f^j\}_j$ of measurable functions and all $t\in(0,\fz)$,
\begin{eqnarray}\label{3.8}
\lf\|\lf\{M^{\loc}(f^j)\r\}_j\r\|_{L^p_{\fai(\cdot,t)}(l_q)}\le C
\lf\|\lf\{f^j\r\}_j\r\|_{L^p_{\fai(\cdot,t)}(l_q)},
\end{eqnarray}
here and in what follows,
$$\|\{f^j\}_j\|_{L^p_{\fai(\cdot,t)}(l_q)}:=\lf\|\lf\{\sum_{j}|f^j|^q\r\}^{1/q}
\r\|_{L^p_{\fai(\cdot,t)}(\rn)}.$$
Also, there exist positive
constants $C$ and $B_0:=B_0 (\fai,n)$ such that for all $B\ge B_0/p$,
\begin{eqnarray}\label{3.9}
\lf\|\lf\{K_B (f^j)\r\}_j\r\|_{L^p_{\fai(\cdot,t)}(l_q)}\le C
\lf\|\lf\{f^j\r\}_j\r\|_{L^p_{\fai(\cdot,t)}(l_q)}.
\end{eqnarray}
\end{lemma}

Lemma \ref{l3.2} is just \cite[Lemma\,2.11]{r01}. Moreover, from the
proof of \cite[Lemma\,2.11]{r01}, it is easy to deduce that
\eqref{3.8} also holds for $M^{\loc}_{\wz{C}}$ with any given
positive constant $\wz{C}$. In this case, the positive constant $C$
in \eqref{3.8} depends on $\wz{C}$.

\begin{lemma}\label{l3.3}
Let $\pz_0$ be as in \eqref{3.3} and $r\in(0,\fz)$. Then there
exists a positive constant $A_0$, depending only on the support of
$\pz_0$, such that for any $A\in(\max\{A_0,\frac{n}{r}\},\fz)$ and
$B\in[0,\fz)$, there exists a positive constant $C$, depending only
on $n,\,r,\,\pz_0,\,A$ and $B$, satisfying the fact that for all
$f\in\cd'(\rn)$, $x\in\rn$ and $j\in\nn$,
$$\lf[(\pz_0 )^{\ast}_{j,\,A,\,B}(f)(x)\r]^r\le C\sum_{k=j}^{\fz}
2^{(j-k)(Ar-n)}\lf\{M^{\loc}(|(\pz_0 )_k \ast
f|^r)(x)+K_{Br}(|(\pz_0 )_k \ast f|^r)(x)\r\},$$
where, for all $x\in\rn$,
$$(\pz_0)^{\ast}_{j,\,A,\,B}(f)(x):=\sup_{y\in\rn}\frac{|(\pz_0 )_j \ast
f(x-y)|}{m_{j,\,A,\,B}(y)}.$$
\end{lemma}

Lemma \ref{l3.3} is just \cite[Lemma 3.11]{yys}.

\begin{theorem}\label{t3.1}
Let $\fai$ be a growth function as in Definition \ref{d2.3},
$R\in(0,\fz)$, $\pz_0,\,q(\fai)$ and $i(\fai)$ be respectively as
in \eqref{3.3}, \eqref{2.9} and $\eqref{2.3}$, and $\pz_0^{+}(f),\,
\pz^{\ast\ast}_{0,\,A,\,B}(f)$, and $\wz{\cg}_{N,\,R} (f)$ be
respectively as in \eqref{3.4}, \eqref{3.5} and \eqref{3.1}. Let
$A_1 :=\max\{A_0,\,nq(\fai)/i(\fai)\}$, $B_1:= B_0/i(\fai)$ and
integer $N_0:=\lfz2A_1\rfz+1$, where $A_0$ and $B_0$ are
respectively as in Lemmas \ref{3.3} and \ref{3.2}. Then for any
$A\in(A_1,\fz)$, $B\in(B_1,\fz)$ and integer $N\ge N_0$, there
exists a positive constant $C$, depending only on
$A,\,B,\,N,\,R,\,\pz_0,\,\fai$ and $n$, such that for all
$f\in\cd'(\rn)$,
\begin{eqnarray}\label{3.10}
\lf\|\pz^{\ast\ast}_{0,\,A,\,B}(f)\r\|_{L^{\fai}(\rn)}\le C
\lf\|\pz_0^{+}(f)\r\|_{L^{\fai}(\rn)}
\end{eqnarray}
and
\begin{eqnarray}\label{3.11}
\lf\|\wz{\cg}_{N,\,R} (f)\r\|_{L^{\fai}(\rn)}\le C
\lf\|\pz_0^{+}(f)\r\|_{L^{\fai}(\rn)}.
\end{eqnarray}
\end{theorem}

\begin{proof}
Let $f\in\cd'(\rn)$. First, we prove \eqref{3.10}. Let
$A\in(A_1,\fz)$ and $B\in(B_1,\fz)$. By $A_1
=\max\{A_0,\,nq(\fai)/i(\fai)\}$ and $B_1=B_0 /i(\fai)$, we know
that there exists $r_0\in(0,\frac{i(\fai)}{q(\fai)})$ such that
$A>\frac{n}{r_0}$ and $Br_0>\frac{B_0}{q(\fai)}$,  where $A_0$ and
$B_0$ are respectively as in Lemmas \ref{3.3} and \ref{l3.2}. Thus,
by Lemma \ref{l3.3}, for all $x\in\rn$, we know that
\begin{eqnarray}\label{3.12}
\lf[(\pz_0)^{\ast}_{j,\,A,\,B}(f)(x)\r]^{r_0}&\ls&\sum_{k=j}^{\fz}
2^{(j-k)(Ar_0 -n)}\lf\{M^{ \loc}\lf(|(\pz_0)_k\ast f|^{r_0}\r)(x)
+K_{Br_0}\lf(|(\pz_0)_k\ast f|^{r_0}\r)(x)\r\}.\qquad
\end{eqnarray}
 Let $\pz^{+}_0 (f)$ and $\pz^{\ast\ast}_{0,\,A,\,B}(f)$ be
respectively as in \eqref{3.4} and \eqref{3.5}. We notice that for
any $x\in\rn$ and $k\in\nn$, $|(\pz_0)_k \ast f (x)|\le\pz^{+}_0
(f)(x)$, which, together with \eqref{3.12}, implies that for all
$x\in\rn$,
\begin{eqnarray}\label{3.13}
\lf[\pz^{\ast\ast}_{0,\,A,\,B}(f)(x)\r]^{r_0}\ls
M^{\loc}\lf([\pz^{+}_0 (f)]^{r_0})(x)+K_{Br_0}([\pz^{+}_0
(f)]^{r_0}\r)(x).
\end{eqnarray}
From \eqref{3.13} and Lemma \ref{l2.1}(i), we deduce that
\begin{eqnarray}\label{3.14}
\int_{\rn}\fai\lf(x,\pz^{\ast\ast}_{0,\,A,\,B}(f)(x)\r)\,dx
&&\ls\int_{\rn}\fai\lf(x,\lf\{M^{\loc}\lf([\pz^{+}_0
(f) ]^{r_0}\r)(x)\r\}^{1/{r_0}}\r)\,dx\nonumber \\
&&\hs\hs+\int_{\rn}\fai\lf(x,\lf\{K_{Br_0}\lf([\pz^{+}_0
(f)]^{r_0}\r)(x)\r\}^{1/{r_0}}\r)\,dx\nonumber\\
&&=:\mathrm{I_1}+\mathrm{I_2}.
\end{eqnarray}

Now, we estimate $\mathrm{I_1}$.  By $r_0 <\frac{i(\fai)}{q(\fai)}$,
we see that there exist $q\in(q(\fai),\fz)$ and $p_0\in(0,i(\fai))$
such that $r_0 q<p_0$, $\fai\in \aa^{\loc}_q (\rn)$ and $\fai$ is of
uniformly lower type $p_0$. For any $\az\in(0,\fz)$ and $g\in
L^1_{\loc} (\rn)$, let $g=g \chi_{\{x\in\rn:\ |g(x)|\le\az\}}+g
\chi_{\{x\in\rn:\ |g(x)|>\az\}}=: g_1 +g_2$. It is easy to see that
$$\{x\in\rn:\ M^{\loc}(g)(x)>2\az\}\subset \{x\in\rn:\
M^{\loc}(g_2)(x)>\az\},$$
which, together with Lemma \ref{l2.4}(iii),
implies that for all $t\in(0,\fz)$,
\begin{eqnarray}\label{3.15}
\int_{\{x\in\rn:\ M^{\loc}(g)(x)>2\az\}}\fai(x,t)\,dx
&&\le\int_{\{x\in\rn:\ M^{\loc}(g_2)(x)>\az\}}\fai(x,t)\,dx\noz\\
&&\le\frac{1}{\az^q}\int_{\rn} \lf[M^{\loc}(g_2)(x)\r]^q\fai(x,t)\,dx\nonumber \\
&&\ls\frac{1}{\az^q}\int_{\rn}|g_2 (x)|^q\fai(x,t)\,dx\noz\\
&&\sim\frac{1}{\az^q}\int_{\{x\in\rn:\
|g(x)|>\az\}}|g(x)|^q\fai(x,t)\,dx.
\end{eqnarray}
Thus, for any $\az\in(0,\fz)$, from \eqref{3.15}, we infer that
\begin{eqnarray}\label{3.16}
&&\int_{\{x\in\rn:\ [M^{\loc}([\pz^+_0
(f)]^{r_0})(x)]^{1/{r_0}}>\az\}}\fai(x,t)\,dx\noz\\
&&\hs\ls\frac{1}{\az^{r_0 q}}\int_{\{x\in\rn:\ [\pz^{+}_0
(f)(x)]^{r_0}>\frac{\az^{r_0}}{2}\}}[\pz^{+}_0 (f)(x)]^{r_0
q}\fai(x,t)\,dx\nonumber\\
&&\hs\sim\sz_{\pz^+_0
(f),\,t}\lf(\frac{\az}{2^{1/{r_0}}}\r)+\frac{1}{\az^{r_0
q}}\int^{\fz}_{\frac{\az}{2^{1/{r_0}}}}r_0 qs^{r_0 q-1}\sz_{\pz^+_0
(f),\,t}(s)\,ds,
\end{eqnarray}
here and in what follows,
$$\sz_{\pz^+_0(f),\,t}(s):=\int_{\{x\in\rn:\ \pz^+_0 (f)(x)>s\}}\fai(x,t)\,dx.$$
By Lemma \ref{l2.1}(ii), we know that $\fai(x,t)\sim\int_0^t
\frac{\bfai(x,s)}{s}\,ds$ for all $x\in\rn$ and $t\in(0,\fz)$. From
this, \eqref{3.16} and the uniformly lower type $p_0$ property of
$\fai$, $r_0 q<p_0$ and Fubini's theorem, it follows that
\begin{eqnarray}\label{3.17}
\hs\mathrm{I_1}&\sim&\int_{\rn}\lf\{\int^{\lf\{M^{\loc}([\pz^{+}_0
(f)]^{r_0})(x)\r\}^{1/{r_0}}}_0\frac{\fai(x,t)}{t}\,dt\r\}\,dx\nonumber\\
 &\sim&\int^{\fz}_0\frac{1}{t}\int_{\{x\in\rn:\
\lf[M^{\loc}([\pz^{+}_0 (f)]^{r_0})\r]^{1/{r_0}}>t\}}
\fai(x,t)\,dx\,dt\nonumber\\
&\ls&\int^{\fz}_0\frac{1}{t}\int_{\{x\in\rn:\ \pz^+_0
(f)>\frac{t}{2^{1/r_0}}\}}\fai(x,t)\,dx\,dt\nonumber\\
&&\hs+\int_0^{\fz}\frac{1}{t^{r_0 q+1}}
\int^{\fz}_{\frac{t}{2^{1/{r_0}}}}r_0 qs^{r_0 q-1}\int_{\{x\in\rn:\
\pz^+_0(f)>s\}}\fai(x,s)\,dx\,ds\,dt\nonumber\\
&\ls& \int_{\rn}\fai\lf(x,\pz^+_0(f)(x)\r)\,dx
+\int^{\fz}_0 r_0 q s^{r_0 q-1}\int_{\{x\in\rn:\
\pz^+_0(f)(x)>s\}}\int_0^{2^{\frac{1}{r_0}}s}\frac{\fai(x,t)}{t^{r_0
q+1}}\,dt\,dx\,ds\nonumber\\
&\sim&\int_{\rn}\fai\lf(x,\pz^+_0(f)(x)\r)\,dx.
\end{eqnarray}

Now, we estimate $\mathrm{I_2}$. For any $\az\in(0,\fz)$ and $g\in
L^1_{\loc} (\rn)$, let $g_1$ and $g_2$ be as above. For
$H\in[\frac{B_0}{q},\fz)$, let $\int_{\rn}2^{-H|x-y|}\,dy=: c_H$. It
is easy to see that for all $x\in\rn$, $K_H (g_1) (x)\le c_H \az$,
which implies that
$$\lf\{x\in\rn:\,K_H (g)(x)>(c_H +1)\az\r\}\subset\lf\{x\in\rn:\,K_H
(g_2) (x)>\az\r\},$$ where $K_H$ is as in \eqref{3.7}. Thus, by
Lemma \ref{l3.2}, we know that for all $t\in(0,\fz)$,
\begin{eqnarray*}
\int_{\{x\in\rn:\ K_H g (x)>(c_H
+1)\az\}}\fai(x,t)&\le&\int_{\{x\in\rn:\ K_H g_2
(x)>\az\}}\fai(x,t)\,dx\\
&\ls&\frac{1}{\az^q}\int_{\{x\in\rn:\,|g(x)|>\az\}}|g(x)|^q\fai(x,t)\,dx.
\end{eqnarray*}
Similar to \eqref{3.16}, from the above estimate, $Br_0>B_0/q$ and
Lemma \ref{3.2}, we deduce that
\begin{eqnarray*}
&&\int_{\{x\in\rn:\ [K_{Br_0}([\pz^+_0
(f)]^{r_0})(x)]^{1/{r_0}}>\az\}}\fai(x,t)\,dx\\
&&\hs\ls\sz_{\pz^+_0
(f),\,t}\lf(\frac{\az}{(c_{Br_0}+1)^{1/r_0}}\r)+\frac{1}{\az^{r_0
q}}\int^{\fz}_{\frac{\az}{(c_{Br_0}+1)^{1/r_0}}}r_0 qs^{r_0
q-1}\sz_{\pz^+_0 (f),\,t}(s)\,ds.
\end{eqnarray*}
By this, similar to the estimate of $\mathrm{I_1}$, we know that
\begin{eqnarray}\label{3.18}
\mathrm{I_2}\ls\int_{\rn}\fai\lf(x,\pz^+_0 (f)(x)\r)\,dx.
\end{eqnarray}

Thus, from \eqref{3.14}, \eqref{3.17} and \eqref{3.18}, we infer
that
$$\int_{\rn}\fai\lf(x,\pz^{\ast\ast}_{0,\,A,\,B}(f)(x)\r)\,dx\ls
\int_{\rn}\fai\lf(x,\pz^+_0 (f)(x)\r)\,dx.$$ Replacing $f$ by
$f/\lz$ with $\lz\in(0,\fz)$ in the above inequality, and noticing
that
$$\pz^{\ast\ast}_{0,\,A,\,B}(f/\lz)
=\pz^{\ast\ast}_{0,\,A,\,B}(f)/\lz$$ and $\pz^+_0 (f/\lz)=\pz^+_0
(f)/\lz$, we see that
\begin{eqnarray}\label{3.19}
\int_{\rn}\fai\lf(x,\frac{\pz^{\ast\ast}_{0,\,A,\,B}(f)(x)}{\lz}\r)\,dx\ls
\int_{\rn}\fai\lf(x,\frac{\pz^+_0 (f)(x)}{\lz}\r)\,dx,
\end{eqnarray}
which, together with the arbitrariness of $\lz\in(0,\fz)$, implies
\eqref{3.10}.

 Now, we prove \eqref{3.11}. Similar to the proof of
 \cite[pp.\,20-22]{yys}, by Lemma \ref{l3.1}, for all $x\in\rn$,
 we have
$\wz{\cg}_{N,R}(f)(x)\ls\pz^{\ast\ast}_{0,\,A,\,B}(f)(x)$, which,
together with \eqref{3.19}, implies that for any $\lz\in(0,\fz)$,
$$\int_{\rn}\fai\lf(x,\frac{\wz{\cg}_{N,\,R}(f)(x)}{\lz}\r)\,dx\ls
\int_{\rn}\fai\lf(x,\frac{\pz^{+}_0 (f)(x)}{\lz}\r)\,dx.$$ From
this, we deduce that \eqref{3.11} holds, which completes the proof
of Theorem \ref{t3.1}.
\end{proof}

As a corollary of Theorem \ref{3.1}, we immediately obtain that the
local vertical and the local nontangential maximal function
characterizations of $h_{\fai,\,N}(\rn)$ with $N\ge N_{\fai}$ as
follows. Here and in what follows,
\begin{eqnarray}\label{3.20}
N_{\fai}:=\max\lf\{\wz{N}_{\fai},\,N_0\r\},
\end{eqnarray}
where $\wz{N}_{\fai}$ and $N_0$ are  respectively as in Definition
\ref{d3.2} and Theorem \ref{t3.1}.

\begin{theorem}\label{t3.2}
Let $\fai$ be a growth function as in Definition \ref{d2.3}, $\pz_0$
and $N_{\fai}$ be respectively as in \eqref{3.3} and \eqref{3.20}.
Then for any integer $N\ge N_{\fai}$, the following are equivalent:
\begin{enumerate}
\vspace{-0.25cm}
\item[\rm(i)] $f\in h_{\fai,\,N}(\rn);$
\vspace{-0.25cm}
\item[\rm(ii)] $f\in\cd'(\rn)$ and $\pz^{+}_0 (f)\in L^{\fai}(\rn);$
\vspace{-0.25cm}
\item[\rm(iii)] $f\in\cd'(\rn)$ and $(\pz_0)^{\ast}_{\triangledown}
(f)\in L^{\fai}(\rn);$ \vspace{-0.25cm}
\item[\rm(iv)] $f\in\cd'(\rn)$ and $\wz{\cg}_N (f)\in L^{\fai}(\rn);$
\vspace{-0.25cm}
\item[\rm(v)] $f\in\cd'(\rn)$ and $\wz{\cg}^0_N (f)\in L^{\fai}(\rn);$
\vspace{-0.25cm}
\item[\rm(vi)] $f\in\cd'(\rn)$ and $\cg^0_N (f)\in L^{\fai}(\rn)$.
\vspace{-0.25cm}
\end{enumerate}

Moreover, for all $f\in h_{\fai,\,N}(\rn)$,
\begin{eqnarray*}
\|f\|_{h_{\fai,\,N}(\rn)}&\sim&\lf\|\pz^{+}_0
(f)\r\|_{L^{\fai}(\rn)}\sim
\lf\|(\pz_0)^{\ast}_{\triangledown}(f)\r\|_{L^{\fai}(\rn)}\\
\nonumber &\sim&\lf\|\wz{\cg}_N
(f)\r\|_{L^{\fai}(\rn)}\sim\lf\|\wz{\cg}^0_N
(f)\r\|_{L^{\fai}(\rn)}\sim\lf\|\cg^0_N (f)\r\|_{L^{\fai}(\rn)},
\end{eqnarray*}
where the implicit constants are independent of $f$.
\end{theorem}

\begin{proof}
Replacing \cite[Theorem 3.12]{yys} by Theorem \ref{t3.1} and
repeating the proof of \cite[Theorem 3.14]{yys}, we conclude the
proof of Theorem \ref{t3.2}.
\end{proof}

\begin{remark}\label{r3.3}
We point out that when $\fai$ is as in \eqref{1.1}, Theorems
\ref{t3.1} and \ref{t3.2} are respectively Theorems 3.12 and 3.14 in
\cite{yys}.
\end{remark}

As a corollary of Theorems \ref{t3.1} and \ref{t3.2}, we have the
following local tangential maximal function characterization of
$h_{\fai,\,N}(\rn)$. We omit the details.
\begin{corollary}\label{c3.1}
Let $\fai$ be a growth function as in Definition \ref{d2.3}, $\pz_0$
and $N_{\fai}$ as in \eqref{3.3} and \eqref{3.20}, and $A$ and $B$
as in Theorem \ref{t3.1}. Then for integer $N\ge N_{\fai}$, $f\in
h_{\fai,\,N}(\rn)$ if and only if $f\in\cd'(\rn)$ and
$\pz^{\ast\ast}_{0,\,A,\,B}(f)\in{L^{\fai}(\rn)}$; moreover,
$$\|f\|_{h_{\fai,\,N}(\rn)}\sim
\lf\|\pz^{\ast\ast}_{0,\,A,\,B}(f)\r\|_{L^{\fai}(\rn)}.$$
\end{corollary}

 Now, we give some basic properties concerning $h_{\fai,\,N}(\rn)$ and
$h^{\fai,\,q,\,s}(\rn)$.

\begin{proposition}\label{p3.2}
Let $\fai$ be a growth function as in Definition \ref{d2.3} and
$N_{\fai}$ as in \eqref{3.20}. If integer $N\ge N_{\fai}$, then the
inclusion $h_{\fai,\,N}(\rn)\hookrightarrow\cd'(\rn)$ is continuous.
\end{proposition}

\begin{proposition}\label{p3.3}
Let $\fai$ be a growth function as in Definition \ref{d2.3} and
$N_{\fai}$ as in \eqref{3.20}. If integer $N\ge N_{\fai}$, then
the space $h_{\fai,\,N}(\rn)$ is complete.
\end{proposition}

The proofs of Propositions \ref{p3.2} and \ref{p3.3} are similar
to those of \cite[Propositions 5.1 and 5.2]{k}, respectively. We
omit the details.

\begin{theorem}\label{t3.3}
Let $N_{\fai}$ be as in \eqref{3.20}. If $(\fai,\,q,\,s)$ is
admissible and integer $N\ge N_{\fai}$, then
$h^{\fai,\,q,\,s}(\rn)\subset h_{\fai,\,N_{\fai}}(\rn)\subset
h_{\fai,\,N}(\rn)$ and, moreover, there exists a positive constant
$C$ such that for all $f\in h^{\fai,\,q,\,s}(\rn)$,
$$\|f\|_{h_{\fai,\,N}(\rn)}\le\|f\|_{h_{\fai,\, N_{\fai}}(\rn)} \le
C\|f\|_{h^{\fai,\,q,\,s}(\rn)}.$$
\end{theorem}

To prove Theorem \ref{t3.3}, we need the following lemma.

\begin{lemma}\label{l3.4}
Let $(\fai,\,q,\,s)$ be an admissible triplet and $N\in\zz_+$ with
$N\ge s$. Then there exists a positive constant $C$ such that for
any multiple of $(\fai,\,q,\,s)$-atom or $(\fai,\,q)$-single-atom
$f$,
\begin{equation}\label{3.21}
\int_{\rn}\fai\lf(x,\cg_N^0(f)(x)\r)\,dx\le C
\fai\lf(Q,\|f\|_{L^q_{\fai}(Q)}\r),
\end{equation}
where $\supp(f)\subset Q$ and, when $f$ is a multiple of some
$(\fai,\,q)$-single-atom, $Q:=\rn$.
\end{lemma}

\begin{proof}
 The proof of the case $q=\fz$ is easy and we omit the details.
 We just consider the case that $q\in(q(\fai),\fz)$.
Let $f$ be a multiple of some $(\fai,\,q)$-single-atom  and
$f\neq0$. Then we know that $\fai(\rn,t)<\fz$ for all $t\in(0,\fz)$.
From the uniformly upper type 1 property of $\fai$, H\"older's
inequality, Lemma \ref{l2.4}(iii), together with the fact that
$\cg^0_N(f)\ls M^{\loc}(f)$ and Definition \ref{3.3}(ii), we deduce
that
\begin{eqnarray}\label{3.22}
&&\int_{\rn}\fai\lf(x,\cg^0_{N}(f)(x)\r)\,dx\nonumber\\
&&\hs\le\int_{\rn}\lf(1+\frac{\cg_N^0(f)(x)}
{\|f\|_{L^q_{\fai}(\rn)}}\r) \fai\lf(x,\|f\|_{L^q_{\fai}(\rn)}\r)\,dx\nonumber\\
&&\hs\le\fai\lf(\rn,\|f\|_{L^q_{\fai}(\rn)}\r)+
\frac{1}{\|f\|_{L^q_{\fai}(\rn)}}\lf\{\int_{\rn}
\lf|\cg^0_N(f)(x)\r|^q\fai\lf(x,\|f\|_{L^q_{\fai}(\rn)}\r)\,dx\r\}^{1/q}\noz\\
&&\hs\hs\times\lf[\fai(\rn,\|f\|_{L^q_{\fai}(\rn)})\r]^{(q-1)/q}\noz\\
&&\hs\ls\fai\lf(\rn,\|f\|_{L^q_{\fai}(\rn)}\r).
\end{eqnarray}
Thus, \eqref{3.21} holds in this case.

Now, let $f$ be a multiple of some $(\fai,\,q,\,s)$-atom with
$f\neq0$, and $\supp(f)\subset Q:=Q(x_0,r_0)$. We consider the
following two cases for $Q$.

{\it Case} 1) $|Q|<1$. In this case, letting $\wz{Q}:=2\sqrt{n}Q$,
then we know that
\begin{eqnarray}\label{3.23}
&&\int_{\rn}\fai\lf(x,\cg^0_N(f)(x)\r)\,dx
=\int_{\wz{Q}}\fai\lf(x,\cg^0_N(f)(x)\r)\,dx
+\int_{\wz{Q}^{\complement}} \cdots=:\mathrm{I_1}+\mathrm{I_2}.
\end{eqnarray}

For $\mathrm{I_1}$, similar to \eqref{3.22}, we see that
\begin{eqnarray}\label{3.24}
\mathrm{I_1}\ls\fai\lf(Q,\|f\|_{L^q_{\fai}(Q)}\r),
\end{eqnarray}
which is desired.

To estimate $\mathrm{I_2}$, we claim that for all
$x\in\wz{Q}^{\complement}$,
\begin{eqnarray}\label{3.25}
\hspace{3 em}\cg^0_N(f)(x)\ls\frac{|Q|^{\frac{s
+n+1}{n}}}{|x-x_0|^{s+n+1}}\|f\|_{L^q_{\fai}(Q)}
\chi_{B(x_0,2\sqrt{n})}(x).
\end{eqnarray}
Indeed, for any $\pz\in\cd^0_{N}(\rn)$ and $t\in(0,1)$, let $P$ be
the Taylor expansion of $\pz$ about $(x-x_0)/t$ with degree $s$. By
Taylor's remainder theorem, for any $y\in\rn$, we know that
$$\lf|\pz\lf(\frac{x-y}{t}\r)-P\lf(\frac{x-y}{t}\r)\r|\ls
\sum_{\gfz{\az\in\nn^n}{|\az|=s+1}}\lf|\lf(\partial^{\az}\pz\r)
\lf(\frac{\theta(x-y)+(1-\theta)(x-x_0)}{t}\r)\r|\lf|\frac{x_0
-y}{t}\r|^{s+1},$$
where $\theta\in(0,1)$. From $t\in(0,1)$ and
$x\in\wz{Q}^{\complement}$, we deduce that $\supp
(f\ast\pz_t)\subset B(x_0,2\sqrt{n})$ and that $f\ast\pz_t
(x)\neq0$ implies that $t>\frac{|x-x_0|}{2}$. Thus, by these
observations, Definition \ref{d3.4}(iii), \eqref{2.7}, H\"older's
inequality and Definition \ref{d3.3}(i), we see that
\begin{eqnarray*}
|f\ast\pz_t (x)|&=&\frac{1}{t^n}\lf|\int_{Q}f(y)
\lf[\pz\lf(\frac{x-y}{t}\r)-P\lf(\frac{x-y}{t}\r)\r]\,dy\r|
\chi_{B(x_0,2\sqrt{n})}(x)\\
&\ls&|x-x_0|^{-(s +n+1)}\lf\{\int_{Q}|f(y)||x_0 -y|^{s
+1}\,dy\r\} \chi_{B(x_0,2\sqrt{n})}(x)\\
&\ls&|Q|^{\frac{s+1}{n}}\lf\{\int_Q|f(y)|^q\fai(x,\lz)\,dy\r\}^{1/q}
\lf\{\int_Q[\fai(y,\lz)]^{-1/(q-1)}\,dy\r\}^{(q-1)/q}\\
&&\times|x-x_0|^{-(s+n+1)}\chi_{B(x_0,2\sqrt{n})}(x)\\
&\ls&\frac{|Q|^{\frac{s
+n+1}{n}}}{|x-x_0|^{s+n+1}}\|f\|_{L^q_{\fai}(Q)}
\chi_{B(x_0,2\sqrt{n})}(x),
\end{eqnarray*}
where $\lz\in(0,\fz)$, which, together with the arbitrariness of
$\pz\in\cd^0_{N}(\rn)$, implies \eqref{3.25}. Thus, the claim holds.

Let $Q_k:=2^k\sqrt{n}Q$ for all $k\in\zz_+$, and $k_0 \in\zz_+$
satisfy $2^{k_0}r\le4<2^{k_0 +1}r$. By $s=\lfz
n[\frac{q(\fai)}{i(\fai)}-1]\rfz$, we know that there exist
$q_0\in(q(\fai),\fz)$ and $p_0\in(0,i(\fai))$ such that $\fai$ is of
uniformly lower type $p_0$ and $p_0(s+n+1)>nq_0$. From Lemma
\ref{l2.4}(i) and Remark \ref{r2.3}, it follows that there exists a
$\wz{\fai}\in \aa_{q_0}(\rn)$ such that
$\fai(\cdot,t)=\wz{\fai}(\cdot,t)$ on $Q(x_0,8\sqrt{n})$ for all
$t\in[0,\fz)$. By this, \eqref{3.25}, the uniformly lower type $p_0$
property of $\fai$ and Lemma \ref{l2.4}(ii), we know that
\begin{eqnarray*}
\mathrm{I_2}&\le&\int_{\sqrt{n}r\le|x-x_0|<2\sqrt{n}}
\fai\lf(x,\cg^0_N(f)(x)\r)\,dx\\
&\ls&\int_{\sqrt{n}r\le|x-x_0|<2\sqrt{n}}\wz{\fai}\lf(x,
\frac{|Q|^{\frac{s
+n+1}{n}}}{|x-x_0|^{s+n+1}}\|f\|_{L^q_{\fai}(Q)}\r)\,dx\\
&\ls&\sum_{k=1}^{k_0}\int_{Q_{k+1}\setminus
Q_{k}}\wz\fai\lf(x,2^{-k(s +n+1)}\|f\|_{L^q_{\fai}(Q)}\r)\,dx\\
&\ls&\sum_{k=1}^{k_0}
2^{-k(s++n+1)p_0}\wz{\fai}\lf(Q_{k+1},\|f\|_{L^q_{\fai}(Q)}\r)\\
&\ls&\sum_{k=1}^{k_0} 2^{-k[(s+n+1)p_0-n
q_0]}\fai\lf(Q,\|f\|_{L^q_{\fai}(Q)}\r)\ls\fai\lf(Q,\|f\|_{L^q_{\fai}(Q)}\r),
\end{eqnarray*}
which, together with \eqref{3.24} and \eqref{3.23}, implies
\eqref{3.21} in Case 1).

{\it Case} 2) $|Q|\ge1$. In this case, let $Q^{\ast}:= Q(x_0,
r+2)$. Thus, from $\supp(\cg^0_N(f))\subset Q^{\ast}$, the
uniformly upper type 1 property of $\fai$, H\"older's inequality,
Proposition \ref{p3.1}(ii) and Lemma \ref{l2.4}(ii), we deduce
that
\begin{eqnarray*}
&&\int_{\rn}\fai\lf(x,\cg^0_N(f)(x)\r)\,dx\\
&&\hs=\int_{Q^{\ast}}\fai\lf(x,\cg^0_N(f)(x)\r)\,dx
\le\int_{Q^{\ast}}\lf(1+\frac{\cg_N^0(f)(x)}
{\|f\|_{L^q_{\fai}(Q)}}\r)\fai\lf(x,\|f\|_{L^q_{\fai}(Q)}\r)\,dx\\
&&\hs\le\fai\lf(Q^{\ast},\|f\|_{L^q_{\fai}(Q)}\r)+
\frac{1}{\|f\|_{L^q_{\fai}(Q)}}\lf\{\int_{Q^{\ast}}
|\cg^0_N(f)(x)|^q\fai\lf(x,\|f\|_{L^q_{\fai}(Q)}\r)\,dx\r\}^{1/q}\\
\nonumber &&\hs\hs\times
\lf[\fai\lf(Q^{\ast},\|f\|_{L^q_{\fai}(Q)}\r)\r]^{(q-1)/q}
\ls\fai\lf(Q^{\ast},\|f\|_{L^q_{\fai}(Q)}\r)\ls
\fai\lf(Q,\|f\|_{L^q_{\fai}(Q)}\r).
\end{eqnarray*}
which proves \eqref{3.21} in Case 2) and hence completes the proof
of Lemma \ref{l3.4}.
\end{proof}

Now, we show Theorem \ref{t3.3} by using Lemma \ref{l3.4}.

\begin{proof}[Proof of Theorem \ref{t3.3}]
Obviously, by Definition \ref{d3.2}, we only need to prove that
$$h^{\fai,\,q,\,s}(\rn)\subset h_{\fai,\,N_{\fai}}(\rn)$$
and, for all $f\in h^{\fai,\,q,\,s}(\rn)$,
$\|f\|_{h_{\fai,\,N_{\fai}}(\rn)} \ls
\|f\|_{h^{\fai,\,q,\,s}(\rn)}$. Indeed, for any $f\in
h^{\fai,\,q,\,s}(\rn)$, $f=\sum_{i=0}^{\fz}f_i$ in $\cd'(\rn)$,
where $f_0$ is a multiple of some $(\fai,\,q)$-single-atom and for
any $i\in\zz_+$, $f_i$ is a multiple of some $(\fai,\,q,\,s)$-atom
supported in the cube $Q_i$. It is easy to see that
$\cg^0_{N_{\fai}}(f)\le\sum_{i=0}^{\fz}\cg^0_{N_{\fai}}(f_i)$. From
this, Lemma \ref{2.1}(i) and Lemma \ref{l3.4}, we deduce that
\begin{eqnarray*}
\int_{\rn}\fai\lf(x,\frac{\cg^0_{N_{\fai}}
(f)(x)}{\blz_q(\{f_i\}_{i=0}^{\fz})}\r)\,dx
&&\ls\sum_{i=0}^{\fz}\int_{\rn}\fai\lf(x,\frac{\cg^0_{N_{\fai}}
(f_i)(x)}{\blz_q(\{f_i\}_{i=0}^{\fz})}\r)\,dx\\
&&\ls\fai\lf(\rn,\frac
{\|f_0\|_{L^q_{\fai}(\rn)}}{\blz_q(\{f_i\}_{i=0}^{\fz})}\r)+\sum_{i=1}^{\fz}
\fai\lf(Q_i,\frac
{\|f_i\|_{L^q_{\fai}(Q_i)}}{\blz_q(\{f_i\}_{i=0}^{\fz})}\r)\ls1,
\end{eqnarray*}
which, together with Theorem \ref{t3.2}, implies that
$\|f\|_{h_{\fai,\,N_{\fai}}(\rn)} \ls
\|f\|_{h^{\fai,\,q,\,s}(\rn)}$. This finishes the proof of Theorem
\ref{t3.3}.
\end{proof}

\section{Calder\'on-Zygmund decompositions\label{s4}}

\hskip\parindent In this section, we recall some subtle estimates
for the Calder\'on-Zygmund decomposition associated with local
grand maximal functions from \cite{ta1}. Notice that the
construction of the Calder\'on-Zygmund decomposition in \cite{ta1}
is similar to those in \cite{b03,blyz08,st93}.

Let $\fai$ be a growth function as in Definition \ref{d2.3} and
$q(\fai)$ as in \eqref{2.9}. Let integer $N\ge 2$, $\cg_N (f)$ and
$\cg_N^0 (f)$ be as in \eqref{3.2}.

\emph{Throughout this section}, let $f\in\cd'(\rn)$ satisfy the fact that for
all $\lz\in(0,\fz)$ and $t\in[0,\fz)$,
$$\int_{\{x\in\rn:\ \cg_N
(f)(x)>\lz\}}\fai(x,t)\,dx<\fz.$$ For a given
$\lz>\inf_{x\in\rn}\cg_{N} (f)(x)$, set
\begin{eqnarray}\label{4.1}
\boz_{\lz}:=\{x\in\rn:\ \cg_{N} (f)(x)>\lz\}.
\end{eqnarray}
It is obvious that $\boz_{\lz}$ is a proper open subset of $\rn$. By
the Whitney decomposition of $\boz_{\lz}$ given in \cite{ta1} (see
also \cite{b03,blyz08,st93}), we have closed cubes $\{Q_i\}_i$ such
that
\begin{eqnarray}\label{4.2}
\boz_{\lz}=\bigcup_{i}Q_i,
\end{eqnarray}
their interiors are away from $\boz_{\lz}^\complement$ and
$\diam(Q_i)\le2^{-(6+n)}\dist(Q_i,\boz_{\lz}^\complement)
\le4\diam(Q_i).$ In what follows, fix $a:=1+2^{-(11+n)}$ and
denote $aQ_i$ by $Q^{\ast}_i$ for all $i$. Then we have
$Q_i\subset Q^{\ast}_i$. Moveover,
$\boz_{\lz}=\cup_{i}Q^{\ast}_i$, and $\{Q^{\ast}_i\}_i$ have the
\emph{bounded interior property}, namely, every point in
$\boz_{\lz}$ is contained in at most a fixed number of
$\{Q^{\ast}_i\}_i$.

Now, we take a function $\xz\in \cd(\rn)$ such that $0\le\xz\le1$,
$\supp\xz\subset aQ(0,1)$ and $\xz\equiv1$ on $Q(0,1)$. For
$x\in\rn$, set $\xz_i (x):=\xz((x-x_i)/l_i)$, here and in what follows,
$x_i$ is the {\it center} of the cube $Q_i$ and $l_i$ its
{\it sidelength}. Obviously, by the constructions of
$\{Q_i^{\ast}\}_i$ and $\{\xz_i\}_i$, for any $x\in\rn$, we have
$1\le\sum_{i}\xz_i (x)\le L$, where $L$ is a fixed positive
integer independent of $x$. Let
\begin{eqnarray}\label{4.3}
\zez_i:=\xz_i\lf[\sum_j\xz_j\r]^{-1}.
\end{eqnarray}
Then $\{\zez_i\}_i$ forms a smooth partition of unity for
$\boz_{\lz}$ subordinate to the locally finite cover
$\{Q_i^{\ast}\}_i$ of $\boz_{\lz}$, namely,
$\chi_{\boz_{\lz}}=\sum_i \zez_i$ with each $\zez_i\in \cd(\rn)$
supported in $Q_i^{\ast}$.

Let $s\in\zz_{+}$ be some fixed integer and $\cp_s (\rn)$ denote the
\emph{linear space of polynomials in $n$ variables of degrees no
more than $s$}. For each $i\in\zz_+$ and $P\in\cp_s (\rn)$, set
\begin{eqnarray}\label{4.4}
\|P\|_i:=\lf\{\lf[\int_{\rn}\zez_i
(y)\,dy\r]^{-1}\int_{\rn}|P(x)|^2\zez_i (x)\,dx\r\}^{1/2}.
\end{eqnarray}
Then it is easy to know that $(\cp_s (\rn),\,\|\cdot\|_i)$ is a
finite dimensional Hilbert space. Let $f\in\cd'(\rn)$. Since $f$
induces a linear functional on $\cp_s (\rn)$ via
$P\mapsto[\int_{\rn}\zez_i (y)\,dy]^{-1}\langle f,
P\zez_i\rangle$, by the Riesz representation theorem, there exists a
unique polynomial
\begin{eqnarray}\label{4.5}
P_i \in\cp_s (\rn)
\end{eqnarray}
 for each $i$ such that for
all $P\in\cp_s (\rn)$, $\langle f, P\zez_i\rangle=\langle P_i,
P\zez_i\rangle$. For each $i$, define the distribution
\begin{eqnarray}\label{4.6}
b_i:=(f-P_i)\zez_i\,\, \text{when}\,\,l_i\in(0,1),\ \text{and} \
b_i:= f\zez_i \,\,\text{when}\,\,l_i\in[1,\fz).
\end{eqnarray}
We show that for suitable choices of $s$ and $N$, the series $\sum_i
b_i$ converges in $\cd'(\rn)$ and, in this case, we let $g:=
f-\sum_i b_i$ in $\cd'(\rn)$. We point out that the representation
\begin{eqnarray}\label{4.7}
f=g+\sum_i b_i,
\end{eqnarray}
where $g$ and $b_i$ are as above, is called a {\it
Calder\'on-Zygmund decomposition} of $f$ of degree $s$ and height
$\lz$ associated with $\cg_N (f)$.

The remainder of this section consists of a series of lemmas. Lemma
\ref{l4.1} gives a property of the smooth partition of unity
$\{\zez_i\}_i$, Lemmas \ref{l4.2} through \ref{l4.5} are devoted to
some estimates for the bad parts $\{b_i\}_i$, and Lemmas \ref{l4.6}
and \ref{l4.7} give some controls over the good part $g$.
Finally, Corollary \ref{c4.1} shows that the density of
$L^{q}_{\fai(\cdot,1)}(\rn)\cap h_{\fai,\,N}(\rn)$ in
$h_{\fai,\,N}(\rn)$, where $q\in(q(\fai),\fz)$. The following Lemmas
\ref{l4.1} through \ref{l4.3} are respectively Lemmas 4.2 through
4.5 in \cite{ta1}.
\begin{lemma}\label{l4.1}
There exists a positive constant $C_1$ such that for all
$f\in\cd'(\rn)$, $\lz>\inf_{x\in\rn}\cg_{N} (f)(x)$ and
$l_i\in(0,1)$, $\sup_{y\in\rn}|P_i (y)\zez_i (y)|\le C_1\lz$.
\end{lemma}

\begin{lemma}\label{l4.2}
There exists a positive constant $C_2$ such that for all $i\in\zz_+$
and $x\in Q_i^{\ast}$,
\begin{eqnarray}\label{4.8}
\cg^0_N (b_i)(x)\le C_2\cg_N (f)(x).
\end{eqnarray}
\end{lemma}

\begin{lemma}\label{l4.3}
Assume that integers $s$ and $N$ satisfy $0\le s<N$ and $N\ge2$.
Then there exist positive constants $C$, $C_3$ and $C_4$ such that
for all $i\in\zz_+$ and $x\in(Q_i^{\ast})^\complement$,
\begin{eqnarray}\label{4.9}
\cg^0_N (b_i)(x)\le C\frac{\lz l_i^{n+s+1}}{(l_i
+|x-x_i|)^{n+s+1}}\chi_{B(x_i, C_3)}(x),
\end{eqnarray}
where $x_i$ is the center of the cube $Q_i$. Moreover, if
$x\in(Q_i^{\ast})^\complement$ and $l_i\in[C_4,\fz)$, then $\cg^0_N
(b_i)(x)=0$.
\end{lemma}

\begin{lemma}\label{l4.4}
Let $\fai$ be a growth function as in Definition \ref{d2.3} and
$m(\fai)$ as in \eqref{2.10}. If integers $s\ge m(\fai)$, $N>s$
and $N\ge N_{\fai}$, where $N_{\fai}$ is as in \eqref{3.20}, then
there exists a positive constant $C_5$ such that for all $f\in
h_{\fai,\,N}(\rn)$, $\lz>\inf_{x\in\rn}\cg_{N} (f)(x)$ and
$i\in\zz_+$,
\begin{eqnarray}\label{4.10}
\int_{\rn}\fai\lf(x,\cg^0_N (b_i)(x)\r)\,dx\le C_5
\int_{Q_i^{\ast}}\fai\lf(x,\cg_N (f)(x)\r)\,dx.
\end{eqnarray}
Moreover, the series $\sum_i b_i$ converges in $h_{\fai,\,N}(\rn)$
and
\begin{eqnarray}\label{4.11}
\int_{\rn}\fai\lf(x,\cg^0_N \lf(\sum_i b_i\r)(x)\r)\,dx\le C_5
\int_{\boz_{\lz}}\fai\lf(x,\cg_N (f)(x)\r)\,dx.
\end{eqnarray}
\end{lemma}

\begin{proof}
By Lemmas \ref{l4.2} and \ref{l4.3} and the uniformly upper type 1
property of $\fai$, we know that
\begin{eqnarray}\label{4.12}
\int_{\rn}\fai\lf(x,\cg^0_N (b_i)(x)\r)\,dx\ls
\int_{Q_i^{\ast}}\fai\lf(x,\cg_N (f)(x)\r)\,dx
+\int_{(2C_3 Q_i^0)\setminus
Q_i^{\ast}}\fai\lf(x,\cg^0_N (b_i)(x)\r)\,dx,
\end{eqnarray}
where $Q_i^0:= Q(x_i,1)$. Notice that $s\ge\lfz
n(q(\fai)/i(\fai)-1)\rfz$ implies $(s+n+1)i(\fai)>nq(\fai)$. Thus,
we take $q_0\in(q(\fai),\fz)$ and $p_0\in(0,i(\fai))$ such that
$\fai$ is of uniformly lower type $p_0$, $(s+n+1)p_0>nq_0$ and
$\fai\in \aa^{\loc}_{q_0}(\rn)$. From Lemma \ref{l2.4}(i), it
follows that there exists a $\wz{\fai}\in \aa_{q_0}(\rn)$ such that
$\wz{\fai}(\cdot,t)=\fai(\cdot,t)$ on $2C_3 Q^0_i$ and $\aa_{q_0}
(\wz{\fai})\ls \aa^{\loc}_{q_0}(\fai)$. Using Lemma \ref{4.3}, the
uniformly lower $p_0$ property of $\fai$, Lemma \ref{l2.4}(iv) and
the fact that $\cg_N (f)>\lz$ for all $x\in Q^{\ast}_i$, we see that
\begin{eqnarray}\label{4.13}
\hs&&\int_{(2C_3 Q_i^0)\setminus Q_i^{\ast}}\fai\lf(x,\cg^0_N
(b_i)(x)\r)\,dx\nonumber\\  &&\hs\le\sum_{k=1}^{k_0}\int_{2^k
Q_i^{\ast}\setminus 2^{k-1}Q_i^{\ast}}\fai\lf(x,\cg^0_N
(b_i)(x)\r)\,dx \ls\sum_{k=1}^{k_0}\int_{2^k
Q_i^{\ast}}\wz{\fai}\lf(x,\frac{\lz}{2^{k(n+s+1)}}\r)\,dx\nonumber\\
&&\hs\ls\sum_{k=1}^{k_0} 2^{-k[(n+s+1)p_0-nq_0]}
\wz{\fai}(Q_i^{\ast},\lz)\ls\sum_{k=1}^{k_0}
2^{-k[(n+s+1)p_0-nq_0]} \fai(Q_i^{\ast},\lz)\nonumber\\
&&\hs\ls\int_{Q_i^{\ast}}\fai\lf(x,\lz\r)\,dx
\ls\int_{Q_i^{\ast}}\fai\lf(x,\cg_N (f)(x)\r)\,dx,
\end{eqnarray}
where $k_0\in\zz_+$ satisfies $2^{k_0 -2}\le C_3<2^{k_0-1}$. From
\eqref{4.12} and \eqref{4.13}, we deduce \eqref{4.10}. Then, by
\eqref{4.10}, we know that
\begin{eqnarray*}
\sum_i\int_{\rn}\fai\lf(x,\cg^0_N
(b_i)(x)\r)\,dx&\ls&\sum_i\int_{Q_i^{\ast}}\fai\lf(x,\cg_N
(f)(x)\r)\,dx\ls\int_{\boz_{\lz}}\fai\lf(x,\cg_N (f)(x)\r)\,dx.
\end{eqnarray*}
Combining the above inequality with the completeness of
$h_{\fai,\,N}(\rn)$, we conclude that $\sum_i b_i$ converges in
$h_{\fai,\,N}(\rn)$. So by Proposition \ref{p3.2}, the series
$\sum_i b_i$ converges in $\cd'(\rn)$ and hence
$$\cg^0_N\lf(\sum_i b_i\r)(x)\le \sum_i \cg^0_N (b_i)(x)$$
for all $x\in\rn$, which gives \eqref{4.11} and hence completes
the proof of Lemma \ref{4.4}.
\end{proof}

The following Lemmas \ref{l4.5} and \ref{l4.6} are respectively
\cite[Lemmas 4.7 and 4.8]{ta1}.

\begin{lemma}\label{l4.5}
Let $q\in(q(\fai),\fz)$. If $f\in L^q_{\fai(\cdot,1)}(\rn)$, then
the series $\sum_i b_i$ converges in $L^q_{\fai(\cdot,1)}(\rn)$ and
there exists a positive constant $C_6$, independent of $f$, such
that
$$\lf\|\sum_i |b_i|\r\|_{L^q_{\fai(\cdot,1)}(\rn)}\le C_6
\|f\|_{L^q_{\fai(\cdot,1)}(\rn)}.$$
\end{lemma}

\begin{lemma}\label{l4.6}
Let integers $s$ and $N$ satisfy $0\le s<N$ and $N\ge2$,
$f\in\cd'(\rn)$ and $\lz>\inf_{x\in\rn}\cg_{N} (f)(x)$. If $\sum_i
b_i$ converges in $\cd'(\rn)$, then there exists a positive constant
$C_7$, independent of $f$ and $\lz$, such that for all $x\in\rn$,
$$\cg^0_N (g)(x)\le\cg^0_N (f)(x)\chi_{\boz_{\lz}^\complement}(x)+C_7\lz\sum_{i}
\frac{l_i^{n+s+1}}{(l_i
+|x-x_i|)^{n+s+1}}\chi_{B(x_i,C_3)}(x),$$where $x_i$ is the center
of $Q_i$ and $C_3$ as in Lemma \ref{l4.3}.
\end{lemma}

\begin{lemma}\label{l4.7}
Let $\fai$ be a growth function as in Definition \ref{d2.3},
$q(\fai)$ and $i(\fai)$ respectively as in \eqref{2.9} and
\eqref{2.3}, $q\in(q(\fai),\fz)$ and $p_0\in(0,i(\fai))$.

$\mathrm{(i)}$ If integers $s$ and $N$ satisfy $N>s\ge\lfz
n(q(\fai)/p_0-1)\rfz$ and $f\in h_{\fai,\,N}(\rn)$, then $\cg^0_N
(g)\in L^q_{\fai(\cdot,1)}(\rn)$ and there exists a positive
constant $C_8$, independent of $f$ and $\lz$, such that
\begin{eqnarray}\label{4.14}
\hs\int_{\rn}\lf[\cg^0_N (g)(x)\r]^q\fai(x,1)\,dx\le C_8
\lz^q\max\{1/\lz,1/\lz^{p_0}\}\int_{\rn}\fai\lf(x,\cg_N
(f)(x)\r)\,dx.
\end{eqnarray}

$\mathrm{(ii)}$  If $f\in L^q_{\fai(\cdot,1)}(\rn)$, then $g\in
L^{\fz}_{\fai(\cdot,1)}(\rn)$ and there exists a positive constant
$C_9$, independent of $f$ and $\lz$, such that
$\|g\|_{L^{\fz}_{\fai(\cdot,1)}(\rn)}\le C_9 \lz$.
\end{lemma}

The proof of Lemma \ref{l4.7} is similar to those of \cite[Lemma
5.3]{k} and \cite[Lemma 4.7]{yys}. We omit the details. Moreover,
from Lemma \ref{l4.7}, we deduce the following corollary, whose
proof is similar to that of \cite[Proposition 5.3]{k}. We omit the details.

\begin{corollary}\label{c4.1}
Let $\fai$ be a growth function as in Definition \ref{d2.3},
$q(\fai)$ as in \eqref{2.9}, $q\in(q(\fai),\fz)$ and integer $N\ge
N_{\fai}$, where $N_{\fai}$ is as in \eqref{3.20}. Then
$h_{\fai,\,N}(\rn)\cap L^q_{\fai(\cdot,1)} (\rn)$ is dense in
$h_{\fai,\,N}(\rn)$.
\end{corollary}

\section{Weighted atomic characterizations of $h
_{\fai,\,N}(\rn)$\label{s5}}

\hskip\parindent In this section, we establish the equivalence
between $h_{\fai,\,N}(\rn)$ and $h^{\fai,\,q,\,s}(\rn)$ by using the
Calder\'on-Zygmund decomposition associated with the local grand
maximal function stated in Section \ref{s4}.

Let $\fai$ be a growth function, $q(\fai)$, $i(\fai)$ and
$N_{\fai}$ respectively as in \eqref{2.9}, \eqref{2.3} and
\eqref{3.20}, integer $N\ge N_{\fai}$ and $s_0:=\lfz
n[q(\fai)/i(\fai)-1]\rfz$. {\it Throughout this section, let $f\in
h_{\fai,\,N}(\rn)$.} We take $k_0\in\zz$ such that
$$2^{k_0-1}<\inf_{x\in\rn}\cg_N (f)(x)\le2^{k_0}$$
when $\inf_{x\in\rn}\cg_N (f)(x)>0$ and, when $\inf_{x\in\rn}\cg_N
(f)(x)=0$, let $k_0 :=-\fz$. {\it Throughout the whole section, we
always assume that $k\ge k_0$.} For each integer $k\ge k_0$,
consider the Calder\'on-Zygmund decomposition of $f$ of degree $s$
and height $\lz=2^k$ associated with $\cg_N (f)$. Namely, for any
$k\ge k_0$, by taking $\lz:=2^k$ in \eqref{4.1}, we now write the
Calder\'on-Zygmund decomposition in \eqref{4.7} by
\begin{eqnarray}\label{5.1}
f=g^k +\sum_{i}b_i^k,
\end{eqnarray}
where and in the remainder of this section, we write $\{Q_i\}_i$ in
\eqref{4.2}, $\{\zez_i\}_i$ in \eqref{4.3}, $\{P_i\}_i$ in
\eqref{4.5} and $\{b_i\}_i$ in \eqref{4.6}, respectively, as
$\{Q^k_i\}_i$, $\{\zez^k_i\}_i$, $\{P^k_i\}_i$ and $\{b^k_i\}_i$.
Now, the {\it center} and the {\it sidelength} of $Q^k_i$ are
respectively denoted by $x^k_i$ and $l^k_i$. Recall that for all $i$
and $k$,
\begin{eqnarray}\label{5.2}
\sum_i \zez^k_i =\chi_{\boz_{2^k}},\ \supp(b^k_i)\subset\supp
(\zez^k_i)\subset Q^{k\ast}_i,
\end{eqnarray}
$\{Q^{k\ast}_i\}_i$ has the bounded interior property and, for all
$P\in\cp_s (\rn)$,
\begin{eqnarray}\label{5.3}
\langle f, P\zez^k_i\rangle=\langle P^k_i,P\zez^k_i\rangle.
\end{eqnarray}
Recall that $a:=1+2^{-(11+n)}$ and $Q_i^{k\ast}:=aQ_i^{k}$.

For each integer $k\ge k_0$ and $i,\,j\in\zz_+$, let
$P^{k+1}_{i,\,j}$ be the \emph{orthogonal projection} of
$(f-P^{k+1}_j)\zez^k_i$ on $\cp_s (\rn)$ with respect to the norm
$$\|P\|_j^2:=\lf[\int_{\rn}\zez^{k+1}_j (y)\,dy\r]^{-1}
\int_{\rn}|P(x)|^2 \zez^{k+1}_j (x)\,dx,$$ namely, $P^{k+1}_{i,\,j}$
is the \emph{unique polynomial} of $\cp_s (\rn)$ such that for any
$P\in\cp_s (\rn)$,
\begin{eqnarray}\label{5.4}
\langle(f-P^{k+1}_j)\zez^k_i,P\zez^{k+1}_j\rangle=
\int_{\rn}P^{k+1}_{i,\,j} (x)P(x)\zez^{k+1}_j (x)\,dx.
\end{eqnarray}
 In what follows, let $E^k_1:=\{i\in\zz_+:\ |Q^k_i|\ge1/(2^4 n)\}$,
$E^k_2:=\{i\in\zz_+:\ |Q^k_i|<1/(2^4 n)\}$,
$$F^k_1:=\lf\{i\in\zz_+:\ |Q^k_i|\ge1\r\}$$
and $F^k_2:=\{i\in\zz_+:\ |Q^k_i|<1\}$.

Observe that
\begin{eqnarray}\label{5.5}
P^{k+1}_{i,\,j}\neq0 \,\,\text{if and only if}\,\,Q_i^{k\ast}\cap
Q_j^{(k+1)\ast}\neq\emptyset.
\end{eqnarray} Indeed, this follows directly from
the definition of $P^{k+1}_{i,\,j}$. The following Lemmas
\ref{l5.1} through \ref{l5.3} are just \cite[Lemmas 5.1 through
5.3]{ta1}.

\begin{lemma}\label{l5.1}
Let $\boz_{2^k}$ be as in \eqref{4.1} with $\lz=2^k$,
$Q_i^{k\ast}$ and $l^k_i$ as above.

$\mathrm{(i)}$ If $Q_i^{k\ast}\cap Q_j^{(k+1)\ast}\neq\emptyset$,
then $l^{k+1}_j\le2^4\sqrt{n}l^k_i$ and $ Q_j^{(k+1)\ast}\subset2^6
nQ_i^{k\ast}\subset\boz_{2^k}$.

$\mathrm{(ii)}$ There exists a positive integer $L$ such that for
each $i\in\zz_+$, the cardinality of
$$\lf\{j\in\zz_+:\ Q_i^{k\ast}\cap Q_j^{(k+1)\ast}\neq\emptyset\r\}$$
is bounded by $L$.
\end{lemma}

\begin{lemma}\label{l5.2}
 There exists a positive constant $C$, independent
of $f$, such that for all $i,\,j\in\zz_+$ and integer $k\ge k_0$
with $l^{k+1}_j \in(0,1)$,
\begin{eqnarray}\label{5.6}
\sup_{y\in\rn}\lf|P^{k+1}_{i,\,j}(y)\zez^{k+1}_j (y)\r|\le C2^{k+1}.
\end{eqnarray}
\end{lemma}

\begin{lemma}\label{l5.3}
For any $k\in\zz$ with $k\ge k_0$, 
$$\sum_{i\in\zz_+}\lf(\sum_{j\in
F^{k+1}_2}P^{k+1}_{i,\,j}\zez^{k+1}_j\r)=0,$$ 
where the series
converges both in $\cd'(\rn)$ and pointwise.
\end{lemma}

\begin{lemma}\label{l5.4}
Let $f\in h_{\fai,\,N}(\rn)$, $k_0=-\fz$ and $\boz_k:=\{x\in\rn:\
\cg_N(f)(x)>2^k\}$. Then there exists a positive constant $C$ such
that for all $\lz>\inf_{x\in\rn}\cg_N(f)(x)$,
$$\sum_{k=-\fz}^{\fz}\fai\lf(\boz_k,\frac{2^k}{\lz}\r)\le
C\int_{\rn}\fai\lf(x,\frac{\cg_N(f)(x)}{\lz}\r)\,dx.$$
\end{lemma}

The proof of Lemma \ref{l5.4} is similar to that of \cite[Lemma
5.4]{k}. We omit the details.

\begin{lemma}\label{l5.5}
Let $\fai$ be a growth function as in Definition \ref{d2.3},
$q(\fai)$ and $N_{\fai}$ respectively as in \eqref{2.9} and
\eqref{3.20}, integers $N$ and $s$ satisfy $N\ge s\ge N_{\fai}$,
and $q\in(q(\fai),\fz)$. Then
$$\lf[h_{\fai,\,N}(\rn)\cap L^q_{\fai(\cdot,1)}(\rn)\r]\subset h^{\fai,\,\fz,\,s}(\rn)$$
and the inclusion is continuous.
\end{lemma}

\begin{proof}
Let $f\in (L^q_{\fai(\cdot,1)}(\rn)\cap h_{\fai,\,N}(\rn))$. We
first consider the case that $k_0 =-\fz$. As above, for each
$k\in\zz$, $f$ has a Calder\'on-Zygmund decomposition of degree $s$
and height $\lz=2^k$ associated with $\cg_N (f)$ as in \eqref{5.1},
namely,
 $f=g^k +\sum_{i}b_i^k$. By Corollary \ref{c4.1} and
 Proposition \ref{3.2}, we have that $g^k \to f$ in both
 $h_{\fai,\,N}(\rn)$ and $\cd'(\rn)$ as $k\to\fz$. By
 Lemma \ref{l4.7}(i), $\|g^k\|_{L^q_{\fai(\cdot,1)}(\rn)}\to0$ as
 $k\to-\fz$ and, furthermore, by Lemma \ref{l2.5}(ii),
 $g^k\to0$ in $\cd'(\rn)$ as $k\to-\fz$. Therefore,
 \begin{eqnarray}\label{5.7}
 f=\sum_{k=-\fz}^{\fz}\lf(g^{k+1}-g^k\r)
 \end{eqnarray}
 in $\cd'(\rn)$. Moreover,
 since $\supp(\sum_i b^k_i)\subset\boz_{2^k}$ and $\fai(\boz_{2^k},1)\to 0$ as
 $k\to\fz$, then $g^k\to f$ almost everywhere as
 $k\to\fz$. Thus, \eqref{5.7} also holds almost everywhere.
Similar to \cite[(5.10)]{yys}, from Lemma \ref{l5.3} and
\eqref{5.2} with $\boz_{2^{k+1}}\subset
 \boz_{2^k}$, we deduce that
\begin{eqnarray}\label{5.8}
g^{k+1}-g^k=\sum_i
\lf(b_i^k-\sum_j b^{k+1}_j \zez^k_i +\sum_{j\in
F^{k+1}_2}P^{k+1}_{i,\,j}\zez^{k+1}_j\r)=:\sum_i h^k_i,
\end{eqnarray}
where all the series converge in both $\cd'(\rn)$ and almost
everywhere. Furthermore, by the definitions of $b_j^k$ and
$b_j^{k+1}$ as in \eqref{4.6}, we know that when $l^k_i\in(0,1)$,
\begin{eqnarray}\label{5.9}
h^k_i=f\chi_{\boz_{2^{k+1}}^\complement}\zez^k_i-P^k_i \zez^k_i
+\sum_{j\in F^{k+1}_2} P^{k+1}_j \zez^k_i \zez^{k+1}_j+\sum_{j\in
F^{k+1}_2}P^{k+1}_{i,\,j}\zez^{k+1}_j
\end{eqnarray}
and, when $l^k_i\in[1,\fz)$,
\begin{eqnarray}\label{5.10}
h^k_i=f\chi_{\boz_{2^{k+1}}^\complement}\zez^k_i +\sum_{j\in
F^{k+1}_2} P^{k+1}_j \zez^k_i \zez^{k+1}_j+\sum_{j\in
F^{k+1}_2}P^{k+1}_{i,\,j}\zez^{k+1}_j.
\end{eqnarray}
From Proposition \ref{p3.1}(i), we infer that  for almost every
$x\in\boz_{2^{k+1}}^\complement$, $|f(x)|\le\cg_N
(f)(x)\le2^{k+1}$, which, together with Lemmas \ref{l4.1} and
\ref{l5.1}(ii), \eqref{5.5}, Lemma \ref{5.2}, \eqref{5.9} and
\eqref{5.10}, implies that there exists a positive constant
$C_{10}$ such that for all $i\in\zz_+$,
\begin{eqnarray}\label{5.11}
\|h^k_i\|_{L^{\fz}_{\fai}(\rn)}\le C_{10} 2^k.
\end{eqnarray}
 Now, we show that for each $i$ and $k$, $h^k_i$ is a multiple of some
$(\fai,\,\fz,\,s)$-atom by considering the following two cases for
$i$.

{\it Case} 1) $i\in E^k_1$.  In this case, by the fact that
$l^{k+1}_j <1$ for $j\in F^{k+1}_2$, we see that $Q^{(k+1)\ast}_j
\subset Q(x^k_i,a(l^k_i +2))$ for $j$ satisfying $Q^{k\ast}_i\cap
Q^{(k+1)\ast}_j\neq\emptyset$.  Let $\gz:=1+2^{-12-n}$. Thus, when
$l^k_i \ge 2/(\gz-1)$, if letting $\wz{Q}^k_i:=Q(x^k_i,a(l^k_i
+2))$, then, $\supp(h^k_i)\subset \wz{Q}^k_i\subset\gz Q^{k\ast}_i
\subset\boz_{2^k}$. When $l^k_i <2/(\gz-1)$, if letting
$\wz{Q}^k_i:=2^6 n Q^{k\ast}_i$, then from Lemma \ref{5.1}(i), we
deduce that $\supp(h^k_i)\subset\wz{Q}^k_i\subset\boz_{2^k}$. By
this and \eqref{5.11}, we conclude that $h^k_i$ is a multiple  of
$(\fai,\,\fz,\,s)$-atom. Moreover, from the definition of
$\wz{Q}^k_i$, Lemma \ref{l2.4}(ii) and Remark \ref{r2.3} with
$\wz{C}:=2/(\gz-1)$, we deduce that there exists a positive
constant $C_{11}$ such that for all $t\in[0,\fz)$,
\begin{eqnarray}\label{5.12}
\fai(\wz{Q}^k_i,t)\le C_{11}\fai(Q^{k\ast}_i,t).
\end{eqnarray}

{\it Case} 2) $i\in E^k_2$. In this case, if $j\in F^{k+1}_1$, then
$l^k_i<l^{k+1}_j /(2^4 n)$. By Lemma \ref{l5.1}(i), we know that
$Q^{k\ast}_i \cap Q^{(k+1)\ast}_j=\emptyset$ for $j\in F^{k+1}_1$.
From this, \eqref{5.2} and \eqref{5.8}, we infer that
\begin{eqnarray}\label{5.13}
h^k_i&=&\lf(f-P^k_i\r)\zez^k_i-\sum_{j\in F^{k+1}_1}f\zez^{k+1}_j
\zez^k_i -\sum_{j\in F^{k+1}_2}\lf(f-P^{k+1}_{j}\r)\zez^{k+1}_j
\zez^k_i+\sum_{j\in F^{k+1}_2}P^{k+1}_{i,\,j}\zez^{k+1}_j\nonumber\\
&=&\lf(f-P^k_i\r)\zez^k_i -\sum_{j\in
F^{k+1}_2}\lf\{\lf(f-P^{k+1}_{j}\r)\zez^{k+1}_j \zez^k_i
-P^{k+1}_{i,\,j} \zez^{k+1}_j\r\}.
\end{eqnarray}
Let $\wz{Q}^k_i:=2^6 n Q^{k\ast}_i$. Then $\supp
(h^k_i)\subset\wz{Q}^k_i$. Moveover, $h^k_i$ satisfies the desired
moment conditions, which are deduced from the moment conditions of
$(f-P^k_i)\zez^k_i$ (see \eqref{5.3}) and
$(f-P^{k+1}_{j})\zez^{k+1}_j \zez^k_i-P^{k+1}_{i,\,j}\zez^{k+1}_j$
(see \eqref{5.4}). Thus, $h^k_i$ is a multiple of some
$(\fai,\,\fz,\,s)$-atom. Moreover, similar to the proof of
\eqref{5.12}, we see that \eqref{5.12} also holds in this case.

By \eqref{5.11}, \eqref{5.12} and  Lemma \ref{l5.4}, we know that
for all $\lz\in(0,\fz)$,
\begin{eqnarray}\label{5.14}
\sum_{k\in\zz}\sum_i\fai\lf(\wz{Q}^k_i,\frac{\|h^k_i\|_{L^{\fz}_{\fai}
(\rn)}}{\lz}\r)&&\ls\sum_{k\in\zz}\sum_i\fai\lf(Q^{k\ast}_i,
\frac{\|h^k_i\|_{L^{\fz}_{\fai}
(\rn)}}{\lz}\r)\nonumber\\
&&\ls\sum_{k\in\zz}\fai\lf(\boz_k,\frac{2^k}{\lz}\r)
\ls\int_{\rn}\fai\lf(x,\frac{\cg_N(f)(x)}{\lz}\r)\,dx,\hs
\end{eqnarray}
which implies that $f\in h_{\fai,\,N}(\rn)$ and
$\|f\|_{h^{\fai,\,\fz,\,s}(\rn)}\ls\|f\|_{h_{\fai,\,N}(\rn)}$.

Finally, we consider the case that $k_0>-\fz$. In this case, from
$f\in h_{\fai,\,N}(\rn)$, it follows that $\fai(\rn,t)<\fz$ for all
$t\in[0,\fz)$. Adapting the previous arguments, we see that
\begin{eqnarray}\label{5.15}
f=\sum_{k=k_0}^{\fz}\lf(g^{k+1}-g^k\r)+g^{k_0}=:\wz{f}+g^{k_0}
\end{eqnarray}
and, for the function $\wz{f}$, we have a
$(\fai,\,\fz,\,s)$-atomic decomposition same as above, namely,
\begin{eqnarray}\label{5.16}
\wz{f}= \sum_{k\ge k_0,\,i} h^k_i
\end{eqnarray}
and
\begin{eqnarray}\label{5.17}
 &&\blz_{\fz}\lf(\{h^k_i
\}_{k\ge k_0,\,i}\r)\ls\|f\|_{h_{\fai,\,N}(\rn)}.
\end{eqnarray}
By Lemma \ref{l4.7}(ii), we know that
\begin{eqnarray}\label{5.18}
\|g^{k_0}\|_{L^{\fz}_{\fai}(\rn)}\le C_9 2^{k_0}\le2C_9
\inf_{x\in\rn}\cg_N (f)(x)
\end{eqnarray}
and hence, by the nondecreasing property on $t$ and the uniformly
upper type 1 property of $\fai$, we see that
\begin{eqnarray}\label{5.19}
\hs\int_{\rn}\fai\lf(x,\frac{\|g^{k_0}\|_{L^{\fz}_{\fai}
(\rn)}}{\lz}\r)\,dx\ls
\int_{\rn}\fai\lf(x,\frac{2^{k_0}}{\lz}\r)\,dx\ls
\int_{\rn}\fai\lf(x,\frac{\cg_N(f)(x)}{\lz}\r)\,dx,\hs\hs\hs
\end{eqnarray}
where $C_9$ is the same as in Lemma \ref{4.7}(ii). Thus, we see
that $g^{k_0}$ is a multiple of some $(\fai,\,\fz)$-single-atom.
From \eqref{5.15}, \eqref{5.16}, \eqref{5.17} and \eqref{5.19}, it
follows that $f\in h_{\fai,\,N}(\rn)$ and
$\|f\|_{h^{\fai,\,\fz,\,s}(\rn)}\ls\|f\|_{h_{\fai,\,N}(\rn)}$ in
the case that $k_0>-\fz$. This finishes the proof of Lemma
\ref{l5.4}.
\end{proof}

\begin{remark}\label{r5.1}
By its proof, any multiple of $(\fai,\,\fz,\,s)$-atoms in Lemma
\ref{l5.5} can be taken to have supports $Q$ satisfying
$l(Q)\in(0,2]$. Indeed, for any multiple of some
$(\fai,\,\fz,\,s)$-atom $b$ supported in a cube $Q_0$ with
$l(Q_0)>2$, we see that there exist $N_0\in\zz_+$, depending on
$l(Q_0)$ and $n$, and cubes $\{Q_i\}_{i=1}^{N_0}$ satisfying
$l(Q_i)\in[1,2]$ with $i\in\{1,\,\cdots,\,N_0\}$ such that
$\cup_{i=1}^{N_0}Q_i=Q_0$,
\begin{equation}\label{5.20}
1\le\sum_{i=1}^{N_0}\chi_{Q_i}(x)\le C(n)
\end{equation}
for all $x\in Q_0$, and
$b=\frac{1}{\sum_{j=1}^{N_0}\chi_{Q_j}}\sum_{i=1}^{N_0}b\chi_{Q_i}$,
where $C(n)$ is a positive integer, only depending on $n$. For
$i\in\{1,\,\cdots,\,N_0\}$, let
$b_i:=\frac{1}{\sum_{j=1}^{N_0}\chi_{Q_j}}b\chi_{Q_i}$. Then
$\supp(b_i)\subset Q_i$. Moreover, from $b\in
L^{\fz}_{\fai}(Q_0)$, we deduce that $b_i\in L^{\fz}_{\fai}(Q_i)$.
Thus, for any $i\in\{1,\,2,\,\cdots,\,N_0\}$, $b_i$ is a multiple
of some $(\fai,\fz,s)$-atom supported in the cube $Q_i$ and
$b=\sum_{i=1}^{N_0}b_i.$ By the definition of $b_i$,
$\cup_{i=1}^{N_0}Q_i=Q_0$ and \eqref{5.20}, we
know that for any $\lz\in(0,\fz)$,
\begin{eqnarray*}
\sum_{i=1}^{N_0}\fai\lf(Q_i,\frac{\|b_i\|_{L^{\fz}_{\fai}(Q_i)}}{\lz}\r)\ls
\fai\lf(Q_0,\frac{\|b\|_{L^{\fz}_{\fai}(Q_0)}}{\lz}\r).
\end{eqnarray*}
Thus, by the proof of Lemma \ref{l5.4}, we see that the claim holds.
\end{remark}

Now, we state the weighted atomic decompositions of
$h_{\fai,\,N}(\rn)$ as follows.

\begin{theorem}\label{t5.1}
Let $\fai$ be a growth function as in Definition \ref{d2.3},
$q(\fai)$, $m(\fai)$ and $N_{\fai}$ respectively as in
\eqref{2.9}, \eqref{2.10} and \eqref{3.20}. If
$q\in(q(\fai),\fz]$, integers $s$ and $N$ satisfy $N\ge N_{\fai}$
and $N>s\ge m(\fai)$, then
$$h^{\fai,\,q,\,s}(\rn)=h_{\fai,\,N}(\rn)=
h_{\fai,\,N_{\fai}}(\rn)$$ 
with equivalent norms.
\end{theorem}

\begin{proof}
By Theorem \ref{t3.3} and the definitions of $h^{\fai,\,q,\,s}(\rn)$
and $h_{\fai,\,N}(\rn)$, we know that
$$ h^{\fai,\,\fz,\,s_1}(\rn)\subset
h^{\fai,\,q,\,s}(\rn)\subset h_{\fai,\,N_{\fai}}(\rn)\subset
h_{\fai,\,N}(\rn)\subset h_{\fai,\,N_1}(\rn),$$ where the integers
$s_1$ and $N_1$ are respectively not less than $s$ and $N$, and
the inclusions are continuous. Thus, to prove Theorem \ref{t5.1},
it suffices to prove that for any integers $N$, $s$ satisfying
$N>s\ge m(\fai)$, $h_{\fai,\,N}(\rn)\subset
h^{\fai,\,\fz,\,s}(\rn)$ and, for all $f\in h_{\fai,\,N}(\rn)$,
$\|f\|_{h^{\fai,\,\fz,\,s}(\rn)}\ls \|f\|_{h_{\fai,\,N}(\rn)}$.

Let $f\in h_{\fai,\,N}(\rn)$. From Corollary \ref{c4.1}, we deduce
that there exists a sequence of functions,
$$\{f_m\}_{m\in\zz_+}\subset h_{\fai,\,N}(\rn)\cap
L^q_{\fai(\cdot,1)}(\rn),$$
such that for all $m\in\zz_+$,
\begin{eqnarray}\label{5.21}
\|f_m\|_{h_{\fai,\,N}(\rn)}\le2^{-m} \|f\|_{h_{\fai,\,N}(\rn)}
\end{eqnarray}
and $f=\sum_{m\in\zz_+} f_m$ in $h_{\fai,\,N}(\rn)$. By Lemma
\ref{l5.5}, we know that for each $m\in\zz_+$, $f_m$ has an atomic
decomposition $f=\sum_{i\in\nn}h^m_i$ in $\cd'(\rn)$ with
$\blz_q(\{h^m_i\}_{i})\ls\|f_m\|_{h_{\fai,\,N}(\rn)}$, where
$\{h^m_i\}_{i\in\zz_+}$ is a sequence of multiples of
$(\fai,\,\fz,\,s)$-atoms and $h^m_0$ is a multiple of some
$(\fai,\,\fz)$-single-atom.

Let $h_0:=\sum^{\fz}_{m=1}h^m_0$. Then $h_0$ is a multiple of some
$(\fai,\fz)$-single-atom and
$$f=\sum_{m=1}^{\fz}\sum_{i\in\zz_+}h^m_i+h_0.$$
Take $p_0\in(0,i(\fai))$. Then $\fai$ is of uniformly lower type $p_0$.
From this, Lemma \ref{l2.1}(i) and \eqref{5.21}, we deduce that
\begin{eqnarray*}
&&\sum_{m=1}^{\fz}\sum_{i=1}^{\fz}\fai\lf(Q_{m,i},
\frac{\|h^m_i\|_{L^{\fz}_{\fai}(Q_{m,i})}}
{\|f\|_{h_{\fai,\,N}(\rn)}}\r)+\fai\lf(\rn,
\frac{\|h_0\|_{L^{\fz}_{\fai}(\rn)}}
{\|f\|_{h_{\fai,\,N}(\rn)}}\r)\\
&&\hs\ls\sum_{m=1}^{\fz}\sum_{i=1}^{\fz}\fai\lf(Q_{m,i},
\frac{\|h^m_i\|_{L^{\fz}_{\fai}(Q_{m,i})}}
{2^{m}\|f_m\|_{h_{\fai,\,N}(\rn)}}\r)+\sum_{m=1}^{\fz}\fai\lf(\rn,
\frac{\|h^m_0\|_{L^{\fz}_{\fai}(\rn)}}
{2^{m}\|f_m\|_{h_{\fai,\,N}(\rn)}}\r)\\
&&\hs\ls\sum_{m=1}^{\fz}2^{-mp_0}\ls1,
\end{eqnarray*}
where for each $m\in\zz_+$ and $i\in\zz_+$, $\supp(h^m_i)\subset
Q_{m,\,i}$,
 which implies that $f\in h^{\fai,\,\fz,\,s}(\rn)$ and
$\|f\|_{h^{\fai,\,\fz,\,s}(\rn)}\ls\|f\|_{h_{\fai,\,N}(\rn)}$. This
finishes the proof of Theorem \ref{t5.1}.
\end{proof}

\begin{remark}\label{r5.2}
Theorem \ref{t5.1} completely covers \cite[Theorem\,5.6]{yys} by
taking $\fai$ as in \eqref{1.1}.
\end{remark}

For simplicity, from now on, we denote by $h_{\fai}(\rn)$ the {\it
local Musielak-Orlicz Hardy space $h_{\fai,\,N}(\rn)$}
when $N\ge N_{\fai}$.

\section{Spaces of finite weighted atoms\label{s6}}

\hskip\parindent In this section, we prove the existence of finite
atomic decompositions achieving the norm in some dense subspaces
of $h_{\fai}(\rn)$. This extends the main results in
\cite{msv08,msv09,yz09} to the setting of the local Musielak-Orlicz
Hardy space. As applications, we obtain that for a given
admissible triplet $(\fai,\,q,\,s)$ and a $\beta$-quasi-Banach
space $\mathcal{B}_{\beta}$ with $\beta\in(0,1]$, if $T$ is a
$\mathcal{B}_{\beta}$-sublinear operator, then $T$ uniquely
extends to a bounded $\mathcal{B}_{\beta}$-sublinear operator from
$h_{\fai}(\rn)$ to $\mathcal{B}_{\beta}$ if and only if $T$ maps
all $(\fai,\,q,\,s)$-atoms and $(\fai,\,q)$-single-atoms with
$q<\infty$ (or all continuous $(\fai,\,q,\,s)$-atoms with
$q=\infty$) into uniformly bounded elements of
$\mathcal{B}_{\beta}$.

\begin{definition}\label{d6.1}
Let $(\fai,\,q,\,s)$ be admissible as in Definition \ref{d3.3}.
The \emph{space $h^{\fai,\,q,\,s}_{\fin}(\rn)$} is defined
to be the vector space of all finite linear combinations
of $(\fai,\,q,\,s)$-atoms or $(\fai,\,q)$-single-atoms,
and the \emph{norm} of $f$ in $h^{\fai,\,q,\,s}_{\fin}(\rn)$ by
\begin{eqnarray*}
&&\|f\|_{h^{\fai,\,q,\,s}_{\fin}(\rn)}:=\inf\lf
\{\blz_q(\{h_i\}_{i=0}^k):\ f =\sum_{i=0}^{k}h_i,\,k\in\nn,\ h_i
\ \text{is a multiple of a}\ \r.\\
&&\hspace{9 em}
(\fai,\,q,\,s)\text{-atom}\ \text{or}  \,\,\text{ a}\,\,
(\fai,\,q)\text{-single-atom}\Bigg\}.
\end{eqnarray*}
\end{definition}

Obviously, for any admissible triplet $(\fai,\,q,\,s)$,
$h^{\fai,\,q,\,s}_{\fin}(\rn)$ is dense in $h^{\fai,\,q,\,s}(\rn)$
with respect to the quasi-norm $\|\cdot\|_{h^{\fai,\,q,\,s}(\rn)}$.

Similar to the notion of the uniformly locally dominated
convergence condition introduced in \cite{k}, we introduce the
following notion of the uniformly locally $q$-dominated
convergence condition with $q\in(q(\fai),\fz)$.

\begin{definition}\label{d6.2}
Let $q\in(q(\fai),\fz)$. A growth function $\fai$ is said to
satisfy \emph{uniformly locally $q$-dominated convergence
condition} if the following holds:

Let $K$ be a compact subset of $\rn$ and $\{f_m\}_{m\in\zz_+}$ a
sequence of measurable functions such that $f_m(x)$ tends to
$f(x)$ for almost every $x\in\rn$. If there exists a nonnegative
measurable function $g$ such that $|f_m(x)|\le g(x)$ for almost
every $x\in\rn$ and
$$\sup_{t\in(0,\fz)}\lf[\int_K
\fai(y,t)\,dy\r]^{-1}\int_K|g(x)|^q\fai(x,t)\,dx<\fz$$
when $\fai(\rn,1)=\fz$, or
$$\sup_{t\in(0,\fz)}\lf[\int_{\rn}
\fai(y,t)\,dy\r]^{-1}\int_{\rn}|g(x)|^q\fai(x,t)\,dx<\fz$$
when $\fai(\rn,1)<\fz$, then
\begin{eqnarray}\label{6.1}
\sup_{t\in(0,\fz)}\lf[\int_K
\fai(y,t)\,dy\r]^{-1}\int_K|f_m(x)-f(x)|^q\fai(x,t)\,dx\to0
\end{eqnarray}
as $k\to\fz$ when $\fai(\rn,1)=\fz$, or \eqref{6.1} is true with
$K$ replaced by $\rn$ when $\fai(\rn,1)<\fz$.
\end{definition}

We point out that the growth functions $\fai(x,t):=\oz(x)\Phi(t)$,
with $\oz\in A^{\loc}_{\fz}(\rn)$ and $\Phi$ being an Orlicz
function, or
$$\fai(x,t):=\frac{t^p}{[\log(e+|x|)+\log(e+t^p)]^p}$$
for all $x\in\rn$ and $t\in[0,\fz)$, with
$p\in(0,1]$, satisfy the uniformly locally $q$-dominated
convergence condition with $q\in(q(\fai),\fz)$; see \cite{k}.
In what follows, let $h^{\fai,\,\fz,\,s}_{\fin,\,c}(\rn)$
denote the \emph{set of all $f\in h^{\fai,\,\fz,\,s}_{\fin}(\rn)$
with compact support}.

\begin{theorem}\label{t6.1}
Let $q(\fai)$ be as in \eqref{2.9} and $(\fai,\,q,\,s)$
admissible.

$\mathrm{(i)}$ If $q\in(q(\fai),\fz)$ and $\fai$ satisfies the
uniformly locally $q$-dominated convergence condition, then
$\|\cdot\|_{h^{\fai,\,q,\,s}_{\fin}(\rn)}$ and
$\|\cdot\|_{h_{\fai}(\rn)}$ are equivalent quasi-norms on
$h^{\fai,\,q,\,s}_{\fin}(\rn)\cap L^q_{\fai(\cdot,1)}(\rn)$.

$\mathrm{(ii)}$  The quasi-norms $\|\cdot\|_{h^{\fai,\,\fz,\,s}_{\fin}(\rn)}$ and
$\|\cdot\|_{h_{\fai}(\rn)}$ are equivalent  on
$h^{\fai,\,\fz,\,s}_{\fin,\,c}(\rn)\cap C(\rn)$.
\end{theorem}

\begin{proof}
We first show (i). Let $q\in(q(\fai),\fz)$ and $(\fai,\,q,\,s)$ be
admissible. Obviously, from Theorem \ref{t5.1}, we deduce that
$h^{\fai,\,q,\,s}_{\fin}(\rn)\subset
h^{\fai,\,q,\,s}(\rn)=h_{\fai}(\rn)$ and, for all $f\in
h^{\fai,\,q,\,s}_{\fin}(\rn)$,
$\|f\|_{h_{\fai}(\rn)}\ls
\|f\|_{h^{\fai,\,q,\,s}_{\fin}(\rn)}.$
Thus, to prove (i), we only need to show that for all $f\in
h^{\fai,\,q,\,s}_{\fin}(\rn)\cap L^q_{\fai(\cdot,1)}(\rn)$,
\begin{eqnarray}\label{6.2}
\|f\|_{h^{\fai,\,q,\,s}_{\fin}(\rn)}\ls \|f\|_{h_{\fai}(\rn)}.
\end{eqnarray}
Moreover, by homogeneity, without loss of generality, we may
assume that $f\in h^{\fai,\,q,\,s}_{\fin}(\rn)\cap
L^q_{\fai(\cdot,1)}(\rn)$ with $\|f\|_{h_{\fai}(\rn)}=1$. In the
remainder of this section, for any $f\in
h^{\fai,\,q,\,s}_{\fin}(\rn)$, let $k_0$ be as in Section \ref{s5}
and $\boz_{2^k}$ with $k\ge k_0$ as in \eqref{4.1} with $\lz=2^k$.
Since $f\in(h_{\fai,\,N}(\rn)\cap L^{q}_{\fai(\cdot,1)}(\rn))$, by
Lemma \ref{l5.5}, there exist a multiple of some
$(\fai,\,\fz)$-singe-atom $h_0$ and a sequence $\{h^k_i\}_{k\ge
k_0,\,i}$ of multiples of $(\fai,\,\fz,\,s)$-atoms such that
\begin{eqnarray}\label{6.3}
f=\sum_{k\ge k_0}\sum_{i}h^k_i+h_0
\end{eqnarray}
holds both in $\cd'(\rn)$ and almost everywhere.  First, we claim
that \eqref{6.3} also holds in $L^q_{\fai(\cdot,1)}(\rn)$. For any
$x\in\rn$, by $\rn=\cup_{k\ge k_0-1}(\boz_{2^k}\setminus
\boz_{2^{k+1}})$, we see that there exists $j\in\zz$ such that
$x\in(\boz_{2^j}\setminus \boz_{2^{j+1}})$. By the proof of Lemma
\ref{l5.5}, we know that for all $k>j$, $\supp(h^k_i)\subset
\wz{Q}^k_i\subset\boz_{2^k}\subset\boz_{2^{j+1}}$; then from
\eqref{5.11} and \eqref{5.18}, it follows that
$$\lf|\sum_{k\ge k_0}\sum_{i}h^k_i
(x)\r|+|h_0 (x)|\ls\sum_{k_0\le k\le j}2^k +2^{k_0}\ls
2^j\ls\cg_N(f)(x).$$ Since $f\in L^q_{\fai(\cdot,1)}(\rn)$, by
Proposition \ref{p3.1}(ii), we know that $\cg_N (f)\in
L^q_{\fai(\cdot,1)}(\rn)$. This, combined with the Lebesgue
dominated convergence theorem, implies that $\sum_{k\ge
k_0}\sum_{i}h^k_i +h_0$ converges to $f$ in
$L^q_{\fai(\cdot,1)}(\rn)$, which proves the claim.

Now, we show \eqref{6.2} by considering the following two cases for
$\fai$.

{\it Case} 1) $\fai(\rn,1)=\fz$. In this case, we know that
$k_0=-\fz$ and $h_0 (x)=0$ for almost every $x\in\rn$. Thus, in
this case, \eqref{6.3} has the version $f=\sum_{k\in\zz}
\sum_{i}h^k_i$. Since, when $\fai(\rn,1)=\fz$, all
$(\fai,\,q)$-single-atoms are 0, and hence if $f\in
h^{\fai,\,q,\,s}_{\fin}(\rn)$, then $f$ has compact support.
Assume that $\supp f\subset Q_0:= Q(x_0,r_0)$ and let $\wz{Q}_0:=
Q(x_0,\sqrt{n}r_0 +2^{3(10+n)+1})$. Then for any $\pz\in\cd_N
(\rn)$, $x\in\rn\setminus\wz{Q}_0$ and $t\in(0,1)$, we see that
\begin{eqnarray*}
\pz_t \ast f(x)=\int_{Q(x_0,r_0)}\pz_t
(x-y)f(y)\,dy=\int_{B(x,2^{3(10+n)})\cap Q(x_0,r_0)}\pz_t
(x-y)f(y)\,dy=0.
\end{eqnarray*}
Thus, for any $k\in\zz$, $\boz_{2^k}\subset\wz{Q}_0$, which
implies that  $\supp(\sum_{k\in\zz} \sum_{i}h^k_i)\subset
\wz{Q}_0$. For each positive integer $M$, let
$$F_M:=\{(i,k):\
k\in\zz,\ k\ge k_0, i\in\zz_+,\ |k|+i\le M\}$$
and $f_M:=\sum_{(k,i)\in F_M}h^k_i$. Then, from the above claim, we
deduce that $f_M$ converges to $f$ in $L^q_{\fai(\cdot,1)}(\rn)$.
Moreover, by $f\in h^{\fai,\,q,\,s}(\rn)$, we see that there exist
$N\in\zz_+$ and a sequence $\{h_i\}_{i=1}^N$ of multiples of some
$(\fai,\,q,\,s)$-atoms such that $f=\sum_{i=1}^N h_i$ almost
everywhere. Let $g:=\sum_{i=1}^N |h_i|$. It is easy to see that
for any compact set $K$ of $\rn$,
$$\sup_{t\in(0,\fz)}\lf\{\lf[\int_K
\fai(y,t)\,dy\r]^{-1}\int_{K}|g(x)|^q\fai(x,t)\,dx\r\}<\fz.$$
Then from the assumption that $\fai$ satisfies the uniformly locally
$q$-dominated convergence condition, we infer that
$$\sup_{t\in(0,\fz)}\lf\{\lf[\int_K
\fai(y,t)\,dy\r]^{-1}\int_{K}|f_M(x)-f(x)|^q\fai(x,t)\,dx\r\}\to0$$
as $M\to\fz$. Thus, there exists $M_0\in\zz_+$ such that $f-f_{M_0}$
is a multiple of some $(\fai,\,q,\,s)$-atom and
$\fai(\wz{Q}_0,\|f-f_{M_0}\|_{L^q_{\fai}(\wz{Q}_0)})\ls1$. By this
and Lemma \ref{l5.5}, we see that
$$\|f\|_{h^{\fai,\,q,\,s}(\rn)}\ls\blz_q(\{h^k_i\}_{(k,i)\in
F_{M_0}})+\blz_q(\{f-f_{M_0}\})\ls1,
$$
which implies \eqref{6.2} in Case 1).

{\it Case} 2) $\fai(\rn,1)<\fz$. In this case, $f$ may not have
compact support. For any positive integer $M$, let
$f_M:=\sum_{(k,i)\in F_M}h^k_i+h_0$ and $b_M:=f-f_M$, where $F_M$ is
as in Case 1). Similar to the proof of Case 1), there exists a
positive integer $M_1\in\zz_+$ large enough such that $b_{M_1}$ is a
multiple of some $(\fai,\,q)$-single-atom and
$\fai(\rn,\|b_{M_1}\|_{L^q_{\fai}(\rn)})\ls1$. Thus,
$f=f_{M_1}+b_{M_1}$ is a finite linear atom combination of $f$ and
\begin{eqnarray*}
\|f\|_{h^{\fai,\,q,\,s}_{\fin}(\rn)}&\ls&\blz_q\lf(\{h^k_i\}
_{(i,k)\in F_{M_1}}\r)+\blz_q\lf(\{b_{M_1}\}\r)\\
&\ls&\|f\|_{h^{\fai,\,q,\,s}(\rn)}+\inf\lf\{\lz\in(0,\fz):\
\fai\lf(\rn,\frac{\|b_{M_1}\|_{L^q_{\fai}(\rn)}}{\lz}\r)\le1\r\}\ls1,
\end{eqnarray*}
which implies \eqref{6.2} in Case 2). This finishes the proof of
(i).

We now prove (ii). In this case, similar to the proof of (i), we
only need to prove that for all $f\in
h^{\fai,\,\fz,\,s}_{\fin,\,c}(\rn)$,
$\|f\|_{h^{\fai,\,\fz,\,s}_{\fin}(\rn)}\ls \|f\|_{h_{\fai}(\rn)}$.
Again, by homogeneity, without loss of generality, we may assume
that $\|f\|_{h_{\fai}(\rn)}=1$. Since $f$ has compact support, by
the definition of $\cg_N (f)$, we know that $\cg_N (f)$ also has
compact support. Assume that $\supp(\cg_N (f))\subset B(0,R_0)$
for some $R_0 \in(0,\fz)$. From $f\in L^{\fz}(\rn)$, we deduce
that $\cg_N (f)\in L^{\fz}(\rn)$. Thus, there exists $k_1\in\zz$
such that $\boz_{2^k}=\emptyset$ for any $k\in\zz$ with $k\ge k_1
+1$. By Lemma \ref{l5.5}, there exist a multiple of some
$(\fai,\,\fz)$-singe-atom $h_0$ and a sequence $\{h^k_i\}_{k_1 \ge
k\ge k_0,\,i}$ of multiples of $(\fai,\,\fz,\,s)$-atoms such that
$f=\sum^{k_1}_{k=k_0}\sum_{i}h^k_i+h_0$ holds both  in $\cd'(\rn)$
and almost everywhere. From the fact that $f$ is uniformly
continuous, it follows that for any given $\varepsilon\in(0,\fz)$,
there exists a $\delta\in(0,\fz)$ such that if
$|x-y|<\sqrt{n}\delta/2$, then $|f(x)-f(y)|<\varepsilon$. Without
loss of generality, we may assume that $\delta<1$. Write
$f=f^{\varepsilon}_1 +f^{\varepsilon}_2$ with
$f^{\varepsilon}_1:=\sum_{(i,k)\in G_1}h^k_i +h_0$ and
$f^{\varepsilon}_2:=\sum_{(i,k)\in G_2}h^k_i $, where
$$G_1:=\lf\{(i,k):\ l(\wz{Q}^k_i)\ge\delta,\ k_0 \le k\le k_1\r\},$$
$G_2:=\{(i,k):\ l(\wz{Q}^k_i)<\delta,\ k_0 \le k\le k_1\}$, and
$\wz{Q}^k_i$ is the support of $h^k_i$ (see the proof of Lemma
\ref{l5.5}). For any fixed integer $k\in[k_0,k_1]$, by Lemma
\ref{l5.1}(ii) and $\boz_{2^k}\subset B(0,R_0)$, we see that $G_1$
is a finite set.

For $f^{\varepsilon}_2$, similar to the proof of
\cite[pp.\,44-45]{yys}, we know that
$|f^{\varepsilon}_2|\ls\sum_{k=k_0}^{k_1}\varepsilon\ls(k_1-k_0)
\varepsilon.$ By the arbitrariness of $\varepsilon$, $\supp
(f^{\varepsilon}_2)\subset B(0,R_0)$ and
$|f^{\varepsilon}_2|\ls(k_1-k_0) \varepsilon$, we choose
$\varepsilon$ small enough such that $f^{\varepsilon}_2$ is an
arbitrarily small multiple of some $(\fai,\,\fz,\,s)$-atom. In
particular, we choose $\varepsilon_0\in(0,\fz)$ such that
$\fai(B(0,R_0),\|f^{\varepsilon_0}_2\|_{L^{\fz}_{\fai}(\rn)})\ls1$.
Then $f=\sum_{(i,k)\in G_1}h^k_i+h_0 +f^{\varepsilon_0}_2$ is a
finite atomic decomposition of $f$, and
$\|f\|_{h^{\fai,\,\fz,\,s}(\rn)}\ls\|f\|_{h_{\fai}(\rn)}+1\ls1$,
which completes the proof of Theorem \ref{t6.1}.
\end{proof}

\begin{remark}\label{r6.1}
(i) Let $q(\fai)$ be as in \eqref{2.9} and $(\fai,\,q,\,s)$
admissible with $q\in(q(\fai),\fz)$. From the proof of Lemma
\ref{l5.5}, we infer that $h^{\fai,\,q,\,s}_{\fin}(\rn)\cap
L^q_{\fai(\cdot,\,1)}(\rn)$ is dense in $h_{\fai}(\rn)\cap
L^q_{\fai(\cdot,\,1)}(\rn)$ with respect to the quasi-norm
$\|\cdot\|_{h_{\fai}(\rn)}$, which, together with Corollary
\ref{c4.1}, implies that $h^{\fai,\,q,\,s}_{\fin}(\rn)\cap
L^q_{\fai(\cdot,\,1)}(\rn)$ is dense in $h_{\fai}(\rn)$ with
respect to the quasi-norm $\|\cdot\|_{h_{\fai}(\rn)}$.

(ii) Obviously, when $\fai(\rn,1)=\fz$,
$$h^{\fai,\,\fz,\,s}_{\fin,\,c}(\rn)\cap
C(\rn)=h^{\fai,\,\fz,\,s}_{\fin}(\rn)\cap C(\rn).$$
\end{remark}

As an application of Theorem \ref{t6.1}, we establish the
boundedness on $h_{\fai}(\rn)$ of quasi-Banach-valued sublinear
operators.

Recall that a {\it quasi-Banach space $\cb$} is a vector space
endowed with a quasi-norm $\|\cdot\|_{\cb}$ which is nonnegative,
non-degenerate (namely, $\|f\|_{\cb}=0$ if and only if $f= 0$),
homogeneous, and obeys the quasi-triangle inequality, namely,
there exists a positive constant $K\in[1,\fz)$ such that for all
$f,\,g\in\cb$,
$$\|f + g\|_{\cb}\le K(\|f\|_{\cb} + \|g\|_{\cb}).$$

Let $\bz\in(0, 1]$. As in \cite{k,yz08,yz09}, a \emph{quasi-Banach
space $\cb_{\bz}$} with the quasi-norm $\|\cdot\|_{\cb_{\bz}}$ is
called a {\it $\bz$-quasi-Banach space} if there exists a positive
constant $\kappa\in[1,\fz)$ such that for all $f_j\in\cb_{\bz}$
with $j\in\{1,\,\cdots,\,m\}$,
$$\lf\|\sum_{j=1}^mf_j\r\|^{\bz}_{\cb_{\bz}}\le\kappa\sum_{j=1}^m
\|f_j\|^{\bz}_{\cb_{\bz}}.$$

Notice that any Banach space is a 1-quasi-Banach space, and the
quasi-Banach spaces $l^{\bz}$ and $L^{\bz}_{\oz}(\rn)$ with $\oz\in
A^{\loc}_{\fz}(\rn)$ are typical $\bz$-quasi-Banach spaces. Let
$\fai$ be a growth function as in Definition \ref{d2.3} with the
uniformly lower type $p_0\in(0,1]$. We know that $h_{\fai}(\rn)$ is
a $p_0$-quasi-Banach space.

For any given $\bz$-quasi-Banach space $\cb_{\bz}$ with $\bz\in(0,
1]$ and a linear space $\cy$, an operator $T$ from $\cy$ to
$\cb_{\bz}$ is called {\it $\cb_{\bz}$-sublinear} if there exists
a positive constant $\kappa\in[1,\fz)$  such that for all
$f_j\in\cb_{\bz}$ and $\lz_j\in\ccc$ with $j\in\{1,\,\cdots,\,m\}$,
$$\lf\|T\lf(\sum_{j=1}^m \lz_j f_j\r)\r\|^{\bz}_{\cb_{\bz}}\le\kappa\sum_{j=1}^m
|\lz_j|^{\bz}\|T(f_j)\|_{\cb_{\bz}}^{\bz},$$ and
$\|T(f)-T(g)\|_{\cb_{\bz}}\le \kappa\|T(f-g)\|_{\cb_{\bz}}$ for
all $f,\,g\in\cy$ (see \cite{yz08,yz09}).

We remark that if $T$ is linear, then $T$ is
$\cb_{\bz}$-sublinear. Moreover, if $\cb_{\bz}$ is a space of
functions, and $T$ is nonnegative and sublinear in the classical
sense, then $T$ is also $\cb_{\bz}$-sublinear.

\begin{theorem}\label{t6.2}
Let $q(\fai)$ be as in \eqref{2.9} and $(\fai,\,q,\,s)$
admissible. Let $\cb_{\bz}$ be a $\bz$-quasi-Banach space with
$\bz\in(0,1]$ and $\wz{p}$ a uniformly upper type of $\fai$
satisfying $\wz{p}\in(0,\bz]$.

$\mathrm{(i)}$ Let $q\in(q(\fai),\fz)$, $\fai$ satisfy the
uniformly locally $q$-dominated convergence condition and $T:\
h^{\fai,\,q,\,s}_{\fin}(\rn)\to \cb_{\bz}$ be a
$\cb_{\bz}$-sublinear operator. Then $T$ uniquely extends to a
bounded $\cb_{\bz}$-sublinear operator from $h_{\fai}(\rn)$ to
$\cb_{\bz}$ if and only if
\begin{eqnarray*}
&&S:=\sup\lf\{\|T(a)\|_{\cb_{\bz}}:\ a \ \text{is
any}\,(\fai,\,q,\,s)\text{-atom with}\,\supp (a)\subset
Q\ \r.\\
&& \hspace{4 em}\text{or\ any}\ (\fai,\,q)\text{-single-atom}\big\}<\fz.
\end{eqnarray*}

$\mathrm{(ii)}$ Let $\fai$ satisfy the uniformly locally
$q_0$-dominated convergence condition for some
$q_0\in(q(\fai),\fz)$ and $T$ be a $\cb_{\bz}$-sublinear operator
defined on all continuous $(\fai,\,\fz,\,s)$-atoms. Then $T$
uniquely extends to a bounded $\cb_{\bz}$-sublinear operator from
$h_{\fai}(\rn)$ to $\cb_{\bz}$ if and only if
$$S:=\sup\{\|T(a)\|_{\cb_{\bz}}:\ a \ \text{is any continuous}\
(\fai,\,\fz,\,s)\text{-atom}\}<\fz.$$
\end{theorem}

\begin{proof}
We first show (i). Obviously, it suffices to show that when
$S<\fz$, $T$ uniquely extends to a bounded $\cb_{\bz}$-sublinear
operator from $h_{\fai}(\rn)$ to $\cb_{\bz}$.   By Theorem
\ref{t6.1}(i) and Remark \ref{r3.1}, without loss of generality,
we may assume that for any $f\in h^{\fai,\,q,\,s}_{\fin}(\rn)$,
there exist a sequence $\{\lz_j\}^l_{j=0}\subset\ccc$ with some
$l\in\zz_+$, a $(\fai,\,q)$-single-atom $a_0$ and $(\fai,\,q,\,
s)$-atoms $\{a_j\}^l_{j=1}$ satisfying $\supp (a_j)\subset Q_j$
for $j\in\{1,\,2,\,\cdots,\,l\}$ such that $f =\sum^l_{j=0}\lz_j
a_j$ almost everywhere and
\begin{eqnarray}\label{6.4}
\blz_q\lf(\{\lz_j a_j\}_{j=0}^l\r)
&&=\inf\lf\{\lz\in(0,\fz):\ \sum_{j=1}^l\fai\lf(Q_j,\frac{|\lz_j|
\|\chi_{Q_j}\|_{L^{\fai}(\rn)}^{-1}}{\lz}\r)\r.\nonumber\\
&&\hs\hs\lf.+\fai\lf(\rn,\frac{|\lz_0|
\|\chi_{\rn}\|_{L^{\fai}(\rn)}^{-1}}{\lz}\r)\le 1\r\}\ls\|f\|_{h_{\fai}(\rn)}.
\end{eqnarray}
Then from $S<\fz$ and the assumption that $T$ is
$\cb_{\bz}$-sublinear, we infer that
\begin{eqnarray}\label{6.5}
\|T(f)\|_{\cb_{\bz}}&\ls&\lf\{\sum_{i=0}^l
|\lz_i|^{\bz}\|T(a)\|_{\cb_{\bz}}^{\bz}\r\}^{1/\bz}\ls
\lf\{\sum_{i=0}^l |\lz_i|^{\wz{p}}\|T(a)\|_{\cb_{\bz}}^{\wz{p}}\r\}
^{1/\wz{p}}\ls\lf\{\sum_{i=0}^l
|\lz_i|^{\wz{p}}\r\}^{1/\wz{p}}.\hs\hs\hs
\end{eqnarray}
Since $\fai$ is of uniformly upper type $\wz{p}$, then for all
$x\in\rn$, $t\in(0,1]$ and $s\in(0,\fz)$, we know that
$\fai(x,st)\gs t^{\wz{p}}\fai(x,s)$. Let
$\wz{\lz}_0:=\{\sum_{i=0}^l |\lz_i|^{\wz{p}}\}^{1/\wz{p}}$. Then,
\begin{eqnarray*}
&&\sum_{i=0}^l\fai\lf(Q_i,\frac{|\lz_i|
\|\chi_{Q_i}\|_{L^{\fai}(\rn)}^{-1}}{\wz{\lz}_0}\r)+
\fai\lf(\rn,\frac{|\lz_0|
\|\chi_{\rn}\|_{L^{\fai}(\rn)}^{-1}}{\wz{\lz}_0}\r)\\
&&\hs\gs\lf[\sum_{j=0}^l|\lz_j|^{\wz{p}}\r]^{-1}\lf[\sum_{i=1}^l
|\lz_i|^{\wz{p}}\fai\lf(Q_i,\|\chi_{Q_i}\|_{L^{\fai}(\rn)}^{-1}\r)\r.\\
&&\hs\hs+
|\lz_0|^{\wz{p}}
\fai\lf(\rn,\|\chi_{\rn}\|_{L^{\fai}(\rn)}^{-1}\r)\bigg] \sim1.
\end{eqnarray*}
Thus, from this, we deduce that $\wz{\lz}_0\ls\blz_q(\{\lz_i
a_i\}_{i=0}^l)$, which, together with \eqref{6.4} and \eqref{6.5},
implies that
$$\|T(f)\|_{\cb_{\bz}}\ls\wz{\lz}_0\ls\blz_q(\{\lz_i
a_i\}_{i=0}^l)\ls\|f\|_{h_{\fai}(\rn)}.$$
By Remark \ref{r6.1}(i),
we know that $h^{\fai,\,q,\,s}_{\fin}(\rn)\cap
L^q_{\fai(\cdot,1)}(\rn)$ is dense in $h_{\fai}(\rn)$, which,
together with a density argument, implies the desired conclusion
in this case.

Now, we prove (ii). Similar to the proof of (i), it suffices to
show that when $S<\fz$, $T$ uniquely extends to a bounded
$\cb_{\bz}$-sublinear operator from $h_{\fai}(\rn)$ to
$\cb_{\bz}$. We prove this by considering the following two cases
for $\fai$.

{\it Case} 1) $\fai(\rn,1)=\fz$. In this case, similar to the
proof of (i), using Theorem \ref{t6.1}(ii) and Remark
\ref{r6.1}(ii), we know that for all $f\in
h^{\fai,\,\fz,\,s}_{\fin}(\rn)\cap C(\rn)$,
$\|T(f)\|_{\cb_{\bz}}\ls\|f\|_{h_{\fai}(\rn)}$. To extend $T$ to
the whole $h_{\fai}(\rn)$, we only need to prove that
$h^{\fai,\,\fz,\,s}_{\fin}(\rn)\cap C(\rn)$ is dense in
$h_{\fai}(\rn)$. Since $h^{\fai,\,\fz,\,s}_{\fin}(\rn)$ is dense
in $h_{\fai}(\rn)$, it suffices to prove that
$h^{\fai,\,\fz,\,s}_{\fin}(\rn)\cap C(\rn)$ is dense in
$h^{\fai,\,\fz,\,s}_{\fin}(\rn)$ with respect to the quasi-norm
$\|\cdot\|_{h_{\fai}(\rn)}$.

To see this, let $f\in h^{\fai,\,\fz,\,s}_{\fin}(\rn)$. In this
case, for any $(\fai,\,\fz)$-single-atom $b$, $b(x)=0$ for almost
every $x\in\rn$. Thus, $f$ is a finite linear combination of
$(\fai,\,\fz,\,s)$-atoms. Then there exists a cube $Q_0:=Q(x_0,
r_0)$ such that $\supp(f)\subset Q_0$. Take $\phi\in\cd(\rn)$ such
that $\supp\phi\subset Q(0,1)$ and $\int_{\rn}\phi(x)\,dx=1$. Then
it is easy to see that for any $k\in\zz_+$, $\supp(\phi_k \ast f)
\subset Q(x_0, r_0 +1)$ and $\phi_k \ast f\in\cd(\rn)$. Assume that
$f=\sum_{i=1}^{N}\lz_i a_i$ with some $N\in\zz_+$,
 $\{\lz_i\}_{i=1}^N \subset\ccc$ and $\{a_i\}_{i=1}^N $ being
$(\fai,\,\fz,\,s)$-atoms. Then for any $k\in\zz_+$, $\phi_k \ast
f=\sum_{i=1}^{N}\lz_i \phi_k \ast a_i$. For any $k\in\zz_+$ and
$i\in\{1,\,2\,\cdots,\,N\}$, we now prove that $\phi_k \ast a_i$
 is a  multiple of some continuous $(\fai,\,\fz,\,s)$-atom, which
implies that for any $k\in\zz_+$,
\begin{eqnarray}\label{6.6}
\phi_k \ast f\in h^{\fai,\,\fz,\,s}_{\fin}(\rn)\cap C(\rn).
\end{eqnarray}
For $i\in\{1,\,2,\,\cdots,\,N\}$, assume that
$$\supp(a_i) \subset Q_i := Q(x_i,r_i).$$
Then $\supp(\phi_k \ast a_i)\subset
\wz{Q}_{i,\,k}:= Q(x_i,r_i+1/2^k)$. Moreover,
$$\|\phi_k \ast a_i\|_{L^{\fz}_{\fai}(\rn)}\le\|a_i\|_{L^{\fz}_{\fai}(\rn)}\le
\|\chi_{Q_i}\|_{L^{\fai}(\rn)}^{-1}.$$
Furthermore, for any $\az\in\nn^n$, $\int_{\rn}a_i (x)x^{\az}\,dx=0$ implies that
$\int_{\rn}\phi_k \ast a_i (x) x^{\az}\,dx=0$. Thus,
$$\frac{\|\chi_{Q_i}\|_{L^{\fai}(\rn)}}{\|\chi_{\wz{Q}_{i,\,k}}
\|_{L^{\fai}(\rn)}}\phi_k \ast a_i$$
is a $(\fai,\,\fz,\,s)$-atom.

Likewise, $\supp(f-\phi_k \ast f)\subset Q(x_0, r_0 +1)$ and
$f-\phi_k \ast f$ has the same vanishing moments as $f$. By the
assumption that $\fai$ satisfies the uniformly locally
$q_0$-dominated convergence condition, we know that for any compact
set $K$ of $\rn$,
$$\sup_{t\in(0,\fz)}\lf\{\lf[\int_K\fai(y,t)\,dy\r]^{-1}
\int_K|f(x)-\phi_k\ast f(x)|^{q_0}\fai(x,t)
\,dx\r\}\to0,$$
as $k\to\fz$. Let
\begin{eqnarray*}
c_k:=\sup_{t>0}\lf\{\lf[\int_{Q(x_0,r_0+1)}\fai(y,t)\,dy\r]^{-1}
\int_{Q(x_0,r_0+1)}
|f(x)-\phi_k\ast
f(x)|^{q_0}\fai(x,t)\,dx\r\}^{1/q_0}\|\chi_{Q(x_0,r_0+1)}\|_{L^{\fai}(\rn)}
\end{eqnarray*}
and $a_k:=(f-\phi_k\ast f)/c_k$. Then $a_k$ is a
$(\fai,\,q_0,\,s)$-atom, $f-\phi_k\ast f=c_ka_k$ and $|c_k|\to0$,
as $k\to\fz$. Then
\begin{eqnarray}\label{6.7}
\|f-\phi_k\ast f\|_{h_{\fai}(\rn)}
&&\ls\blz_q(\{c_ka_k\})\nonumber\\
&&=\inf\lf\{\lz\in(0,\fz):\
\fai\lf(Q(x_0,r_0+1),\frac{\|f-\phi_k\ast
f\|_{L^{q_0}_{\fai}(Q(x_0,r_0+1))}}{\lz}\r)\le1\r\}\nonumber\\
&&\hs\ls|c_k|\to0,
\end{eqnarray}
as $k\to\fz$, which, together with \eqref{6.6}, shows the desired
conclusion in this case.

{\it Case} 2) $\fai(\rn,1)<\fz$. In this case, similar to the
proof of Case 1), by Theorem \ref{t6.1}(ii), it suffices to prove
that $h^{\fai,\,\fz,\,s}_{\fin,\,c}(\rn)\cap C(\rn)$ is dense in
$h^{\fai,\,\fz,\,s}_{\fin}(\rn)$ with respect to the quasi-norm
$\|\cdot\|_{h_{\fai}(\rn)}$.

For any $f\in h^{\fai,\,\fz,\,s}_{\fin}(\rn)$, by Remark
\ref{r3.1}, without loss of generality, we may assume that
$$f=\sum_{i=1}^{N_1} \lz_i a_i+\lz_0 a_0,$$ 
where $N_1\in\zz_+$,
$\{\lz_i\}_{i=0}^{N_1}\subset\ccc$ and $a_0$ is a
$(\fai,\,\fz)$-single-atom and $\{a_i\}_{i=1}^{N_1}$
 are $(\fai,\,\fz,\,s)$-atoms. Let
$\{\psi_k\}_{k\in\zz_+}\subset\cd(\rn)$ satisfy $0\le\psi_k\le1$,
$\psi_k\equiv1$ on the cube $Q(0,2^k)$ and $\supp\psi_k\subset
Q(0,2^{k+1})$. We assume that $\supp(\sum_{i=1}^{N_1} \lz_i
a_i)\subset Q(0,R_0)$ for some $R_0\in(0,\fz)$ and $k_0$ is the
smallest nonnegative integer such that $2^{k_0}\ge R_0$. For any
integer $k\ge k_0$, let $f_k:= f\psi_k$. Then $f_k\in
h^{\fai,\,\fz,\,s}_{\fin,\,c}(\rn)$. Indeed, by the choice of
$\pz_k$, we know that 
$$f_k=\sum_{i=1}^{N_1}\lz_i a_i+\lz_0 a_0\pz_k$$ 
and $\supp f_k\subset Q(0,2^{k+1})$. Furthermore, from
$\supp(a_0 \pz_k)\subset Q(0,2^{k+1})$ and
$$\|a_0\pz_k\|_{L^{\fz}_{\fai}(\rn)}\le\|a_0\|_{L^{\fz}_{\fai}(\rn)}
\le\|\chi_{\rn}\|^{-1}_{L^{\fai}(\rn)}\le\|\chi_{Q(0,2^{k+1})}\|
^{-1}_{L^{\fai}(\rn)},$$ we infer that $a_0 \pz_k$ is a
$(\fai,\,\fz,\,s)$-atom. Thus, $f_k\in
h^{\fai,\,\fz,\,s}_{\fin,\,c}(\rn)$. For any fixed integer $k\ge
k_0$ and any $i\in\zz_+$, let $\wz{f}_{k,\,i}:= f_k \ast\phi_i$,
where $\phi$ is as in Case 1). Similar to the proof of
\eqref{6.6}, we know that $\wz{f}_{k,\,i}\in
h^{\fai,\,\fz,\,s}_{\fin,\,c}(\rn)\cap C(\rn)$. For any
$q\in(q(\fai),\fz)$, by the choice of $f_k$ and $\fai(\rn,1)<\fz$,
we see that
\begin{eqnarray}\label{6.8}
\|f-f_k\|_{L^q_{\fai(\cdot,1)}(\rn)}&\le&
\lf\{\int_{[Q(0,2^{k})]^{\complement}}|f(x)|^q \fai(x,1)\,dx\r\}^{1/q}\nonumber \\
&\le&\|\lz_0
a_0\|_{L^{\fz}_{\fai}(\rn)}\lf\{\int_{[Q(0,2^{k})]^{\complement}}
\fai(x,1)\,dx\r\}^{1/q}\to0,
\end{eqnarray}
as $k\to\fz$. Furthermore, for any fixed $k\in\zz$ with $k\ge
k_0$, similar to the proof of \eqref{6.8}, we see that
$\|f_k-\wz{f}_{k,\,i}\|_{L^q_{\fai(\cdot,1)}(\rn)}\to0\,
\,\text{as}\,\,i\to\fz$, which, together with \eqref{6.8}, implies
that $\|f-\wz{f}_{k,\,i}\|_{L^q_{\fai(\cdot,1)}(\rn)}\to0,\,
\,\text{as}\,\,k,\,i\to\fz$. Then similar to the proof of
\eqref{6.7}, we know that $\|f-\wz{f}_{k,\,i}
\|_{h_{\fai}(\rn)}\to0$, as $k,\,i\to\fz$, which completes the
proof of Case 2) and hence Theorem \ref{t6.2}.
\end{proof}

\begin{remark}\label{r6.2}
Theorems \ref{t6.1} and \ref{t6.2} completely cover
\cite[Theorem\,6.2]{yys} and \cite[Theorem \,6.4]{yys},
respectively, by taking $\fai$ as in \eqref{1.1}.
\end{remark}

\section{The dual space of $h_\fai(\rn)$ with applications to
pointwise multipliers on local $\bbmo$-type spaces\label{s7}}

\hskip\parindent In Subsection \ref{s7.1}, we give out the dual
space, $\bmo_\fai(\rn)$, of $h_\fai(\rn)$. As an application, in
Subsection \ref{s7.2}, we characterize the class of pointwise
multipliers for the local $\bbmo$-type space $\bmo^\phi(\rn)$,
where $\bmo^\phi(\rn)$ is the local $\bbmo$-type space introduced
by Nakai and Yabuta \cite{ny85}.

\subsection{The dual space of $h_\fai(\rn)$ \label{s7.1}}

\hskip\parindent In this subsection, we introduce the $\bbmo$-type
space $\bmo_{\fai}(\rn)$ and show that the dual space of
$h_{\fai}(\rn)$ is $\bmo_{\fai}(\rn)$. We begin with some notions.

For any locally integrable function $f$ on $\rn$, we denote the
{\it minimizing polynomial} of $f$ on the cube $Q$ with degree at most
$s$ by $P^s_Q f$, namely, for all multi-indices $\tz\in\nn^n$ with
$0\le|\tz|\le s$,
\begin{eqnarray}\label{7.1}
\int_{Q}\lf[f(x)-P^s_Q f(x)\r]x^{\tz}\,dx=0.
\end{eqnarray}
It is well known that if $f$ is locally integrable, then $P^s_Q f$
uniquely exists (see, for example, \cite{tw80}). Now, we introduce
the $\bbmo$-type space $\bmo_{\fai}(\rn)$ as follows.
\begin{definition}\label{d7.1}
Let $\fai$ be a growth function as in Definition \ref{d2.3} and
$m(\fai)$ as in \eqref{2.10}. Let $s\in\nn$ with $s\ge m(\fai)$.
When $\fai(\rn,1)=\fz$, a locally integrable function $f$ on $\rn$
is said to belong to the {\it space $\bmo_{\fai}(\rn)$}, if
\begin{eqnarray*}
\|f\|_{\bmo_{\fai}
(\rn)}&:=&\sup_{Q\subset\rn,\,|Q|<1}\frac{1}{\|\chi_{Q}
\|_{L^{\fai}(\rn)}}
\int_{Q}\lf|f(x)-P^s_Q f(x)\r|\,dx\\
&&+\sup_{Q\subset\rn,\,|Q|\ge1}\frac{1}{\|\chi_{Q}\|_{L^{\fai}(\rn)}}
\int_{Q}\lf|f(x)\r|\,dx<\fz,
\end{eqnarray*}
where the supremums are taken over all the cubes $Q\subset\rn$ with
the indicated properties, and $P^s_Q f$ as in \eqref{7.1}. When
$\fai(\rn,1)<\fz$, a function $f$ on $\rn$ is said to belong to
the {\it space $\bmo_{\fai} (\rn)$} if
\begin{eqnarray*}
\|f\|_{\bmo_{\fai} (\rn)}&:=&\sup_{Q\subset\rn,\,|Q|<1}\frac{1}
{\|\chi_{Q}\|_{L^{\fai}(\rn)}}
\int_{Q}\lf|f(x)-P^s_Q f(x)\r|\,dx\\
&&+\sup_{Q\subset\rn,\,|Q|\ge1}\frac{1}{\|\chi_{Q}\|_{L^{\fai}(\rn)}}
\int_{Q}\lf|f(x)\r|\,dx+\frac{1}{\|\chi_{\rn}\|_{L^{\fai}(\rn)}}
\int_{\rn}\lf|f(x)\r|\,dx<\fz,
\end{eqnarray*}
where the supremums are taken over all the cubes $Q\subset\rn$ with
the indicated properties, and $P^s_Q f$ as in \eqref{7.1}.
\end{definition}

When $\fai(x,t):= t$ for all $x\in\rn$ and $t\in[0,\fz)$,
$\bmo_{\fai} (\rn)$ is just $\bmo(\rn)$ introduced in \cite{go79};
when $\fai$ is as in \eqref{1.1}, $\bmo_{\fai}(\rn)$ is just
$\bmo^1_{\rho,\,\oz}(\rn)$ introduced in \cite{yys}.

Now, we show that the dual space of $h^{\fai,\,\fz,\,s}(\rn)$ is
$\bmo_{\fai}(\rn)$. We first recall the notion of the \emph{atomic
Musielak-Orlicz Hardy space $H^{\fai,\,\fz,\,s} (\rn)$},
which, when $\fai\in\aa_q(\rn)$ with $q\in(q(\fai),\fz)$, was
introduced by Ky \cite{k}.

\begin{definition}\label{d7.2}
Let $\fai$ be a growth function as in Definition \ref{d2.3} and
$m(\fai)$ as in \eqref{2.10}. Let $s\in\nn$ with $s\ge m(\fai)$. A
measurable function $a$ on $\rn$ is called an {\it
$H^{\fai,\,\fz,\,s}(\rn)$-atom} if there exists a cube $Q\subset\rn$
such that

$\mathrm{(i)}$ $\supp(a)\subset Q$;

$\mathrm{(ii)}$
$\|a\|_{L^{\fz}_{\fai}(Q)}\le\|\chi_Q\|_{L^\fai(\rn)}^{-1}$;

$\mathrm{(iii)}$ $\int_{\rn}a(x)x^{\az}\,dx=0$ for all
$\az\in\nn^n$ with $|\az|\le s$.\\
The {\it atomic Musielak-Orlicz Hardy space}
$H^{\fai,\,\fz,\,s}(\rn)$ is defined to be the set of all
$f\in\cd'(\rn)$ satisfying the fact that $f=\sum_{i=1}^{\fz}b_i$ in
$\cd'(\rn)$, where $\{b_i\}_{i=1}^\fz$ is a sequence of multiples
of $H^{\fai,\,\fz,\,s}(\rn)$-atoms with $\supp(b_i)\subset Q_i$, and
$\sum_{i=1}^{\fz}\fai(Q_i,\|b_i\|_{L^{\fz}_{\fai}(Q_i)}) <\fz.$
Moreover, letting
\begin{eqnarray*}
&&\blz_{\fz}(\{b_i\}_i):= \inf\lf\{\lz\in(0,\fz):\
\sum_{i=1}^{\fz}\fai\lf(Q_i,\frac
{\|b_i\|_{L^{\fz}_{\fai}(Q_i)}}{\lz}\r)\le1\r\},
\end{eqnarray*}
the \emph{quasi-norm} of $f\in H^{\fai,\,\fz,\,s}(\rn)$ is defined by 
$$\|f\|_{H^{\fai,\,\fz,\,s}(\rn)}:=\inf\lf\{
\blz_{\fz}\lf(\{b_i\}_{i\in\zz_+}\r)\r\},$$ 
where the infimum is taken over
all the decompositions of $f$ as above.

Furthermore, the {\it space} $H^{\fai,\,\fz,\,s}_{\fin}(\rn)$ is
defined to be the set of all finite linear combinations of
$H^{\fai,\,\fz,\,s}(\rn)$-atoms.
\end{definition}

Obviously, $H^{\fai,\,\fz,\,s}_{\fin}(\rn)$ is dense in
$H^{\fai,\,\fz,\,s}(\rn)$ with respect to the quasi-norm
$\|\cdot\|_{H^{\fai,\,\fz,\,s}(\rn)}$.

The following $\bbmo_{\fai}(\rn)$, when $\fai\in\aa_q(\rn)$ with
$q\in(q(\fai),\fz)$, was introduced by Ky \cite{k}.

\begin{definition}\label{d7.3}
Let $\fai$ be a growth function as in Definition \ref{d2.3} and
$m(\fai)$ as in \eqref{2.10}. Let $s\in\nn$ with $s\ge m(\fai)$. A
locally integrable function $f$ on $\rn$ is said to belong to the
{\it space $\bbmo_{\fai}(\rn)$} if
\begin{eqnarray*}
\|f\|_{\bbmo_{\fai}
(\rn)}&:=&\sup_{Q\subset\rn}\frac{1}{\|\chi_{Q}\|_{L^{\fai}(\rn)}}
\int_{Q}\lf|f(x)-P^s_Q f(x)\r|\,dx<\fz,
\end{eqnarray*}
where the supremum is taken over all cubes $Q\subset\rn$ and
$P^s_Q f$ as in \eqref{7.1}.
\end{definition}

The next Lemma \ref{l7.1}, when $\fai\in\aa_q(\rn)$ with
$q\in(q(\fai),\fz)$, is just \cite[Theorem 3.2]{k}, whose proof for
its current version is similar to that. We omit the details.

\begin{lemma}\label{l7.1}
Let $\fai$ be a growth function as in Definition \ref{d2.3} and
$m(\fai)$ as in \eqref{2.10}. Let $s\in\nn$ with $s\ge m(\fai)$.
Then the dual space of $H^{\fai,\,\fz,\,s}(\rn)$,
$[H^{\fai,\,\fz,\,s}(\rn)]^{\ast}$, coincides with
$\bbmo_{\fai}(\rn)$ in the following sense.

$\mathrm{(i)}$ Let $g\in \bbmo_{\fai}(\rn)$. Then the linear
functional $L$, which is initially defined on
$H^{\fai,\,\fz,\,s}_{\fin}(\rn)$ by
\begin{eqnarray}\label{7.2}
L(f)=\langle g,f\rangle,
\end{eqnarray}
has a unique extension to $H^{\fai,\,\fz,\,s}(\rn)$ with
$\|L\|_{\lf[H^{\fai,\,\fz,\,s}(\rn)\r]^{\ast}}\le
C\|g\|_{\bbmo_{\fai}(\rn)}$, where $C$ is a positive constant
independent of $g$.

$\mathrm{(ii)}$ Conversely, for any
$L\in\lf[H^{\fai,\,\fz,\,s}(\rn)\r]^{\ast}$, there exists $g\in
\bbmo_{\fai}(\rn)$ such that \eqref{7.2} holds for all $f\in
H^{\fai,\,\fai,\,s}_{\fin}(\rn)$ and $\|g\|_{\bbmo_{\fai} (\rn)}\le
C\|L\|_{\lf[H^{\fai,\,\fz,\,s}(\rn)\r]^{\ast}}$, where $C$ is a
positive constant independent of $L$.
\end{lemma}

Now, we show that the dual space of $h^{\fai,\,\fz,\,s}(\rn)$ is
$\bmo_{\fai}(\rn)$ by invoking Lemma \ref{l7.1}.

\begin{theorem}\label{t7.1}
Let $\fai$ be a growth function as in Definition \ref{d2.3} and
$m(\fai)$ as in \eqref{2.10}. Let $s\in\nn$ with $s\ge m(\fai)$.
Then the dual space of $h^{\fai,\,\fz,\,s}(\rn)$,
$\lf[h^{\fai,\,\fz,\,s}(\rn)\r]^{\ast}$, coincides with
$\bmo_{\fai}(\rn)$ in the following sense.

$\mathrm{(i)}$ Let $g\in \bmo_{\fai}(\rn)$. Then the linear
functional $L$, which is initially defined on
$h^{\fai,\,\fz,\,s}_{\fin}(\rn)$ by
\begin{eqnarray}\label{7.3}
L(f)=\langle g,f\rangle,
\end{eqnarray}
has a unique extension to $h^{\fai,\,\fz,\,s}(\rn)$ with
$\|L\|_{\lf[h^{\fai,\,\fz,\,s}(\rn)\r]^{\ast}}\le
C\|g\|_{\bmo_{\fai}(\rn)}$, where $C$ is a positive constant
independent of $g$.

$\mathrm{(ii)}$ Conversely, for any
$L\in\lf[h^{\fai,\,\fz,\,s}(\rn)\r]^{\ast}$, there exists $g\in
\bmo_{\fai}(\rn)$ such that \eqref{7.3} holds for all $f\in
h^{\fai,\,\fai,\,s}_{\fin}(\rn)$ and $\|g\|_{\bmo_{\fai} (\rn)}\le
C\|L\|_{\lf[h^{\fai,\,\fz,\,s}(\rn)\r]^{\ast}}$, where $C$ is a
positive constant independent of $L$.
\end{theorem}

\begin{proof}
The proof of $\mathrm{(i)}$ is similar to that of \cite[Theorem
3.2(i)]{k}. We omit the details.

Now, we prove $\mathrm{(ii)}$ by considering the following two
cases for $\fai$.

{\it Case} 1) $\fai(\rn,1)=\fz$. In this case, take a sequence
$\{Q_j\}_{j\in\zz_+}$ of cubes  such that for any $j\in\zz_+$,
$Q_j \subset Q_{j+1}$, $\lim_{j\to\fz}Q_j =\rn$ and
$l(Q_1)\in[1,\fz)$. Assume that
$L\in\lf[h^{\fai,\,\fz,\,s}(\rn)\r]^{\ast}$. Similar to the proof
of \cite[(7.16)]{yys}, we know that there exists a function $g$ on
$\rn$ such that for all $f\in L^{\fz}_{\fai}(Q_j)$ with
$j\in\zz_+$,
\begin{eqnarray}\label{7.4}
Lf=\int_{Q_j}f(x)g (x)\,dx.
\end{eqnarray}

Now, we show that $g\in\bmo_{\fai}(\rn)$ and for all $f\in
h^{\fai,\,\fz,\,s}_{\fin}(\rn)$,
\begin{eqnarray}\label{7.5}
Lf=\int_{\rn}f(x)g(x)\,dx.
\end{eqnarray}
Indeed, since $\fai(\rn,1)=\fz$, all $(\fai,\,\fz)$-single-atoms
are 0 and, for any $(\fai,\,\fz,\,s)$-atom $b$, there exists a
$j_0\in\zz_+$ such that $b\in L^{\fz}_{\fai}(Q_{j_0})$. By this
and the fact that \eqref{7.4} holds for all $j\in\zz_+$, we see
that \eqref{7.5} holds.

Next, we prove that $g \in\bmo_{\fai}(\rn)$. Take any cube
$Q\subset\rn$ with $l(Q)\in[1,\fz)$ as well as any $f\in
L^{\fz}_{\fai}(Q)$ with $\|f\|_{L^{\fz}_{\fai}(Q)}\le1$. Let
$b:=\|\chi_Q\|_{L^\fai(\rn)}^{-1}f \chi_Q$. Then $b$ is a
$(\fai,\,\fz,\,s)$-atom and $\supp(b)\subset Q$. From the equality
$Lb=\int_{Q}b(x)g(x)\,dx$ and $L\in
[h^{\fai,\,\fz,\,s}(\rn)]^{\ast}$, we infer that
$$|Lb|=\lf|\int_{Q}b(x)g(x)\,dx\r|\le
\|L\|_{[h^{\fai,\,\fz,\,s}(\rn)]^{\ast}}.$$ 
Thus, for any $f\in L^{\fz}_{\fai}(Q)$ with 
$\|f\|_{L^{\fz}_{\fai}(Q)}\le1$, we know that
$$\|\chi_Q\|_{L^\fai(Q)}^{-1}
\lf|\int_{Q}f(x)g(x)\,dx\r|\ls
\|L\|_{\lf[h^{\fai,\,\fz,\,s}(\rn)\r]^{\ast}}.$$ Take
$f:=\mathrm{sign}(g)$. Then
\begin{eqnarray}\label{7.6}
\frac{1}{\|\chi_Q\|_{L^\fai(Q)}} \int_{Q}|g(x)|\,dx\ls
\|L\|_{\lf[h^{\fai,\,\fz,\,s}(\rn)\r]^{\ast}}.
\end{eqnarray}
Furthermore, by $h^{\fai,\,\fz,\,s}(\rn)\supset
H^{\fai,\,\fz,\,s}(\rn)$ and $\|f\|_{h^{\fai,\,\fz,\,s}(\rn)}\le
\|f\|_{H^{\fai,\,\fz,\,s}(\rn)}$ for  all $f\in
H^{\fai,\,\fz,\,s}(\rn)$, we know that
$[h^{\fai,\,\fz,\,s}(\rn)]^{\ast}\subset[H^{\fai,\,\fz,\,s}(\rn)]
^{\ast}$ and
$L\mid_{H^{\fai,\,\fz,\,s}(\rn)}\in[H^{\fai,\,\fz,\,s}(\rn)]
^{\ast}$. Since \eqref{7.5} holds for all $f\in
h^{\fai,\,\fz,\,s}_{\fin}(\rn)$, from Lemma \ref{l7.1}(ii), we
deduce that $g\in \bbmo_{\fai}(\rn)$ and
$$\|g\|_{\bbmo_{\fai}(\rn)}\ls
\|L\mid_{H^{\fai,\,\fz,\,s}(\rn)}\|_{\lf[H^{\fai,\,\fz,\,s}(\rn)\r]^{\ast}}
\ls\|L\|_{\lf[h^{\fai,\,\fz,\,s}(\rn)\r]^{\ast}}.$$ This, together
with \eqref{7.6}, implies that $g\in\bmo_{\fai}(\rn)$ and
$\|g\|_{\bmo_{\fai}(\rn)}\ls
\|L\|_{\lf[h^{\fai,\,\fz,\,s}(\rn)\r]^{\ast}}$, which completes
the proof of (ii) in Case 1).

{\it Case} 2) $\fai(\rn,1)<\fz$. In this case, let
\begin{eqnarray*}
\widetilde{h^{\fai,\,\fz,\,s}}(\rn):=&&\lf\{f=\sum_{i=1}^{\fz}b_i
\,\,\text{in}\,\,\cd'(\rn):\ \text{For}\ i\in\zz_+,\ b_i\,\,\text{is
a multiple of some}
\r.\\
&&\hs(\fai,\,\fz,\,s)\text{-atom,}\ \supp (b_i)\subset Q_i,
\,\text{and}\ \sum_{i=1}^{\fz}\fai\lf(Q_i,
\|b_i\|_{L^{\fz}_{\fai}(Q_i)}\r)<\fz\Bigg\}
\end{eqnarray*}
and for all $f\in\widetilde{h^{\fai,\,\fz,\,s}}(\rn)$,
$\|f\|_{\widetilde{h^{\fai,\,\fz,\,s}}(\rn)}:=\inf\{\blz_{\fz}
(\{b_i\}_{i=1}^{\fz})\}$, where the infimum is taken over all the
decompositions of $f$ as above. For any $f\in L^1_{\loc}(\rn)$, let
\begin{eqnarray*}
\|f\|_{\widetilde{\bmo_{\fai}} (\rn) }&:=&
\sup_{Q\subset\rn,\,|Q|<1}\frac{1} {\|\chi_{Q}\|_{L^{\fai}(Q)}}
\int_{Q}\lf|f(x)-P^s_Q f(x)\r|\,dx\\
&&+\sup_{Q\subset\rn,\,|Q|\ge1}\frac{1}{\|\chi_{Q}\|_{L^{\fai}(Q)}}
\int_{Q}\lf|f(x)\r|\,dx,
\end{eqnarray*}
where the supremums are taken over all the cubes $Q\subset\rn$
with the indicated properties and $P^s_Q f$ as in \eqref{7.1},
and
\begin{eqnarray*}
\widetilde{\bmo_{\fai}} (\rn):=\lf\{f\in L^1_{\loc}(\rn):\
\|f\|_{\widetilde{\bmo_{\fai}}(\rn)}<\fz\r\}.
\end{eqnarray*}
Similar to the proofs of \cite[Theorem 3.2(i)]{k} and Case 1), we
have
\begin{eqnarray}\label{7.7}
\lf[\widetilde{h^{\fai,\,\fz,\,s}}(\rn)\r]^{\ast}
=\widetilde{\bmo_{\fai}} (\rn).
\end{eqnarray}

Assume that $L\in\lf[h^{\fai,\,\fz,\,s}(\rn)\r]^{\ast}$. Similar
to the proofs of \cite[(7.20)]{yys} and \cite[p.\,28]{k}, we know
that there exists  $g\in L^1(\rn)$ such that for all $f\in
L^{\fz}_{\fai}(\rn)$,
\begin{equation}\label{7.8}
 Lf=\int_{\rn}f(x)g(x)\,dx.
\end{equation}

Finally, we prove that $g\in\bmo_{\fai}(\rn)$ and
$\|g\|_{\bmo_{\fai}(\rn)}\ls\|L\|_{\lf[h^{\fai,\,\fz,\,s}(\rn)\r]^{\ast}}$.
Obviously, \eqref{7.8} holds for all $f\in
h^{\fai,\,\fz,\,s}_{\fin}(\rn)$. For any $f\in
L^{\fz}_{\fai}(\rn)$ with $\|f\|_{L^{\fz}_{\fai}(\rn)}\le1$, let
$b:=\|\chi_{\rn}\|^{-1}_{L^\fai(\rn)}f$. Then $b$ is a
$(\fai,\,\fz)$-single-atom. From \eqref{7.8} with $f:=b$ and
$L\in [h^{\fai,\,\fz,\,s}(\rn)]^{\ast}$, we deduce that
$$|Lb|=\lf|\int_{\rn}b(x)g(x)\,dx\r|\le
\|L\|_{[h^{\fai,\,\fz,\,s}(\rn)]^{\ast}},$$ 
which implies that
$$\|\chi_{\rn}\|_{L^\fai(\rn)}^{-1}
\lf|\int_{\rn}f(x)g(x)\,dx\r|\le
\|L\|_{\lf[h^{\fai,\,\fz,\,s}(\rn)\r]^{\ast}}.$$ Take
$f:=\mathrm{sign}(g)$. Then
\begin{eqnarray}\label{7.9}
\hspace{2 em}\|\chi_{\rn}\|_{L^\fai(\rn)}^{-1}\int_{\rn} |g(x)|\,dx
\le\|L\|_{\lf[h^{\fai,\,\fz,\,s}(\rn)\r]^{\ast}}.
\end{eqnarray}
Moveover, by $h^{\fai,\,\fz,\,s}(\rn)\supset
\widetilde{h^{\fai,\,\fz,\,s}}(\rn)$ and
$\|f\|_{h^{\fai,\,\fz,\,s}(\rn)}\le
\|f\|_{\widetilde{h^{\fai,\,\fz,\,s}}(\rn)}$
 for all $f\in
\widetilde{h^{\fai,\,\fz,\,s}}(\rn)$, we know that
$[h^{\fai,\,\fz,\,s}(\rn)]^{\ast}\!\subset\![\widetilde{
h^{\fai,\,\fz,\,s}}(\rn)]^{\ast}$ and
$L\!\mid_{\widetilde{h^{\fai,\,\fz,\,s}}(\rn)}
\in[\widetilde{h^{\fai,\,\fz,\,s}}(\rn)]^{\ast}$. Thus, from
\eqref{7.7} and \eqref{7.8}, we deduce that
$g\in\widetilde{\bmo_{\fai}}(\rn)$ and
$$\|g\|_{\widetilde{\bmo_{\fai}}(\rn)}\ls
\|L\mid_{\widetilde{h^{\fai,\,\fz,\,s}}(\rn)}\|_{[
\widetilde{h^{\fai,\,\fz,\,s}}(\rn)]^{\ast}}\ls
\|L\|_{\lf[h^{\fai,\,\fz,\,s}(\rn)\r]^{\ast}},$$ which, together
with \eqref{7.9}, implies that $g\in\bmo_{\fai}(\rn)$ and
$\|g\|_{\bmo_{\fai}(\rn)}\ls\|L\|_{\lf[h^{\fai,\,\fz,\,s}
(\rn)\r]^{\ast}}.$ This finishes the proof of Theorem \ref{t7.1}.
\end{proof}

From Theorems \ref{t5.1} and \ref{t7.1}, we deduce the following
conclusion.

\begin{corollary}\label{c7.1}
Let $\fai$ be a growth function as in Definition \ref{d2.3}. Then
$[h_{\fai} (\rn)]^{\ast}=\bmo_{\fai}(\rn)$.
\end{corollary}

\subsection{Characterizations of pointwise multipliers for
$\bmo^\phi(\rn)$\label{s7.2}}

\hskip\parindent We first recall some notions about pointwise
multipliers for functions of bounded mean oscillation from
\cite{ny85}.

\begin{definition}\label{d7.4}
Let $\phi$ be a positive increasing function on $\rr_+$. A locally
integrable function $f$ on $\rn$ is said to belong to the {\it space
$\bbmo^{\phi}(\rn)$}, if
\begin{eqnarray*}
\|f\|_{\bbmo^{\phi}(\rn)}:=\sup_{Q\subset\rn}
\frac{1}{\phi(l(Q))|Q|} \int_{Q}\lf|f(x)-f_Q\r|\,dx<\fz,
\end{eqnarray*}
where the supremum is taken over all the cubes $Q\subset\rn$ and
$f_Q:=\frac{1}{|Q|}\int_{Q}f(y)\,dy$. Moreover, a locally integrable
function $f$ on $\rn$ is said to belong to the {\it space
$\bmo^{\phi}(\rn)$}, if
\begin{eqnarray*}
\|f\|_{\bmo^{\phi}(\rn)}&:=&\sup_{Q\subset\rn,\,|Q|<1}
\frac{1}{\phi(l(Q))|Q|} \int_{Q}\lf|f(x)-f_Q\r|\,dx \\
&&\hs+\sup_{Q\subset\rn,\,|Q|\ge1}\frac{1}{\phi(l(Q))|Q|}
\int_{Q}\lf|f(x)\r|\,dx<\fz,
\end{eqnarray*}
where the supremums are taken over all the cubes $Q\subset\rn$ with
the indicated properties.

A function $g$ on $\rn$ is called a \emph{pointwise multiplier} on
$\bbmo^{\phi}(\rn)$ (resp. $\bmo^{\phi}(\rn)$), if the pointwise
multiplication $fg$ belongs to $\bbmo^{\phi}(\rn)$ (resp.
$\bmo^{\phi}(\rn)$) for all $f\in\bbmo^{\phi}(\rn)$ (resp.
$f\in\bmo^{\phi}(\rn)$).
\end{definition}

The following Proposition \ref{p7.1} is
just \cite[Theorem 3]{ny85}.

\begin{proposition}\label{p7.1}
Let $\phi$ be a positive increasing function on $\rr_+$ and
$\phi(r)/r$ almost decreasing. Then a function $g$ on $\rn$ is a
pointwise multiplier on $\bmo^{\phi}(\rn)$ if and only if $g\in
\bbmo^{\pz}(\rn)\cap L^{\fz}(\rn)$, where, for all
$r\in(0,\fz)$,
$$\pz(r):=\phi(r)\lf[\int^2_{\min(1,r)}\phi(t)t^{-1}\,dt\r]^{-1}.$$ 
\end{proposition}

Another main result of this section is as follows.

\begin{theorem}\label{t7.2}
Let $\phi$ and $\pz$ be as in Proposition \ref{p7.1}.

$\mathrm{(i)}$ If  $\phi$ is further assumed to satisfy the fact that for
all $t\in(0,\fz)$,
\begin{eqnarray}\label{7.10}
\phi(t)\sim\int_0^t\frac{\phi(r)}{r}\,dr,
\end{eqnarray}
then there exists an Orlicz function $\Phi_0$, which is of type
$(n/(n+1),1)$, such that the class of pointwise multipliers for
$\bmo^{\phi}(\rn)$ is the dual of $L^1(\rn)+h_{\Phi_0}(\rn)$.

$\mathrm{(ii)}$ The class of pointwise multipliers for $\bmo(\rn)$
is the dual of $L^1(\rn)+h_{\tz}(\rn)$, where $h_{\tz}(\rn)$ is
the local Musielak-Orlicz Hardy space related to the
Orlicz function $\tz(t):=\frac{t}{\ln(e+t)}$.
\end{theorem}

\begin{remark}\label{r7.1}
(i) A similar result to Theorem \ref{t7.2}(ii) for $\bbmo(\rn)$
was obtained by \cite[Theorem 3.3]{k}.

(ii) From the assumption that $\phi(r)/r$ is almost decreasing, we
infer that $\phi$ is of upper type 1. Moreover, if $\phi$ is of
lower type $p_0\in(0,1]$, then $\phi$ satisfies \eqref{7.10}.
Conversely, if $\phi$ satisfies \eqref{7.10}, then $\phi$ is of
lower type 0; however, if $\phi(r)\equiv1$ for all $r\in(0,\fz)$,
then $\phi(r)/r=1/r$ for $r\in(0,\fz)$ is decreasing and
$\phi$ is of lower type $0$, but
$\phi$ does not satisfy \eqref{7.10} and hence the assumptions of
Theorem \ref{t7.2}(i).

(iii) We point out that when $\phi\equiv1$, $\bmo^\phi(\rn)$ in
Definition \ref{d7.4} is just $\bmo(\rn)$ introduced by Goldberg
in \cite{go79}. Since the function $\phi\equiv1$ does not satisfy
the assumption in Theorem \ref{t7.1}(i), Theorem \ref{t7.2}(ii) is
not contained in Theorem \ref{t7.2}(i).
\end{remark}

To show Theorem \ref{t7.2}, we need the following lemmas.

\begin{lemma}\label{l7.2}
Let $\phi$ and $\pz$ be as in Proposition \ref{p7.1}. Then
$$\bbmo^{\pz}(\rn)\cap L^{\fz}(\rn)=\bmo^{\pz}(\rn)\cap L^{\fz}(\rn)$$
with equivalent norms.
\end{lemma}

\begin{proof}
First, we recall that for all $f\in\bbmo^{\pz}(\rn)\cap
L^{\fz}(\rn)$ (resp. $f\in\bmo^{\pz}(\rn)\cap L^{\fz}(\rn)$),
$$\|f\|_{\bbmo^{\pz}(\rn)\cap
L^{\fz}(\rn)}=\|f\|_{\bbmo^{\pz}(\rn)}+\|f\|_{L^\fz(\rn)}$$
(resp.
$\|f\|_{\bmo^{\pz}(\rn)\cap
L^{\fz}(\rn)}=\|f\|_{\bmo^{\pz}(\rn)}+\|f\|_{L^\fz(\rn)}).$ By
the definitions of $\bbmo^{\pz}(\rn)$ and $\bmo^{\pz}(\rn)$, we
know that $\bmo^{\pz}(\rn)\subset\bbmo^{\pz}(\rn)$ and for all
$f\in \bmo^{\pz}(\rn)$,
$\|f\|_{\bbmo^{\pz}(\rn)}\ls\|f\|_{\bmo^{\pz}(\rn)}$, which
implies that $\bmo^{\pz}(\rn)\cap
L^{\fz}(\rn)\subset\bbmo^{\pz}(\rn)\cap L^{\fz}(\rn)$ and for all
$f\in\bmo^{\pz}(\rn)\cap L^{\fz}(\rn)$,
$\|f\|_{\bbmo^{\pz}(\rn)\cap L^{\fz}(\rn)}\ls
\|f\|_{\bmo^{\pz}(\rn)\cap L^{\fz}(\rn)}.$

Now, we prove that $\bbmo^{\pz}(\rn)\cap
L^{\fz}(\rn)\subset\bmo^{\pz}(\rn)\cap L^{\fz}(\rn)$. Let
$f\in\bbmo^{\pz}(\rn)\cap L^{\fz}(\rn)$. Then it is easy to see that
$$\|f\|_{\bmo^{\pz}(\rn)}\ls
\|f\|_{\bbmo^{\pz}(\rn)}+\|f\|_{L^{\fz}(\rn)}
\ls\|f\|_{\bbmo^{\pz}(\rn)\cap L^{\fz}(\rn)},$$
which implies that
$f\in\bmo^{\pz}(\rn)\cap L^{\fz}(\rn)$ and
$$\|f\|_{\bmo^{\pz}(\rn)\cap
L^{\fz}(\rn)}\ls\|f\|_{\bbmo^{\pz}(\rn)\cap L^{\fz}(\rn)}.$$
This finishes the proof of Lemma \ref{l7.2}.
\end{proof}

\begin{lemma}\label{l7.3}
Let $\phi$ and $\pz$ be as in Proposition \ref{p7.1} and,
moreover, let $\phi$ satisfy \eqref{7.10}. Then there exists an Orlicz
function $\Phi_0$ such that $\Phi_0$ is of type $(n/(n+1), 1)$
and, for all $Q\subset\rn$,
$$\|\chi_{Q}\|_{L^{\Phi_0}(\rn)}=\pz(l(Q))|Q|.$$
\end{lemma}

\begin{proof}
For all $t\in(0,\fz)$, let $\eta(t):=\pz(t^{-1/n})t^{-1}$. Then by
the definition of $\pz$ and the assumption that $\phi$ is
increasing, we conclude that $\eta$ is strictly decreasing on
$(0,\fz)$. Then, the \emph{inverse} of $\eta$ exists and let
$\Phi_0(t):=\eta^{-1}(t^{-1})$ for all $t\in(0,\fz)$. Thus, for
any cube $Q\subset\rn$,
\begin{eqnarray*}
\|\chi_Q\|_{L^{\Phi_0}(\rn)}=\frac{1}{\Phi_0^{-1}(1/|Q|)}
=\eta(1/|Q|)=\pz(l(Q))|Q|.
\end{eqnarray*}
From the fact that $\eta$ is strictly decreasing and the
definition of $\Phi_0$, it follows that $\Phi_0$ is strictly
increasing on $(0,\fz)$, $\lim_{t\to0^+}\Phi_0(t)=0$ and
$\lim_{t\to\fz}\Phi_0(t)=\fz.$ Now, we prove that $\Phi_0$ is of
type $(n/(n+1),1)$. To this end, by \cite[Proposition 3.10]{vi87},
it suffices to show that $\rho_0$ is of type $(0,1/n)$, where
$\rho_0(t):=t^{-1}/\Phi^{-1}_0(t^{-1})$ for all $t\in(0,\fz)$.

First, we prove that $\rho_0$ is of upper type $1/n$. Let
$s\in[1,\fz)$ and $t\in(0,\fz)$. If $ts\ge1$ and $t\ge1$, then,
from the upper type $1$ property of $\phi$, we infer that
\begin{eqnarray}\label{7.11}
\rho_0(ts)=\phi\lf((ts)^{1/n}\r)\lf\{\int_1^2\phi(r)\,\frac{dr}{r}\r\}^{-1}\ls
s^{1/n}\rho_0(t).
\end{eqnarray}
If $ts\ge1$ and $t\in(0,1]$, then by the upper type $1$ property
of $\phi$ and \eqref{7.10}, we conclude that
\begin{eqnarray}\label{7.12}
\rho_0(ts)&&\ls s^{1/n}\phi\lf(t^{1/n}\r)\lf\{\int_1^2\phi(r)\,
\frac{dr}{r}\r\}^{-1}\nonumber\\
&&\sim\int_{t^{1/n}}^2\phi(r)\,\frac{dr}{r}
\lf\{\int_{1}^2\phi(r)\,\frac{dr}{r}\r\}^{-1}s^{1/n}
\lf\{\int_{t^{1/n}}^2\phi(r)\,\frac{dr}{r}\r\}^{-1}\nonumber\\
&&\ls\phi(2)
\lf\{\int_{1}^2\phi(r)\,\frac{dr}{r}\r\}^{-1}s^{1/n}
\lf\{\int_{t^{1/n}}^2\phi(r)\,\frac{dr}{r}\r\}\ls s^{1/n}\rho_0(t).
\end{eqnarray}
If $ts\in(0,1]$, then $t\in(0,1]$. In this case, from the upper
type $1$ property of $\phi$ and \eqref{7.10}, it follows that
\begin{eqnarray}\label{7.13}
\rho_0(ts)&&\ls s^{1/n}\phi\lf(t^{1/n}\r)
\lf\{\int_{(ts)^{1/n}}^2\phi(r)\,\frac{dr}{r}\r\}^{-1}\nonumber\\
&&\sim
\int_{t^{1/n}}^2\phi(r)\,\frac{dr}{r}
\lf\{\int_{(ts)^{1/n}}^2\phi(r)\,\frac{dr}{r}\r\}^{-1}
s^{1/n}
\lf\{\int_{t^{1/n}}^2\phi(r)\,\frac{dr}{r}\r\}^{-1}\nonumber\\
&&\ls\phi(2)
\lf\{\int_{1}^2\phi(r)\,\frac{dr}{r}\r\}^{-1}s^{1/n}
\lf\{\int_{t^{1/n}}^2\phi(r)\,\frac{dr}{r}\r\}^{-1}\ls s^{1/n}\rho_0(t).
\end{eqnarray}
Thus, by \eqref{7.11}, \eqref{7.12} and \eqref{7.13}, we know that
$\rho_0$ is of upper type $1/n$. Similar to the above argument, by
the assumption \eqref{7.10}, we know that $\phi$ is of lower type
$0$ and hence $\rho_0$ is of lower type $0$, which completes the
proof of Lemma \ref{l7.3}.
\end{proof}

\begin{lemma}\label{l7.4}
Let $\phi\equiv1$, $\pz(t):=[\int^2_{\min(1,t)}1/r\,dr]^{-1}$  and
$\tz(t):=t[\ln(e+t)]^{-1}$ for all $t\in(0,\fz)$. Then for all cubes
$Q\subset\rn$,
\begin{equation}\label{7.14}
[\pz(l(Q))]^{-1}\sim\ln\lf(e+|Q|^{-1}\r)
\end{equation}
and
\begin{equation}\label{7.15}
\|\chi_{Q}\|_{L^{\tz}(\rn)}\sim\frac{|Q|}{\ln(e+|Q|^{-1})},
\end{equation}
where the implicit constants are independent of $Q$.
\end{lemma}

\begin{proof}
We first prove \eqref{7.14}. Let $Q\subset\rn$ be a cube. If
$l(Q)\ge1$, then
$$\ln(e+|Q|^{-1})\le\ln(2e)\sim\ln2\sim[\pz(l(Q))]^{-1}.
$$
Moreover, $\ln(e+|Q|^{-1})\ge1\sim[\pz(l(Q))]^{-1}$. Thus,
$\ln(e+|Q|^{-1})\sim[\pz(l(Q))]^{-1}$ if $l(Q)\ge1$.

If $l(Q)\in(0,1]$, then
$$\ln\lf(e+|Q|^{-1}\r)\le n\ln\lf(e+[l(Q)]^{-1}\r)\le
n\ln\lf(\frac{e+1}{l(Q)}\r)\ls\ln2-\ln l(Q)
\sim[\pz(l(Q))]^{-1}.$$
Furthermore,
$$\ln\lf(e+|Q|^{-1}\r)\ge\ln\lf(2e^{1/2}[l(Q)]^{-n/2}\r)
\gs\ln2-(n/2)\ln l(Q)\sim[\pz(l(Q))]^{-1}.$$ Thus,
$\ln(e+|Q|^{-1})\sim[\pz(l(Q))]^{-1}$ also holds if
$l(Q)\in(0,1]$. By the above estimates, we see that
$\ln(e+|Q|^{-1})\sim[\pz(l(Q))]^{-1}$ for any cube $Q\subset\rn$,
which prove \eqref{7.14}.

Now, we prove \eqref{7.15}. For any cube $Q\subset\rn$, let
$A_Q:=\frac{|Q|}{\ln(e+|Q|^{-1})}$. Then
\begin{eqnarray}\label{7.16}
\int_Q\tz(x,1/A_Q)\,dx=\frac{\ln(e+|Q|^{-1})}
{\ln(e+|Q|^{-1}\ln(e+|Q|^{-1}))}.
\end{eqnarray}
Obviously, $\ln(e+|Q|^{-1}\ln(e+|Q|^{-1}))\ge\ln(e+|Q|^{-1})$ for
all $Q\subset\rn$, which implies that for all $Q\subset\rn$,
$$\frac{\ln(e+|Q|^{-1})} {\ln(e+|Q|^{-1}\ln(e+|Q|^{-1}))}\le1.$$
By this, \eqref{7.16} and the definition of
$\|\chi_Q\|_{L^\fai(\rn)}$, we conclude that for all
$Q\subset\rn$,
\begin{eqnarray}\label{7.17}
\|\chi_{Q}\|_{L^{\tz}(\rn)}\ls A_Q.
\end{eqnarray}

Now, we prove $\|\chi_{Q}\|_{L^{\tz}(\rn)}\gs A_Q$. If
$l(Q)\in[1,\fz)$, then,
$$\ln(e+|Q|^{-1}\ln(e+|Q|^{-1}))\le\ln(e+2|Q|^{-1})
\ls\ln(e+|Q|^{-1}),$$ which implies that $\frac{\ln(e+|Q|^{-1})}
{\ln(e+|Q|^{-1}\ln(e+|Q|^{-1}))}\gs1$. From this and \eqref{7.16},
we infer that
\begin{eqnarray}\label{7.18}
\int_Q\tz(x,1/A_Q)\,dx\gs1.
\end{eqnarray}
If $l(Q)\in(0,1]$, by $\ln(e+|Q|^{-1})>1$, we know that
$e+\frac{\ln(e+|Q|^{-1})}{|Q|}\le(e+|Q|^{-1})\ln(e+|Q|^{-1})$,
which implies that
$$\ln\lf(e+|Q|^{-1}\ln(e+|Q|^{-1})\r)\le\ln(e+|Q|^{-1})
+\ln\lf(\ln(e+|Q|^{-1})\r).$$ 
By this and the fact that $\ln x<x$ for
all $x\in(1,\fz)$, we conclude that
$\ln(e+|Q|^{-1}\ln(e+|Q|^{-1}))<2\ln(e+|Q|^{-1})$, which, together
with \eqref{7.16}, implies that
$\int_Q\tz(x,\frac{1}{A_Q})\,dx>\frac{1}{2}$. From this and
\eqref{7.18}, we infer that for all cubes $Q\subset\rn$,
$\|\chi_{Q}\|_{L^{\tz}(\rn)}\gs A_Q$, which, together with
\eqref{7.17}, implies that \eqref{7.15} holds. This finishes the
proof of Lemma \ref{l7.4}.
\end{proof}

Now, we show Theorem \ref{t7.2} by using Lemmas \ref{l7.3} and
\ref{l7.4}.

\begin{proof}[Proof of Theorem \ref{t7.2}]
We first prove (i). Let $\Phi_0$ be as in Lemma \ref{l7.3}. Then
$\Phi_0$ is an Orlicz function of type $(n/(n+1),1)$ and for all
cubes $Q$, $\|\chi_{Q}\|_{L^{\Phi_0}(\rn)}=\pz(l(Q))|Q|$, which
implies that
\begin{eqnarray}\label{7.19}
\bmo_{\Phi_0}(\rn)=\bmo^{\pz}(\rn)
\end{eqnarray}
with equivalent norms. Furthermore, by a standard argument, we know
that
$$[L^1(\rn)+h_{\Phi_0}(\rn)]^{\ast}=
[L^1(\rn)]^{\ast}\cap[h_{\Phi_0}(\rn)]^{\ast},
$$
which, together with Corollary \ref{c7.1}, \eqref{7.19} and the
well-known fact that $[L^1(\rn)]^{\ast}=L^{\fz}(\rn)$, implies
that $[L^1(\rn)+h_{\Phi_0}(\rn)]^{\ast}=
L^{\fz}(\rn)\cap\bmo^{\pz}(\rn).$ From this and Proposition
\ref{p7.1}, we deduce (i).

Now, we prove (ii). By \eqref{7.14} and \eqref{7.15}, we conclude
that for all cubes $Q\subset\rn$,
$\|\chi_Q\|_{L^{\tz}(\rn)}\sim\pz(l(Q))|Q|$, which implies that
$\bmo_{\tz}(\rn)=\bmo^{\pz}(\rn)$. The remainder of the proof is
similar to that of (i). We omit the details here. This finishes the
proof of (ii) and hence Theorem \ref{t7.2}.
\end{proof}

\section{Boundedness of local Riesz transforms and
$S^0_{1,\,0}(\rn)$ pseudo-differential operators on $h_{\fai}(\rn)$\label{s8}}

\hskip\parindent In Subsection \ref{s8.1}, we show that local
Riesz transforms are bounded on $h_{\fai}(\rn)$, where
$\fai\in\aa^\loc_q(\rn)$ with $q\in [1,\fz)$. Subsection
\ref{s8.2} is devoted to the boundedness of $S^0_{1,\,0}(\rn)$
pseudo-differential operators on $h_{\fai}(\rn)$, where $\fai\in
\aa_p (\phi_\az)$ for some $\az\in (0,\fz)$,
which is contained in $\aa^{\loc}_{p}(\rn)$ (see
Definition \ref{d8.3} below), with $p\in[1,\fz)$ and $\phi_\az$ as in
\eqref{8.7} below.

\subsection{Local Riesz transforms\label{s8.1}}

\hskip\parindent Now, we recall the notion of local Riesz transforms introduced by
Goldberg \cite{go79}. In what follows, $\cs(\rn)$ denotes the
\emph{space of all Schwartz functions on $\rn$}.

\begin{definition}\label{d8.1}
Let $\phi_0\in\cd(\rn)$ be such that $\phi_0\equiv1$ on $Q(0,1)$
and $\supp(\phi_0)\subset Q(0,2)$. For $j\in\{1,\,\cdots,\,n\}$
and $x\in\rn$, let 
$$k_j (x):=\frac{x_j}{|x|^{n+1}}\phi_0(x).$$ 
For $f\in\cs(\rn)$, the \emph{local Riesz transform $r_j (f)$} of $f$
is defined by $r_j (f):= k_j \ast f$.
\end{definition}

We remark that $\phi_0$ in \cite{go79} was assumed that
$\phi_0\equiv1$ in a neighborhood of the origin and
$\phi_0\in\cd(\rn)$. In this paper, for the sake of convenience, we assume
$\phi_0\equiv1$ on $Q(0,1)$ and $\supp(\phi_0)\subset Q(0,2)$. We
have the boundedness on $h_{\fai}(\rn)$ of local Riesz transforms
$\{r_j\}_{j=1}^n$ as follows.

\begin{theorem}\label{t8.1}
Let $\fai$ be a growth function as in Definition \ref{d2.3}. For
$j\in\{1,\,\cdots,\,n\}$, let $r_j$ be the local Riesz operator as
in Definition \ref{d8.1}. Then there exists a positive constant
$C(\fai,\,n)$, depending only on $\fai$ and $n$, such that for all
$f\in h_{\fai}(\rn)$,  
$$\lf\|r_j (f)\r\|_{h_{\fai}(\rn)}\le
C(\fai,\,n) \|f\|_{h_{\fai}(\rn)}.$$
\end{theorem}

To prove Theorem \ref{t8.1}, we need the following lemma established
in \cite[Lemma\,3.2]{ta1}.

\begin{lemma}\label{l8.1}
For $j\in\{1,\,\cdots,\,n\}$, let $r_j$ be the local Riesz
operator as in Definition \ref{d8.1}. For $\fai\in \aa^{\loc}_p
(\rn)$ with $p\in(1,\,\fz)$, there exists a positive constant
$C(p,\,\fai,\,n)$, depending only on $p$, $\fai$ and $n$, such
that for all $f\in L^p_{\fai(\cdot,1)}(\rn)$, 
$$\lf\|r_j(f)\r\|_{L^p_{\fai(\cdot,1)}(\rn)}\le
C(p,\,\fai,\,n)\|f\|_{L^p_{\fai(\cdot,1)}(\rn)}.$$
\end{lemma}

Now, we prove Theorem \ref{t8.1} by using Lemma \ref{l8.1}.

\begin{proof}[Proof of Theorem \ref{t8.1}]
Let $s:=\lfz n[q(\fai)/i(\fai)-1]\rfz$, where $q(\fai)$ and
$i(\fai)$ are respectively as in \eqref{2.9} and \eqref{2.3}. Then
$(n+s+1)i(\fai)>nq(\fai)$, which implies that there exists
$q\in(q(\fai),\fz)$ and $p_0\in(0,i(\fai))$ such that
$(n+s+1)p_0>nq$, $\fai$ is of uniformly lower type $p_0$ and
$\fai\in \aa^{\loc}_q(\rn)$. We first assume that $f\in
h_{\fai}(\rn)\cap L^q_{\fai(\cdot,1)}(\rn)$. To show Theorem
\ref{t8.1}, it suffices to show that for any multiple of some
$(\fai,\,q)$-single-atom $b$,
\begin{eqnarray}\label{8.1}
\int_{\rn}\fai\lf(x,\cg_N^0(r_j(b))(x)\r)\,dx\ls
\fai\lf(\rn,\|b\|_{L^q_{\fai}(\rn)}\r),
\end{eqnarray}
and for any multiple of $(\fai,\,q,\,s)$-atom $b$ supported in
$Q:=Q(x_0,\,R_0)$ with $R_0\in(0,2]$,
\begin{eqnarray}\label{8.2}
\int_{\rn}\fai\lf(x,\cg_N^0(r_j(b))(x)\r)\,dx\ls
\fai\lf(Q,\|b\|_{L^q_{\fai}(Q)}\r).
\end{eqnarray}
Indeed, if \eqref{8.1} and \eqref{8.2} hold, by Lemma \ref{l5.5},
Remark \ref{r5.1} and the claim, proved in the proof of Theorem
\ref{t6.1}(i), that \eqref{6.3} holds in
$L^q_{\fai(\cdot,1)}(\rn)$, we know that there exist a sequence $\{b_i\}_i$ of
the multiple of $(\fai,\,q,\,s)$-atoms, respectively supported in
$\{Q_j\}_j$ with $l(Q_j)\in(0,2]$, and a multiple of some
$(\fai,\,q)$-single-atom $b_0$ such that $f=b_0+\sum_i b_i$ in
$L^q_{\fai(\cdot,1)}(\rn)$ and
$\|f\|_{h^{\fai}(\rn)}\sim\blz_q(\{b_i\}_i)$, which, together with
Lemma \ref{l8.1}, implies that $r_j(f)=r_j(b_0)+\sum_ir_j(b_i)$ in
$L^q_{\fai(\cdot,1)}(\rn)$. From this, we deduce that
$$\cg^0_N(r_j(f))\le\cg^0_N(r_j(b_0))+\sum_i\cg^0_N(r_j(b_i)),
$$
which, together with \eqref{8.1} and \eqref{8.2}, implies that for
all $\lz\in(0,\fz)$,
\begin{eqnarray*}
\int_{\rn}\fai\lf(x,\frac{\cg^0_N(r_j(f))(x)}{\lz}\r)\,dx
&&\ls\sum_{i=0}^{\fz}\int_{\rn}\fai\lf(x,\frac{\cg^0_N(r_j(b_i))(x)}
{\lz}\r)\,dx\\
&&\ls\fai\lf(\rn,\frac{\|b_0\|_{L^q_{\fai}(\rn)}}{\lz}\r)
+\sum_{i}\fai\lf(Q_i,\frac{\|b_i\|_{L^q_{\fai}(\rn)}}{\lz}\r).
\end{eqnarray*}
By this, we conclude that
$\|r_j(f)\|_{h_{\fai}(\rn)}\ls\|f\|_{h_{\fai}(\rn)}$. Since
$h_{\fai}(\rn)\cap L^q_{\fai(\cdot,1)}(\rn)$ is dense in
$h_{\fai}(\rn)$, a density argument then gives the desired
conclusion.

We first prove \eqref{8.1}. In this case, from the uniformly upper
type 1 property of $\fai$, H\"older's inequality, Lemma
\ref{l2.4}(iii), the fact that $\cg^0_N(r_j(b))\ls M^{\loc}(r_j(b))$
and Lemma \ref{l8.1}, we infer that
\begin{eqnarray}\label{8.3}
&&\int_{\rn}\fai\lf(x,\cg^0_{N}(r_j(b))(x)\r)\,dx\nonumber\\
&&\hs\le\int_{\rn}\lf(1+\frac{\cg_N^0(r_j(b))(x)}
{\|b\|_{L^q_{\fai}(\rn)}}\r)
\fai\lf(x,\|b\|_{L^q_{\fai}(\rn)}\r)\,dx\nonumber\\
&&\hs\le\fai\lf(\rn,\|b\|_{L^q_{\fai}(\rn)}\r)+
\frac{1}{\|b\|_{L^q_{\fai}(\rn)}}\lf\{\int_{\rn}
\lf|\cg^0_N(r_j(b))(x)\r|^q\fai\lf(x,\|b\|_{L^q_{\fai}(\rn)}\r)\,dx\r\}^{1/q}
\nonumber \\
&&\hs\hs\times
\lf[\fai\lf(\rn,\|b\|_{L^q_{\fai}(\rn)}\r)\r]^{(q-1)/q}
\ls\fai\lf(\rn,\|b\|_{L^q_{\fai}(\rn)}\r),
\end{eqnarray}
which implies \eqref{8.1}.

Now, we prove \eqref{8.2} for $b$ by considering the following
two cases for $R_0$.

{\it Case} 1) $R_0 \in[1,2]$. In this case, by the definitions of
$r_j (b)$ and $\cg^0_N (r_j (b))$, we see that
$$\supp\lf(\cg^0_N (r_j (b))\r)\subset
Q^{\ast}:= Q(x_0, R_0 +8).$$ From this, the uniformly upper type 1
property of $\fai$, H\"older's inequality, (ii) and (iii) of Lemma
\ref{l2.4}, we infer that
\begin{eqnarray*}
&&\int_{\rn}\fai\lf(x,\cg^0_N(r_j(b))(x)\r)\,dx\\
&&\hs=\int_{Q^{\ast}}\fai\lf(x,\cg^0_N(r_j(b))(x)\r)\,dx\\
&&\hs\le\int_{Q^{\ast}}\lf(1+\frac{\cg_N^0(r_j(b))(x)}
{\|b\|_{L^q_{\fai}(Q)}}\r)\fai\lf(x,\|b\|_{L^q_{\fai}(Q)}\r)\,dx\\
&&\hs\le\fai\lf(Q^{\ast},\|b\|_{L^q_{\fai}(Q)}\r)+
\frac{1}{\|b\|_{L^q_{\fai}(Q)}}\lf\{\int_{Q^{\ast}}
|\cg^0_N(r_j(b))(x)|^q\fai(x,\|b\|_{L^q_{\fai}(Q)})\,dx\r\}^{1/q}\\
\nonumber &&\hs\hs\times
\lf[\fai\lf(Q^{\ast},\|b\|_{L^q_{\fai}(Q)}\r)\r]^{(q-1)/q}
\ls\fai\lf(Q^{\ast},\|b\|_{L^q_{\fai}(Q)}\r)\ls
\fai\lf(Q,\|b\|_{L^q_{\fai}(Q)}\r),
\end{eqnarray*}
which implies \eqref{8.2} in Case 1).

{\it Case} 2) $R_0\in(0,1)$. In this case, let $\wz Q:=8nQ$.
Then
\begin{eqnarray}\label{8.4}
\qquad\int_{\rn}\fai\lf(x,\cg^0_N(r_j(b))(x)\r)\,dx
=\int_{\wz{Q}}\fai\lf(x,\cg^0_N(r_j(b))(x)\r)\,dx
+\int_{(\wz{Q})^\complement}\cdots
=:\mathrm{I_1}+\mathrm{I_2}.\qquad
\end{eqnarray}

For $\mathrm{I_1}$, similar to the proof of Case 1), we know that
\begin{eqnarray}\label{8.5}
\mathrm{I_1}&\le&\fai\lf(Q,\|b\|_{L^q_{\fai}(Q)}\r).
\end{eqnarray}

Now, we estimate $\mathrm{I}_2$. Similar to the proof of
\cite[(8.10)]{yys}, we see that for all
$x\in(\wz{Q})^\complement$,
\begin{eqnarray}\label{8.6}
\cg^0_N \lf(r_j (b)\r)(x)\ls
\frac{R_0^{n+s+1-\dz}}{|x-x_0|^{n+s+1-\dz}}\|b\|_{L^q_{\fai}(Q)},
\end{eqnarray}
where $\dz$ is a  positive constant small enough such that
$p_0(n+s+1-\dz)>nq$. By the fact that $\supp(\cg^0_N (r_j
(b)))\subset Q(x_0, R_0+8)\subset Q(x_0,9)$ and Lemma
\ref{l2.4}(i), we conclude that there exists a $\wz{\fai}\in \aa_q
(\rn)$ such that $\wz{\fai}(\cdot,t)=\fai(\cdot,t)$ on $Q(x_0,9)$
for all $t\in(0,\fz)$. Let $m_0$ be the integer such that $2^{m_0
-1}nR_0\le9<2^{m_0}nR_0$. From \eqref{8.6}, the uniformly lower
type $p_0$ property of $\fai$, Lemma \ref{l2.4}(iii) and
$p_0(n+s+1-\dz)>nq$, we infer that
\begin{eqnarray*}
\mathrm{I_2}&\ls&\int_{Q(x_0,9)\setminus\wz{Q}}\fai\lf(x,\cg^0_N
(r_j (b))(x)\r)\,dx\\
&\ls&\sum_{j=3}^{m_0}\int_{2^{j+1}nQ\setminus2^j n
 Q}\wz\fai\lf(x,
\frac{R_0^{n+s+1-\dz}}{|x-x_0|^{n+s+1-\dz}}\|b\|_{L^q_{\fai}(Q)}\r)\,dx\\
&\ls&\sum_{j=3}^{m_0}\int_{2^{j+1}nQ\setminus2^j
nQ}\lf(\frac{R_0^{n+s+1-\dz}}{|x-x_0|^{n+s+1-\dz}}\r)
^{p_0}\wz{\fai}\lf(x,\|b\|_{L^q_{\fai}(Q)}\r)\,dx\\
&\ls&\sum_{j=3}^{m_0}2^{k[(n+s+1-\dz)p_0-nq]}
\fai\lf(Q,\|b\|_{L^q_{\fai}(Q)}\r)\ls\fai\lf(Q,\|b\|_{L^q_{\fai}(Q)}\r),
\end{eqnarray*}
which, together with \eqref{8.4} and \eqref{8.5}, implies
\eqref{8.2} in Case 2). This finishes the proof of Theorem
\ref{t8.1}.
\end{proof}

\begin{remark}\label{r8.1}
We remark that Theorem \ref{t8.1} completely covers
\cite[Theorem\,8.2]{yys} by taking $\fai$ as in \eqref{1.1}, where
$\Phi$ is further assumed to satisfy the fact that $p_{\bfai}=p_{\bfai}^+$
and $p_{\Phi}^+$ is the attainable critical upper type index of
$\Phi$.
\end{remark}

\subsection{$S^0_{1,\,0}(\rn)$ pseudo-differential operators\label{s8.2}}

\hskip\parindent The pseudo-differential operators have been extensively studied in
the literature, and they are important in the study of partial
differential equations and harmonic analysis (see, for example,
\cite{sc99,st93,tay91,ta2}). Now, we recall the notion of
pseudo-differential operators.

\begin{definition}\label{d8.2}
A {\it  symbol} in $S^0_{1,0}(\rn)$ is a smooth function
$\sz(x,\xi)$ defined on $\rn\times\rn$ such that for all
multi-indices $\az$ and $\bz$, the following estimate holds:
$$\lf|\partial^{\az}_x \partial^{\bz}_{\xi}\sz(x,\xi)\r|\le
C({\az,\,\bz})(1+|\xi|)^{-|\bz|},$$ where $C(\az,\,\bz)$ is a
positive constant independent of $x$ and $\xi$. Let $f\in\cs(\rn)$
and $\widehat{f}$ be its \emph{Fourier transform}. The operator $T$ given by
setting, for all $x\in\rn$,
$$Tf(x):=\int_{\rn}\sz(x,\xi)e^{2\pi ix\xi}\hat{f}(\xi)\,d\xi$$
is called an \emph{$S^0_{1,0}(\rn)$ pseudo-differential
operator}.
\end{definition}

In the remainder of this section, for any given $\az\in(0,\fz)$
and all $t\in(0,\fz)$, let
\begin{eqnarray}\label{8.7}
\phi_\az(t):=(1+t)^{\az}.
\end{eqnarray}
Recall that a weight always means a locally
integrable function which is positive almost everywhere.
\begin{definition}\label{d8.3}
Let $\fai:\rn\times[0,\fz)\to[0,\fz)$ be a uniformly locally
integrable function and $\az\in(0,\fz)$. The function
$\fai(\cdot,t)$ is said to satisfy the \emph{uniformly local
weight condition $\aa_q(\phi_\az)$ with $q\in[1,\fz)$} if there
exists a positive constant $C(\az)$, depending on $\az$, such that
for all cubes $Q:= Q(x, r)$ and $t\in[0,\fz)$,
$$\lf[\frac{1}{\phi_\az(|Q|)|Q|}\int_{Q}\fai(x,t)\,dx\r]
\lf(\frac{1}{\phi_\az(|Q|)|Q|}\int_{Q}[\fai(x,t)]^{-\frac{1}{p-1}}\,dx
\r)^{p-1}\le C(\az)$$ when $q\in(1,\fz)$, and
$$\frac{1}{\phi_\az(|Q|)|Q|}\int_Q \fai(x,t)\,dx \lf(\esup_{y\in
Q}[\fai(y,t)]^{-1}\r)\le C(\az)$$ when $q=1$.
\end{definition}

We point out that $\aa_q(\phi_\az)$ with $q\in[1,\fz)$ when $\fai$
is as in \eqref{1.1} was introduced by Tang \cite{ta1}.

Similar to the classical Muckenhoupt weights, we have the
following properties of $\fai\in \aa_{\fz}(\phi_\az):=\cup_{1\le
p<\fz} \aa_p(\phi_\az)$, whose proofs are similar to those of
\cite[Lemmas 7.3 and 7.4]{ta1}. We omit the details.

\begin{lemma}\label{l8.2} Let $\az\in (0,\fz)$.

$\mathrm{(i)}$ If $1\le p_1 <p_2<\fz$, then
$\aa_{p_1}(\phi_\az)\subset \aa_{p_2}(\phi_\az)$.

$\mathrm{(ii)}$ For $p\in(1,\fz)$, $\fai\in \aa_{p}(\phi_\az)$ if
and only if $\fai^{-\frac{1}{p-1}}\in \aa_{p'}(\phi_\az)$, where
$1/p +1/p' =1$.

$\mathrm{(iii)}$ If $\fai\in \aa_p (\phi_\az)$ for $p\in[1,\fz)$,
then there exists a positive constant $C$ such that for all
$t\in(0,\fz)$, cubes $Q\subset\rn$ and measurable sets $E\subset Q$, 
$$\frac{|E|}{\phi_\az(|Q|)|Q|}\le
C\lf[\frac{\fai(E,t)}{\fai(Q,t)}\r]^{1/p}.$$
\end{lemma}

\begin{lemma}\label{l8.3}
Let $T$ be an $S^0_{1,\,0}(\rn)$ pseudo-differential operator and
$\az\in(0,\fz)$. Then for $\fai\in \aa_p (\phi_\az)$ with
$p\in(1,\fz)$, there exists a positive constant
$C(p,\,\fai,\,\az)$, depending on $p,\,\fai$ and $\az$, such that
for all $f\in L^p_{\fai(\cdot,1)}(\rn)$,
$$\|Tf\|_{L^p_{\fai(\cdot,1)}(\rn)}\le
C(p,\,\fai,\,\az)\|f\|_{L^p_{\fai(\cdot,1)}(\rn)}.$$
\end{lemma}

Now, we establish the boundedness on $h_{\fai}(\rn)$ of
$S^0_{1,\,0}(\rn)$ pseudo-differential operators as follows.

\begin{theorem}\label{t8.2}
Let $T$ be an $S^0_{1,\,0}(\rn)$ pseudo-differential operator,
$\fai$ a growth function as in Definition \ref{d2.3} and $\fai\in
\aa_{\fz}(\phi_\az)$ for some $\az\in(0,\fz)$. Then, there exists a
positive constant $C(\fai,\,n,\,\az)$, depending only on $\fai$,
$n$ and $\az$, such that for all $f\in h_{\fai}(\rn)$,
$$\|Tf\|_{h_{\fai}(\rn)}\le C(\fai,\,n,\,\az)
\|f\|_{h_{\fai}(\rn)}.$$
\end{theorem}

\begin{proof}
Let $q\in(q(\fai),\fz)$, $\az\in(0,\fz)$ and the nonnegative
integer $s$ satisfy $(n+s+1)i(\fai)>nq(1+\az)$, where $i(\fai)$ is
as in \eqref{2.3}. Then, by \eqref{2.3}, we further see that there
exists $p_0\in(0,i(\fai))$ such that $(n+s+1)p_0>nq(1+\az)$. To
show Theorem \ref{t8.2}, similar to the proof of Theorem
\ref{t8.1}, replacing Lemma \ref{l8.1} by Lemma \ref{l8.3} in the
proof of Theorem \ref{t8.1}, we see that it suffices to show that
for any multiple of some $(\fai,\,q)$-single-atom $b$,
\begin{eqnarray}\label{8.8}
\int_{\rn}\fai\lf(x,\cg_N^0(T(b))(x)\r)\,dx\ls
\fai\lf(\rn,\|b\|_{L^q_{\fai}(\rn)}\r),
\end{eqnarray}
and for any multiple of $(\fai,\,q,\,s)$-atom $b$ supported in
$Q_0:=Q(x_0,\,R_0)$ with $R_0\in(0,2]$,
\begin{eqnarray}\label{8.9}
\int_{\rn}\fai\lf(x,\cg_N^0(T(b))(x)\r)\,dx\ls
\fai\lf(Q_0,\|b\|_{L^q_{\fai}(Q_0)}\r).
\end{eqnarray}

The proof of \eqref{8.8} is similar to that of \eqref{8.3}. We omit the
details.

Now, we prove \eqref{8.9}. Let $\wz{Q}_0:=2Q_0$. Similar to the
proof of \eqref{8.5}, we see that
\begin{eqnarray}\label{8.10}
&&\int_{\wz{Q}_0}\fai\lf(x,\cg^0_N(T(b))(x)\r)\,dx\ls
\fai\lf(Q_0,\|b\|_{L^q_{\fai}(Q_0)}\r).
\end{eqnarray}

To estimate $\int_{\rn\setminus \wz{Q}_0}\fai(x,\cg^0_N(T
b)(x))\,dx$, we consider the following two cases for $R_0$.

{\it Case} 1) $R_0\in(0,1)$.  In this case, similar to the proof of
\cite[p.\,74]{yys}, we see that for all $x\in \wz{Q}_0^\complement$,
\begin{eqnarray*}
\cg_N^0(T(b))(x)\ls|
x-x_0|^{-(n+s+1)}|Q_0|^{\frac{s+1}{n}}\|b\|_{L^1 (\rn)}.
\end{eqnarray*}
This, combined with H\"older's inequality, Lemma \ref{l8.2}(iii)
and the definition of $\aa_p (\phi_\az)$, implies that
\begin{eqnarray}\label{8.11}
&&\hs\int_{\rn\setminus \wz{Q}_0}\fai\lf(x,\cg^0_N(T(b))(x)\r)\,dx\nonumber\\
 &&\hs\hs\ls\int_{\rn\setminus \wz{Q}_0}\fai\lf(x,|
x-x_0|^{-(n+s+1)}|Q_0|^{\frac{s+1}{n}}\|b\|_{L^1 (\rn)}\r)\,dx\nonumber\\
 &&\hs\hs\ls\int_{\rn\setminus
\wz{Q}_0}\fai\lf(x,|Q_0|^{\frac{s+1}{n}}
|x-x_0|^{-(n+s+1)}\|b\|_{L^q_{\fai}(\rn)}\phi_{\az}(|Q_0|)
|Q_0|\r)\,dx\nonumber\\
 &&\hs\hs\ls\sum_{k=1}^{\fz}\int_{2^k
Q_0}\fai\lf(x,|Q_0|^{\frac{s+1}{n}} (2^k
R_0)^{-(n+s+1)}\|b\|_{L^q_{\fai}(\rn)}\phi_{\az}(|Q_0|)
|Q_0|\r)\,dx\nonumber\\
 &&\hs\hs\ls\sum_{k=1}^{m_0}\int_{2^k
Q_0}\fai\lf(x,|Q_0|^{\frac{s+1}{n}} (2^k
R_0)^{-(n+s+1)}\|b\|_{L^q_{\fai}(\rn)}
\phi_{\az}(|Q_0|) |Q_0|\r)\,dx\nonumber\\
 &&\hs\hs\hs +\sum_{k=m_0 +1}^{\fz}\int_{2^k
Q_0}\cdots=:\mathrm{I_1}+\mathrm{I_2},
\end{eqnarray}
where the integer $m_0$ satisfies $2^{m_0
-1}\le\frac{1}{R_0}<2^{m_0}$.

To estimate $\mathrm{I_1}$, for any $k\in\{1,\,\cdots,\,m_0\}$, by
the choice of $m_0$ and $R_0\in(0,1)$, we know that $2^k
R_0^n\le1$, which, together with the uniformly lower type $p_0$
property of $\fai$, Lemma \ref{l8.2}(iii) and the fact that
$(n+s+1)p_0>nq(1+\az)$, implies that
\begin{eqnarray}\label{8.12}
\mathrm{I_1}&\ls&\sum_{k=1}^{m_0}\fai\lf(2^kQ_0,|Q_0|^{\frac{s+n+1}{n}}
(2^k
R_0)^{-(n+s+1)}\|b\|_{L^q_{\fai}(\rn)}\r)\nonumber\\
 &\ls&\sum_{k=1}^{m_0}2^{-k(n+s+1)p_0}2^{knq}\fai\lf(Q_0,
\|b\|_{L^q_{\fai}(\rn)}\r)\ls\fai\lf(Q_0,\|b\|_{L^q_{\fai}(\rn)}\r).
\end{eqnarray}
For $\mathrm{I_2}$, similar to the estimate of $\mathrm{I_1}$, we
know that
\begin{eqnarray*}
\mathrm{I_2}&\ls&\sum_{k=m_0
+1}^{\fz}\fai\lf(2^kQ_0,|Q_0|^{\frac{s+n+1}{n}} (2^k
R_0)^{-(n+s+1)}\|b\|_{L^q_{\fai}(\rn)}\r)\\
&\ls&\sum_{k=m_0+1}^{\fz}2^{-k(n+s+1)p_0}\fai\lf(2^kQ_0,
\|b\|_{L^q_{\fai}(\rn)}\r)\\
&\ls&\sum_{k=m_0
+1}^{\fz}2^{-k(n+s+1)p_0}2^{knq}\lf[\phi_{\az}(|2^k Q_0|)\r]^q
\fai\lf(Q_0,\|b\|_{L^q_{\fai}(\rn)}\r)\\
&\ls&\sum_{k=m_0 +1}^{\fz}2^{-k[(n+s+1)p_0-nq(\az+1)]}
\fai\lf(Q_0,\|b\|_{L^q_{\fai}(\rn)}\r)\ls
\fai\lf(Q_0,\|b\|_{L^q_{\fai}(\rn)}\r),
\end{eqnarray*}
which, together with \eqref{8.10}, \eqref{8.11} and \eqref{8.12},
implies \eqref{8.9} in Case 1).

{\it Case} 2) $R_0\in[1,2]$. In this case, similar to the proof of
\cite[(8.44)]{yys}, we know that for all $x\in
\wz{Q}_0^\complement$,
\begin{eqnarray}\label{8.13}
\cg^0_N (T(b))(x)\ls|x-x_0|^{-M} \|b\|_{L^1 (\rn)}.
\end{eqnarray}
Take $M>\frac{nq(1+\az)}{p_0}$. By \eqref{8.13}, H\"older's
inequality and Lemma \ref{l8.2}(iii), we conclude that
\begin{eqnarray*}
&&\int_{\rn\setminus \wz{Q}_0}\fai\lf(x,\cg^0_N(T(b))(x)\r)\,dx\\
 &&\hs\ls\int_{\rn\setminus \wz{Q}_0}\fai\lf(x,|
x-x_0|^{-M}\|b\|_{L^1 (\rn)}\r)\,dx\\
&&\hs\ls\sum_{k=1}^{\fz}\int_{2^k Q_0}\fai\lf(x,(2^k
R_0)^{-M}\|b\|_{L^q_{\fai}(\rn)}\phi_{\az}(|Q_0|)|Q_0|\r)\,dx\\
&&\hs\ls\sum_{k=1}^{\fz}2^{-kMp_0}R_0^{-Mp_0}\fai(2^kQ_0,
\|b\|_{L^q_{\fai}(\rn)})\\
&&\hs \ls\sum_{k=1}^{\fz}2^{-k(Mp_0-nq)}\phi_{\az}(|2^k
Q_0|)R_0^{-Mp_0}\fai(Q_0,
\|b\|_{L^q_{\fai}(\rn)})\\
&&\hs\ls\sum_{k=1}^{\fz}2^{-k[Mp_0-nq(1+\az)]}
R_0^{-(Mp_0-nq\az)}\fai\lf(Q_0,\|b\|_{L^q_{\fai}(\rn)}\r)
\ls\fai\lf(Q_0,\|b\|_{L^q_{\fai}(\rn)}\r),
\end{eqnarray*}
which, together with \eqref{8.10}, implies \eqref{8.9} in Case 2).
This finishes the proof of Theorem \ref{t8.2}.
\end{proof}

By Theorems \ref{t8.2} and \ref{t7.1}, \cite[p.\,233,\,(4)]{st93}
and the proposition in \cite[p.\,259]{st93}, we have the following
result.

\begin{corollary}\label{c8.1}
Let $T$ be an $S^0_{1,\,0}(\rn)$ pseudo-differential operator and
$\fai$ a growth function satisfying $\fai\in \aa_{\fz}(\phi_\az)$ for
some $\az\in(0,\fz)$. Then there exists a positive constant $C
(\fai,\,\az)$, depending on $\fai$ and $\az$, such that for all
$f\in\bmo_{\fai}(\rn)$, 
$$\|Tf\|_{\bmo_{\fai}(\rn)}\le C
(\fai,\,\az) \|f\|_{\bmo_{\fai}(\rn)}.$$
\end{corollary}

\begin{remark}\label{r8.2}
Let $\oz\in A_{\fz}(\phi_{\az})$ for some $\az\in(0,\fz)$ and
$\Phi$ be an Orlicz function satisfying the fact that
$p_{\bfai}=p_{\bfai}^+$ and $p_{\Phi}^+$ is the attainable
critical upper type index of $\Phi$. Then, Theorem \ref{t8.2} and
Corollary \ref{c8.1} cover, respectively,
\cite[Theorem\,8.18]{yys} and \cite[Corollary 8.20]{yys} by taking
$\fai(x,t):=\oz(x)\Phi(t)$ for all $x\in\rn$ and $t\in[0,\fz)$.
\end{remark}

\medskip

\Acknowledgements{This work was supported
by the National Natural Science Foundation (Grant No. 11171027) of
China and Program for Changjiang Scholars and Innovative
Research Team in University of China. The authors would like to
thank the referees for their careful reading and many valuable
remarks which made this article more readable.}


\end{document}